\documentclass[12pt,reqno,tbtags]{amsart}
\usepackage{amssymb}
\usepackage[textsize=small]{todonotes}
\usepackage[square,numbers]{natbib}
\bibpunct[, ]{[}{]}{;}{n}{,}{,}

\title[Component structure: 
  barely supercritical case]
{Component structure of the configuration model:
  barely supercritical case}

\author{Remco van der Hofstad}
\address{Department of Mathematics and
        Computer Science, Eindhoven University of Technology,
        5600 MB Eindhoven, The Netherlands.}
\email{rhofstad@win.tue.nl}
\urladdr{http://www.win.tue.nl/$\sim$rhofstad}

\author{Svante Janson}
\address{Department of Mathematics, Uppsala University,
PO Box 480, SE-751 06 Uppsala, Sweden.}
\email{svante.janson@math.uu.se}
\urladdr{http://www.math.uu.se/$\sim$svante}

\author{Malwina Luczak}
\address{School of Mathematical Sciences, Queen Mary University of London,
  Mile End Road, London E1 4NS, United Kingdom}
\email{m.luczak@qmul.ac.uk}
\urladdr{http://www.maths.qmul.ac.uk/~luczak}


\date{4 November, 2016} 

\theoremstyle{plain}
\newtheorem{theorem}{Theorem}[section]

\newtheorem{lemma}[theorem]{Lemma}

\newtheorem{op}[theorem]{Open problem}

\theoremstyle{definition}
\newtheorem{remark}[theorem]{Remark}

\newtheorem{example}[theorem]{Example}

\numberwithin{equation}{section}
\numberwithin{theorem}{section}

\usepackage{fancyhdr}

\newcommand{\Dn}{D_n}

\newcommand{\e}{{\mathrm e}}

\newcommand {\convd}{\dto}

\def\1{{\mathchoice {1\mskip-4mu\mathrm l}      
{1\mskip-4mu\mathrm l}
{1\mskip-4.5mu\mathrm l} {1\mskip-5mu\mathrm l}}}
\newcommand{\indic}[1]{\1_{\{#1\}}}
\newcommand{\indicwo}[1]{\1_{#1}}

\newcommand{\expec}{{\mathbb{E}}}
\newcommand{\prob}{{\mathbb{P}}}

\newcommand\whp{w.h.p\punkt}
\newcommand\whpx{w.h.p}
\newcommand\punkt{.\spacefactor=1000}    

\newcommand\ie{i.e\punkt}
\newcommand\eg{e.g\punkt}

\newcommand\cf{cf\punkt}
\newcommand{\as}{a.s\punkt}

\newcommand{\sss}   { \scriptscriptstyle }

\newcommand{\gam}{\gamma}

\newcommand {\vep}{\varepsilon}

\def\eqalign#1\enalign{
    \begin{align}#1\end{align}
    }
\newcommand{\lbeq}[1]  {\label{e:#1}}
\newcommand{\refeq}[1] {\eqref{e:#1}}
\newcommand{\eq}{\begin{equation}}
\newcommand{\en}{\end{equation}}
\newcommand{\ben}{\begin{enumerate}}
\newcommand{\een}{\end{enumerate}}
\newcommand{\eqn}[1]{\begin{equation} #1 \end{equation}}
\newcommand{\eqan}[1]{\begin{align} #1 \end{align}}

\newcommand{\nn}{\nonumber}

\renewcommand{\to}{\rightarrow}

\newenvironment{romenumerate}[1][-10pt]{
\addtolength{\leftmargini}{#1}\begin{enumerate}
 \renewcommand{\labelenumi}{\textup{(\roman{enumi})}}%
 \renewcommand{\theenumi}{\textup{(\roman{enumi})}}%
 }{\end{enumerate}}

\newenvironment{Aenumerate}[1][-10pt]{
\addtolength{\leftmargini}{#1}\begin{enumerate}
\renewcommand{\theenumi}{{\upshape{(A\arabic{enumi})}}}
\renewcommand{\labelenumi}{\theenumi}
 }{\end{enumerate}}

\newenvironment{Benumerate}[1][-10pt]{
\addtolength{\leftmargini}{#1}\begin{enumerate}
\renewcommand{\theenumi}{{\upshape{(B\arabic{enumi})}}}
\renewcommand{\labelenumi}{\theenumi}
 }{\end{enumerate}}

\newcommand{\refT}[1]{Theorem~\ref{#1}}

\newcommand{\refL}[1]{Lemma~\ref{#1}}
\newcommand{\refR}[1]{Remark~\ref{#1}}
\newcommand{\refS}[1]{Section~\ref{#1}}
\newcommand{\refSS}[1]{Section~\ref{#1}}

\newcommand{\refE}[1]{Example~\ref{#1}}

\renewcommand\P{\operatorname{\mathbb P{}}}
\newcommand\E{\operatorname{\mathbb E{}}}
\newcommand\Var{\operatorname{Var}}

\newcommand\Po{\operatorname{Po}}
\newcommand\Bin{\operatorname{Bin}}
\newcommand\Be{\operatorname{Be}}

\newcommand\gnd{\ensuremath{G(n,(d_i)_{i\in[n]})}}

\newcounter{CC}
\newcommand{\CC}{\stepcounter{CC}\CCx} 
\newcommand{\CCx}{C_{\arabic{CC}}}     
\newcommand{\CCdef}[1]{\xdef#1{\CCx}}     
\newcounter{cc}

\newcommand\ga{\alpha}
\newcommand\gb{\beta}
\newcommand\gd{\delta}
\newcommand\gD{\Delta}
\newcommand\gf{\varphi}

\newcommand\kk{\kappa}
\newcommand\gl{\mu}

\newcommand\go{\omega}

\newcommand\gs{\sigma}
\newcommand\gss{\sigma^2}
\newcommand\eps{\varepsilon}

\newcommand\cB{\mathcal B}
\newcommand\cC{\mathcal C}

\newcommand\cE{\mathcal E}
\newcommand\cF{\mathcal F}

\newcommand\tA{\widetilde A}
\newcommand\tS{\widetilde S}
\newcommand\tV{\widetilde V}

\newcommand\set[1]{\ensuremath{\{#1\}}}
\newcommand\bigset[1]{\ensuremath{\bigl\{#1\bigr\}}}
\newcommand\Bigset[1]{\ensuremath{\Bigl\{#1\Bigr\}}}

\newcommand\xpar[1]{(#1)}
\newcommand\bigpar[1]{\bigl(#1\bigr)}
\newcommand\Bigpar[1]{\Bigl(#1\Bigr)}
\newcommand\biggpar[1]{\biggl(#1\biggr)}
\newcommand\lrpar[1]{\left(#1\right)}
\newcommand\bigsqpar[1]{\bigl[#1\bigr]}
\newcommand\Bigsqpar[1]{\Bigl[#1\Bigr]}

\newcommand\bigabs[1]{\bigl|#1\bigr|}
\newcommand\Bigabs[1]{\Bigl|#1\Bigr|}

\newcommand\lrabs[1]{\left|#1\right|}

\newcommand{\tend}{\longrightarrow}
\newcommand\dto{\overset{\mathrm{d}}{\tend}}
\newcommand\pto{\overset{\mathrm{p}}{\tend}}
\newcommand\asto{\overset{\mathrm{a.s.}}{\tend}}

\renewcommand\op{o_{\mathrm p}}
\newcommand\Op{O_{\mathrm p}}

\newcommand\bbR{\mathbb R}

\newcommand\ee[1]{{\mathrm e}^{#1}}

\newcommand\eet{{\mathrm e}^{-t}}
\newcommand{\dmax}{\Delta}

\newcommand\tC{{\widetilde \cC}}
\newcommand\pC{{\cC'}}

\newcommand\gan{\ga_n}

\newcommand\eee{\cE_\vep}

\newcommand\qq{^{1/2}}
\newcommand\qqw{^{-1/2}}
\newcommand\qqq{^{1/3}}
\newcommand\qqqb{^{2/3}}
\newcommand\qqqw{^{-1/3}}
\newcommand\qw{^{-1}}
\newcommand\qww{^{-2}}

\newcommand\rhs{right-hand side}

\newcommand{\hg}{\hat g}
\newcommand{\hh}{\hat h}
\newcommand{\gant}{\ga_nt}
\newcommand{\gamn}{\gam_n}
\newcommand{\JL}[1]{\cite[#1]{JanLuc07}}
\newcommand{\sC}[1]{\textup{\textsf{C#1}}}
\newcommand{\taux}{\tau}
\newcommand{\hht}{\hh(\taux)}
\renewcommand{\=}{:=}

\newcommand\pfitemx[1]{\par#1:}
\newcommand\pfitemref[1]{\pfitemx{\ref{#1}}}
\newcommand{\ntoo}{{\ensuremath{n\to\infty}}}

\newcommand\xfrac[2]{#1/#2}

\newcommand\dd{\,\mathrm{d}}

\newcommand\floor[1]{\lfloor#1\rfloor}

\newcommand{\phii}{\Psi}
\newcommand{\xD}{D^*}
\newcommand{\tD}{\widetilde D}
\newcommand{\nk}{_{n,k}}
\newcommand{\Tx}{T'}
\newcommand{\ddn}{(d_i)_{i\in[n]}}
\newcommand{\hddn}{(\hd_i)_{i\in[n]}}
\newcommand{\sumin}{\sum_{i\in[n]}}
\newcommand{\sumko}{\sum_{k=0}^\infty}

\newcommand{\vx}[1]{v(#1)}
\newcommand{\vxk}[1]{v_k(#1)}
\newcommand{\ex}[1]{e(#1)}
\newcommand{\refAA}{\ref{AA}--\ref{AO}}
\newcommand{\refAApos}{\ref{AA}--\ref{AO} and $\eps_n>0$}
\newcommand{\ceb}{\cE(b)}
\newcommand{\cec}{\cE_1(C)}
\newcommand{\cebc}{\cE(b,C)}
\newcommand{\gsLn}{\gs_{L,n}}
\newcommand{\LL}{\bar L}
\newcommand{\KK}{B}
\newcommand{\VV}{V'}
\newcommand{\VVn}{\VV_n}
\newcommand{\ER}{Erd\H{o}s--R\'enyi}
\newcommand{\hd}{\hat d}
\newcommand{\hD}{\hat D}
\newcommand{\hR}{\hat R}

\newcommand{\ea}{\cE_a}
\newcommand{\BT}{\cB_{T'}}

\newcommand{\es}{\cE_{\textsf s}}
\newcommand{\cs}{c_{\textsf s}}

\newcommand{\tnu}{\tilde\nu}
\newcommand{\ettop}{\bigpar{1+\op(1)}}
\newcommand\arxiv[1]{\texttt{arXiv:#1}}
\newcommand\arXiv{\arxiv}

\begin{document}

\begin{abstract}
We study near-critical behavior in the configuration model.
Let $\Dn$ be the degree of a random vertex.
We let $\nu_n=\expec[\Dn(\Dn-1)]/\expec[\Dn]$ and, assuming that
$\nu_n \to 1$ as $n \to \infty$, we write $\vep_n=\nu_n-1$.
We call the setting where
$\vep_n n^{1/3}/(\expec[\Dn^3])^{2/3} \to \infty$
the {\it barely supercritical} regime. We further assume that the variance
of $D_n$ is uniformly bounded as $n \to \infty$.

Let $\Dn^*$ denote the size-biased version of $\Dn$.
We prove that there is
a unique giant component of size $n \rho_n \E D_n (1+o(1))$, where
$\rho_n$ denotes the survival probability of a branching process with offspring distribution
$\Dn^*-1$. This extends earlier results of Janson and Luczak~\cite{JanLuc07}, as well as those of Janson, Luczak, Windridge and House~\cite{SJ300} to the case where the third moment of $D_n$ is unbounded, filling the gap in the literature.

We further study the size of the largest component in the \emph{critical} regime, where $\vep_n = O(n^{-1/3} (\E D_n^3)^{2/3})$, extending and complementing results of Hatami and Molloy~\cite{HatamiMolloy}.
\end{abstract}

\maketitle

\section{Introduction}
\label{sec-intro}

In recent years, the critical and near-critical behaviour of random graphs
has received considerable attention. Here we study random graphs with given
vertex degrees. (See \refS{sec-model} for precise definitions and assumptions.)
In a random graph with given degrees on $n$ vertices, we let $D_n$ denote the degree of
a random vertex; we consider asymptotics as \ntoo.
The fundamental theorem by \citet{MolRee95}
(see also \cite{MolRee98, KangSeierstad,JanLuc07, BolRio15,ReedEtal2016},
and \refS{S2} below)
says that, under suitable
technical assumptions,
there exists \whp{} (meaning `with high probability', \ie, with probability tending to
1 as \ntoo)
a giant component of size $\Theta(n)$ if and only if
$\lim_\ntoo\E D_n(D_n-2)>0$.

The purpose of the present paper is to study \emph{near-critical}
behaviour in greater detail; we assume $\E D_n(D_n-2)\to0$ so
we know that the order $\vx{\cC_1}$
of the largest component is  $\op(n)$, and we
want to find more precise asymptotics of $\vx{\cC_1}$.

\citet{HatamiMolloy} identified the \emph{critical window}; they showed that
(under weak technical conditions)
if $\E D_n(D_n-2)=O\bigpar{n\qqqw(\E D_n^3)\qqqb}$,
then $\vx{\cC_1}$ is of the order $n\qqqb(\E D_n^3)\qqqw$, while
$\vx{\cC_1}$ is larger
if $\E D_n(D_n-2)\gg{n\qqqw(\E D_n^3)\qqqb}$,
and smaller if
$\E D_n(D_n-2)<0$ with
$|\E D_n(D_n-2)|\gg{n\qqqw(\E D_n^3)\qqqb}$. (See also Remark \ref{rem-crit}
for related work identifying
the scaling limits of clusters in the critical window.)
This parallels the well-known critical behaviour of the random graph $G(n,p)$
with $p=(1+\eps_n)/n$, or $G(n,M)$ with $M=(1+\eps_n)n/2$,
where it was shown by \citet{Boll84} and \citet{Luczak90}
that the critical window is characterized by $\eps_n=O(n\qqqw)$;
see also \cite{Bollobas} and \cite{JLR}.

Here we are mainly concerned with the \emph{barely supercritical} regime, where
$\E D_n(D_n-2)\to0$, with 
$\E D_n(D_n-2)>0$ and outside the critical window just defined.
We find (under weak technical conditions)
precise asymptotics of $\vx{\cC_1}$, up to a factor $1+\op(1)$, in this
regime.
In the case when the degree distribution $D_n$ has a bounded $(4+\eps)$-th
moment,
these asymptotics were found by \citet{JanLuc07};
this result was extended to the case when the third power $D_n^3$ is uniformly
integrable by \citet{SJ300}. In this paper, we only assume that
the second moment $\E D_n^2$ exists and is uniformly bounded.
Our study reveals that there is a kind of phase
transition. Roughly speaking,
as long as the asymptotic degree distribution has a finite third moment
(to be precise, as long as $D_n^3$  is uniformly integrable, the case studied in
\cite{JanLuc07} and \cite{SJ300}), the size of the largest component is
proportional to $n\E(D_n(D_n-2))$. However, when the degree distribution has
heavier tails, then
the largest component is smaller; typically (but not always)
of the order $n\E(D_n(D_n-2))/\E D_n^3$.
Precise results are given in
Theorems \ref{T1}--\ref{TD3infty} below, where \refT{TD3} corresponds to the important example when the third moment of the degree distribution converges. Also,
\refE{exam-power-law} discusses power-law degree sequences with possibly unbounded third moment of the degree distribution.
(The same difference between the cases $\E D_n^3=O(1)$ and $\E D_n^3\to\infty$
is also evident in the result on the critical window
by \citet{HatamiMolloy} cited above.)

As said above,
our results (\refT{T1} in particular) show that in the barely supercritical
phase,
the size of the largest component is concentrated within a factor $1+\op(1)$,
\ie, normalized by dividing by a suitable constant, the size converges in
probability to 1. As a complement, we also show (\refT{TC})
that this is \emph{not} true
in the critical window identified by \citet{HatamiMolloy}, and further investigated in
\cite{DhaHofLeeSen16a, DhaHofLeeSen16b, Jose10,Rior12}.
Inside the critical window, the size after normalization will converge in
distribution, at least along subsequences, but the limit will not be
constant; in fact any such limit will be unbounded.
Again, this is precisely as in the well-known case of $G(n,p)$,
see \cite{LPW1994,Aldo97},
so this provides another reason to regard the window defined above as the
critical window, at least on the supercritical side. (We conjecture that the
size of the largest component is concentrated also in the subcritical case,
but, as far as we know, this has not yet been proved.)

It is well known that the process of exploration of the component containing
a given vertex can be approximated by a Galton--Watson branching process;
this gives, for example, a heuristic argument for the condition
$\lim_\ntoo\E D_n(D_n-2)>0$ above. (See further \refR{Rbias}.)
Indeed,
in our main theorem (\refT{T1}), we express the size of the largest
component in terms
of the survival probability of the approximating Galton--Watson process.
In our case,
with $\E D_n(D_n-2)\to0$, we have to consider one Galton--Watson process for
each $n$, so the question of asymtotics of the survival probability
of an asymptotically critical sequence of  branching processes arises.
This was studied by \eg{} \cite{Athreya92} and \cite{Hoppe92};
we give some further general results (needed to prove our results for random graphs)
in \refS{Ssurv}.

Our proofs, however, do not use the branching process approximation directly;
instead, they are based on extending the method of \cite{JanLuc07}, where
the exploration process is considered one vertex at a time, yielding a kind of
random walk with drift (closely related to the branching process), which is then
analysed. \citet{MolRee95,MolRee98} and \citet{HatamiMolloy} use similar
methods, but there are several differences; for example, we use a
continuous-time version of the exploration process, which gives us additional
independence, and we use a different method to obtain bounds for
the random fluctuations.

\section{Model, assumptions and main results}\label{S2}

\subsection{The configuration model}
\label{sec-model}

Given  a positive integer $n$ and a
\emph{degree sequence}, \ie, a sequence of $n$ positive integers
$(d_1,d_2,\ldots, d_n)$, we let $G(n, (d_i)_{i\in[n]} )$ be a simple graph
(i.e., without loops or multiple edges) with the set $[n]=\{1,\ldots,n\}$ of
vertices, chosen
uniformly at random  subject to vertex $i$ having degree $d_i$, for $i\in[n]$. We tacitly assume that there is any such graph at all, so,
for example, $\sum_{i\in[n]} d_i$ must be even.

We follow the standard path of studying
$G(n, (d_i)_{i\in[n]}) $ using the \emph{configuration model}, defined as
follows, see \eg{} \cite{Bollobas}.
Given a degree sequence $\ddn$  with $\sumin d_i$
even,
we start with $d_j$ free half-edges
adjacent to vertex $j$, for $j=1, \ldots, n$. The random multigraph
$G^*(n,(d_i)_1^n)$ is constructed by successively pairing, uniformly at random,
free half-edges into edges, until no free half-edges remain. (In other
words, we create a uniformly random matching of the half-edges.)
Loops and multiple edges may occur in $G^*(n,\ddn)$,
but we can obtain $G(n,\ddn)$ by conditioning $G^*(n,\ddn)$ on being simple (that is, without loops or multiple edges).
Moreover, our condition \ref{A2} below implies that
the probability of obtaining a simple graph
is bounded away from $0$ as $n \to \infty$;
see \cite{Jans06b,simpleII, AngHofHol16}.

We assume that we are given such a degree sequence $(d_i)_{i\in[n]}$ for each $n$
(at least in a subsequence), and we consider asymptotics as \ntoo.
The degrees $d_i=d_i^{\sss(n)}$ may depend on $n$, but for simplicity we do not
show this in the notation.


\subsection{Basic assumptions and notation}
\label{SSassumptions}

All unspecified limits are as \ntoo. We use standard notation for asymptotics.
In particular, $a_n\asymp b_n$, where $a_n$ and $b_n$ are sequences of
positive numbers, means that $a_n/b_n$ is bounded above and below by
positive constants; equivalently, $a_n=O(b_n)$ and $b_n=O(a_n)$.
In contrast, $a_n\sim b_n$ means the stronger $a_n/b_n\to1$.
Furthermore, $a_n\gg b_n$ means $a_n/b_n\to\infty$.

For random variables $X_n$, and positive numbers $a_n$,
$X_n=\op(a_n)$ means $X_n/a_n\pto0$, \ie, $\P(|X_n|>\eps a_n)\to0$ for every
$\eps>0$. Also, $X_n=\Op(a_n)$ means that $X_n/a_n$ is bounded in probability,
\ie, for every $\eps>0$ there exists $C<\infty$ such that $\P(|X_n|>Ca_n)<\eps$
for all $n$ (or, equivalently, for all large $n$).

We let $\Delta_n:=\max_i d_i$ denote the maximum degree in $G(n,\ddn)$ and $G^*(n,\ddn)$.

For $k \in {\mathbb Z}$, we denote by
	\begin{equation}
	n_k:=\#\set{i:d_i=k}
	\end{equation}
the number of vertices of degree
$k$, so that $n=\sum_{k=1}^{\infty} n_k$. Furthermore, let
	\begin{equation}
\label{elln}
	\ell_n:=\sumin d_i=\sum_{k=1}^{\infty} kn_k,
	\end{equation}
the total number of
half-edges; thus the number of edges is $\ell_n/2$.

Let $\Dn$ be the degree of a randomly chosen
vertex in $G(n,\ddn)$ or $G^*(n,\ddn)$; the distribution of $D_n$ is given by
\begin{equation}
    \prob(\Dn=k)=n_k/n.
\end{equation}

Let
\begin{align}
  \label{mun-def}
 \mu_n&:=\E D_n=\sum_{k=1}^{\infty} kn_k/n= \ell_n/n,
\\ \label{nun-def}
\nu_n&:=\frac{\E D_n(D_n-1)}{\E D_n}
=\frac{\sum_{k=1}^{\infty} k(k-1)n_k}{\sum_{k=1}^{\infty} kn_k}
=\frac{\sum_{k=1}^{\infty} k(k-1)n_k}{\ell_n}.
\end{align}
Thus $\mu_n$ is the average degree; $\nu_n$ can be interpreted as the
expected number of new half-edges found when the endpoint of a random
half-edge is explored, see \eqref{EtD} and \refR{Rbias}.


As stated in Section~\ref{sec-intro},
we will study \emph{near-critical}
behaviour; we assume $\nu_n\to 1$
and, for the most part, also that $\nu_n>1$ (and not too small);
this is thus a subcase of the critical case so $v(\cC_1)=\op(n)$.
We define
\begin{equation}\label{eps}
  \eps_n:=\nu_n-1=\frac{\E D_n(D_n-2)}{\E D_n}.
\end{equation}

Our basic assumptions are as follows. (See also the remarks below, and
additional conditions in the theorems.)
\begin{Aenumerate}

\item \label{AA}\label{AD}
$D_n$,
the degree of a randomly chosen vertex,
converges in distribution to a random variable $D$
with a finite and positive mean $\mu:=\E D$.
In other words, there exists a
probability distribution
$(p_k)_{k =  0}^\infty$
such that
\begin{equation}\label{ntopk}
\frac{n_k}{n} \to p_k, \qquad k \ge 0,
\end{equation}
and $\mu=\sumko kp_k\in(0,\infty)$. (Thus $p_k=\P(D=k)$.)

\item \label{A2}
\xdef\QA{\arabic{enumi}}
The second moment $\E D_n^2$ is uniformly bounded:
$\E D_n^2=O(1)$.

\item \label{AD>2}
We have $\P(D\notin\set{0,2})>0$.
Equivalently, $p_0+p_2<1$.

\item \label{Anu=1} \label{AO}
$\nu_n\to1$.  
Equivalently, see \eqref{eps},
\begin{equation}\label{epsn}
\eps_n\to0.
\end{equation}
Assuming \ref{AA}, this is also equivalent to
\begin{equation}\label{EDD-2}
\E D_n(D_n-2)\to0.
\end{equation}
\end{Aenumerate}

\medskip

\begin{remark}\label{RD}
  The assumption \ref{AA} that $D_n$ converges in distribution is mainly for
  convenience. By \ref{A2}, the sequence $D_n$ is always tight, so every
  subsequence has a subsequence that converges in distribution to some $D$;
  moreover $\E D<\infty$ follows from \ref{A2} and $\E D>0$ follows from
  \ref{AD>2},
provided the latter is reformulated as $\liminf_\ntoo\P(D_n\notin\set{0,2})>0$.
It follows, using standard subsequence arguments,
that results such as \refT{T1} that do not use $D$ (explicitly or
implicitly) in the statement hold also without \ref{AA}.
\end{remark}

\begin{remark}  \label{RD2}
\ref{A2} implies uniform integrability of $D_n$ and thus,
together with \ref{AD},
\begin{align}\label{mu}
  \mu_n& \to\mu,
\end{align}
Furthermore, it is easy to see that, assuming \ref{AA}, \ref{A2}
is equivalent to $\nu_n=O(1)$. In particular, \ref{A2} is implied by \ref{AO}; however, we list \ref{A2}
separately for emphasis and for easier comparison with conditions in other
papers.

By Fatou's lemma, \ref{A2} also implies $\E D^2<\infty$.
\end{remark}

\begin{remark}\label{RUI}
Condition \ref{A2} is weaker than the condition
\begin{enumerate}
  \renewcommand{\labelenumi}{\textup{(A\QA${}'$)}}%
\renewcommand{\theenumi}{\labelenumi}%
  \item
\label{A2'} $D_n^2$ are uniformly integrable.
\end{enumerate}
As is well known, \ref{A2'} is, assuming \ref{AA}, equivalent to
$\E D_n^2\to \E D^2<\infty$, and thus also to $\E D^2<\infty$ and
  	\begin{align}
	\nu_n&\to \nu:=\frac{\E D(D-1)}{\E D}. \label{nu}
  	\end{align}
In this case, \ref{AO} is thus equivalent to $\nu=1$,
or, equivalently, $\E D(D-2)=0$, or $\E D^2 = 2\mu$.

On the other hand, if \ref{AA}, \ref{A2} and \ref{AO} are satisfied but
\ref{A2'} is not, then
(by Fatou's lemma) $\E D^2<2\mu$, $\E D(D-2)<0$ and $\nu<1$.

We will \emph{not} need \ref{A2'} in the present paper,
except when explicitly stated; 
it is satisfied in most examples.
\end{remark}

\begin{remark}\label{Rneq2}
\ref{AD>2} rules out the degenerate case when $D\in\set{0,2}$ a.s.; for examples
of exceptional behaviour in this case, see \cite[Remark 2.7]{JanLuc07}.

Since $\E D(D-2)\le0$, see \refR{RUI},
\ref{AD>2} is equivalent to $\P(D=1)>0$.
Furthermore,
if  $D_n^2$ are uniformly integrable, so $\E D(D-2)=0$, see \refR{RUI}, then
\ref{AD>2} is also equivalent to $\P(D>2)>0$.
\end{remark}

\subsection{The size-biased distribution}\label{SSsize-biased}

Let $\Dn^*$ denote the size-biased distribution of $\Dn$, \ie,
	\begin{equation}
    	\prob(\Dn^*=k)=\frac{k}{\expec[\Dn]}\prob(\Dn=k),
	\end{equation}
and let $\tD_n:=\xD_n-1$,
i.e.,
	\begin{equation}
	\label{size-biased-degree}
	\P(\tD_n=k-1)=
 	 \P(\xD_n=k)=\frac{k\P(D_n=k)}{\E D_n}
	=\frac{kn_k}{n\mu_n},
	\qquad k\ge1.
	\end{equation}
For any non-negative function $f$,
	\begin{equation}\label{fxD}
  	\E f(\xD_n)=\frac{\E D_n f(D_n)}{\E D_n};
	\end{equation}
and thus
	\begin{equation}\label{ftD}
  	\E f(\tD_n)=\frac{\E D_n f(D_n-1)}{\E D_n};
	\end{equation}
in particular
	\begin{equation}\label{EtD}
  	\E \tD_n=\E(\xD_n-1)
	=\frac{\E (D_n(D_n-1))}{\E D_n}=\nu_n=1+\eps_n.
	\end{equation}

Similarly,
let $\xD$ have the size-biased distribution of $D$,
and let $\tD:=\xD-1$. Thus $\E \tD=\nu=1$.
Since $D_n\dto D$ by \ref{AA} and $\E D_n\to \E D$ by \eqref{mu}, it follows
that
$\xD_n\dto \xD$ and $\tD_n\dto \tD$.

Note that \ref{AD>2} implies that (and, given \ref{AA}, is equivalent to)
	\begin{equation}
	\lim_\ntoo\P(\tD_n\neq1)=\P(\tD\neq1)=\P(\xD\neq2)>0.
	\end{equation}

Let
$\rho_n$ be the survival probability of a Galton--Watson process with
offspring distribution $\tD_n$, starting from one individual.
By \eqref{EtD} and basic branching process theory, $\rho_n>0\iff \eps_n>0$,
and, in this case $\rho_n$ is the unique solution in $(0,1]$ to
	\begin{equation} \label{rhon}
    	1-\rho_n
	=\E (1-\rho_n)^{\tD_n}
	=\sum_{k=1}^{\infty} \frac{kn_k}{n\mu_n} (1-\rho_n)^{k-1}.
	\end{equation}
We study the asymptotics of $\rho_n$ in \refS{Ssurv}.

\begin{remark}\label{Rbias}
We can interpret $D_n^*$ as the degree of a vertex chosen randomly by
choosing a uniformly random half-edge, and $\tD_n$ as the number of additional
half-edges at that vertex.
Consequently,
the initial stages of the
  exploration of a component of $\gnd$, starting from a random vertex, can be
  approximated by a Galton--Watson process with offspring distribution
  $\tD_n$, except that the first generation has distribution $D_n$.
The survival probability $\rho_n$
  is thus closely connected to the probability that this modified
  Galton--Watson process is infinite, which approximates the probability
  that the chosen vertex lies in a large
component. (In the supercritical case, this is asymptotically
the same as the probability of the chosen vertex lying in the \emph{largest} component.)
To be precise, the modified Galton--Watson  process has survival probability
$\E(1-(1-\rho_n)^{D_n})\sim\mu_n\rho_n$, which agrees with the factor
  $\mu_n\rho_n$ in \refT{T1} below, giving the proportion of vertices in the largest component.
\end{remark}

\subsection{Main results}
\label{SSresults}

Our results in this section
hold for both the random simple graph $G_n:=G(n,\ddn)$ and
the random multigraph $G_n^*:=G^*(n,\ddn)$.
We first prove our theorems for $G_n^*$; they then hold for $G_n$, as is standard,
by conditioning on $G_n^*$ being simple.
To be precise, \ref{A2} implies that $\liminf\P(G_N^*\text{ is simple})>0$,
see \cite{Jans06b,simpleII},
and thus the results below (which all say that certain events have small
probabilities) transfer immediately from
$G_n^*$ to $G_n$, except \refT{TC}\ref{TC>}, which is of a different kind and
requires a special argument (given in \refSS{SSTC>graph}).

In order to state our results,
choose either $G_n$ or $G^*_n$;
let $\cC_1$ denote the largest connected component,
and let $\cC_2$ denote the second largest
component. (For definiteness, we choose the component at random if there is
a tie, and we define $\cC_2:=\varnothing$ if there is only one component.)

For a component $\cC$, we write $v(\cC)$ and $e(\cC)$ to denote
the number of vertices and edges in the component, respectively.
Our main theorem is the following precise and general result concerning the
supercritical case.


\begin{theorem}
\label{T1}
Suppose that \ref{AA}--\ref{AO} are
satisfied,  in particular  $\vep_n=o(1)$.
Suppose also that $\vep_n\gg n^{-1/3} (\expec\Dn^3)^{2/3}$.
Then
\begin{align} 
    \vx{\cC_1}&=\mu_n \rho_n n(1+\op(1)),
\label{t11}
\\
    \vx{\cC_2}&=\op( \rho_n n).
\label{t12}
\end{align}
Furthermore,
$\ex{\cC_1}=(1+\op(1))\vx{\cC_1}=\mu_n \rho_n n(1+\op(1))$ and
$   \ex{\cC_2}=\op( \rho_n n)$.
\end{theorem}

\begin{remark}\label{Vert-k}
Let $\vxk{\cC_1}$ denote the number of vertices of degree $k$ in $\cC_1$.
It can be seen from our proof of Theorem~\ref{T1}
 that $\vxk{\cC_1}=\mu_n \rho_n \prob(D_n^*=k)n(1+\op(1))$.
\end{remark}

In particular, \refT{T1} leads to the following special cases.

Define, recalling \refR{Rneq2}
\begin{equation}\label{kk}
    \kk:=
\E \tD(\tD-1)=
\frac{\expec[D(D-1)(D-2)]}{\expec[D]}\ge0.
\end{equation}
Note that $\kk=\infty$ if and only if $\E D^3=\infty$.
Furthermore, if
$D_n^2$ are uniformly integrable
(\ie, \ref{A2'} holds),
then $\P(D>2)>0$ by \refR{Rneq2}, and thus $\kk>0$.
In this case, we also have $\E [D(D-2)]=0$, see \refR{RUI}, and thus
we also have the
alternative formula
\begin{equation}\label{kk4}
  \kk=\frac{\E D^3-3\E D^2+2\E D}{\E D}
=\frac{\E D^3-3\E [D(D-2)]-4\E D}{\E D}
=\frac{\E D^3}{\mu}-4.
\end{equation}


The next three theorems are easy consequences of Theorem~\ref{T1}, under our assumptions.

\begin{theorem}
\label{TD3}
Suppose that \ref{AA}--\ref{AO} are satisfied,
and that $D_n^3$ is uniformly integrable.
(Thus, $\E D_n^3\to\E D^3<\infty$.)
Suppose further that $\vep_n n^{1/3} \to \infty$.
Then
	\begin{align} 
 	\vx{\cC_1}&=\frac{2\mu}{\kk}\vep_n n(1+\op(1))
	=\frac{2n \E \bigpar{D_n(D_n-2)}}{\kk}(1+\op(1)),\label{tdn3}
	\\    \vx{\cC_2}&=\op( \eps_n n),  \label{td32}
	\end{align}
where $\kk\in(0,\infty)$ is given by \eqref{kk}.
Furthermore,
$\ex{\cC_1}=(1+\op(1))\vx{\cC_1}$ and
$\ex{\cC_2}=\op( \eps_n n)$.
\end{theorem}

\begin{theorem}
\label{TD3infty}
Suppose that \ref{AA}--\ref{AO}
are satisfied,
and that $\E D^3=\infty$. (Thus $\E D_n^3\to\infty$.)
Suppose further that $\vep_n\gg n^{-1/3} (\E D_n^3)^{2/3}$.
Then
	\begin{align} 
   	\vx{\cC_1}&=\op( \eps_n n)  \label{td4}
	.\end{align}
Furthermore,
$\ex{\cC_1}=(1+\op(1))\vx{\cC_1}=\op( \vep_n n)$.
\end{theorem}

The results in Theorems \ref{TD3}--\ref{TD3infty}
are more or less best possible of this type:
in intermediate cases, where $\E D^3<\infty$ but $\limsup \E D_n^3>\E D^3$,
neither \eqref{tdn3} nor \eqref{td4} holds in general, see \refR{Rbadrho}.
To be precise, it follows from Examples \ref{E3a} and \ref{E3c} below that
$\E D_n^3=O(1)$ is neither necessary nor sufficient for \eqref{tdn3}.
Similarly, it follows from Examples \ref{E3a} and \ref{E3b} that
$\E D_n^3\to\infty$ is not sufficient for \eqref{td4} and
$\E D^3=\infty$ is not necessary for \eqref{td4}.
In such intermediate cases, partial answers are given by the following
inequalities. Define, in analogy with \eqref{kk},
\begin{equation}\label{kkn}
    \kk_n:=
\E [\tD_n(\tD_n-1)]=
\frac{\expec[D_n(D_n-1)(D_n-2)]}{\expec[D_n]}.
\end{equation}
Note that, since $\eps_n>0$, by \eqref{eps} we have $\E [D_n(D_n-2)]>0$, which in turn
implies $\kk_n>0$.
Furthermore, by Fatou's lemma and \eqref{EDD-2},
	\eqan{
	\label{kkn>}
 	 \liminf_\ntoo \kk_n
	&=\frac{\liminf_\ntoo\expec[D_n(D_n-2)^2]+\lim_\ntoo \E[D_n(D_n-2)]}{\expec  D}\\
	&\ge \frac{\expec[D(D-2)^2]}{\expec  D}>0.\nn
	}
Thus $\kk_n$ is bounded away from 0, and it follows that
\begin{equation}\label{kknR}
  \kk_n\asymp \E D_n^3.
\end{equation}

\begin{theorem}
\label{TDx}
Suppose that \ref{AA}--\ref{AO} are
satisfied.
Suppose also that $\vep_n\gg n^{-1/3} (\expec\Dn^3)^{2/3}$.
\begin{romenumerate}
\item \label{TDxa}
Then
\begin{align}\label{pyret}
    \vx{\cC_1}&\ge \frac{2\mu_n \eps_n}{\kk_n} n(1+\op(1)).
\end{align}
\item \label{TDxb}
  If\/ $\E D_n^3=O(1)$, then
there exists constant $c,C>0$ such that \whp{}
\begin{equation}
	\label{ele}
c\eps_n n\le \vx{\cC_1}\le C\eps_n n.
  \end{equation}
\item \label{TDxc}
  If\/ $\eps_n\gD_n=o(\E D_n^3)$, then
there exists constants $c,c',C,C'>0$ such that \whp{}
\begin{equation}
	\label{win}
c'\frac{\eps_n n}{\E D_n^3}
\le
c\frac{\eps_n n}{\kk_n}
\le \vx{\cC_1}
\le C\frac{\eps_n n}{\kk_n}
\le C'\frac{\eps_n n}{\E D_n^3}.
  \end{equation}
\end{romenumerate}
\end{theorem}
The lower bounds in \ref{TDxc} are clearly less precise than
the more general
\eqref{pyret},
but are given as companions to the upper bounds.
A weaker and less precise version of the lower bound \eqref{pyret} was given
by \citet[Theorem 1.3]{HatamiMolloy}.

\begin{remark}
We see from Theorems \ref{TD3}--\ref{TDx}
that in the barely supercritical regime,
for a given sequence $\eps_n$, the giant component is smaller
in cases where $\E D^3 = \infty$
than in cases where $\E D_n^3$ is bounded.
(In both cases, the size of the giant component is by \refT{T1} roughly
$n\rho_n$.)
The barely supercritical behaviour of
the largest connected component when $\expec[D_n^3]=O(1)$
is similar to that in the Erd\H{o}s-R\'enyi random graph.
\end{remark}

The condition $\vep_n\gg n^{-1/3} (\E D_n^3)^{2/3}$
in the theorems above is best possible and characterizes supercritical
behaviour in the sense that, if $\vep_n$ is smaller, then, unlike \eqref{t11}, $v(\cC_1)$ is not
concentrated, as is shown by the following theorem for the
critical window. Part \ref{TC<} is proved by \citet[Theorem 1.1]{HatamiMolloy}
under very similar conditions, including a slightly stronger assumption than
\eqref{gDo}.

\begin{theorem}  \label{TC}
Suppose that \refAA{} hold and
$\vep_n=O(n^{-1/3} (\E D_n^3)^{2/3})$.
Suppose further that
	\begin{equation}
  	\label{gDo}
	\gD_n = o\bigpar{(n\E D_n^3)\qqq}.
	\end{equation}
Then the following hold:
\begin{romenumerate}
\item \label{TC<}
$\vx{\cC_1}=\Op\bigpar{n^{2/3} (\E D_n^3)^{-1/3}}$.
In other words, for any $\gd>0$ there exists $K=K(\gd)$ such that
	\begin{equation}\label{tc<}
 	\P\bigpar{\vx{\cC_1}>K{n^{2/3} (\E D_n^3)^{-1/3}}}<\gd.
	\end{equation}
\item \label{TC>}
Moreover, for any $K<\infty$,
	\begin{equation}\label{tc>}
	\liminf_\ntoo  \P\bigpar{\vx{\cC_1}>K{n^{2/3} (\E D_n^3)^{-1/3}}}>0.
	\end{equation}
\end{romenumerate}
Both~\ref{TC<} and~\ref{TC>} hold with $\vx{\cC_1}$ replaced by $e(\cC_1)$.
\end{theorem}


\refT{TC} says that $\vx{\cC_1}/\bigpar{n^{2/3} (\E D_n^3)^{-1/3}}$ is
bounded in probability,
but not \whp{} bounded by any fixed constant. In particular,
$\vx{\cC_1}$ normalized in this way converges in distribution,
at least along suitable subsequences,
but it does not converge to a constant along any subsequence; hence the
limit in distribution (along a subsequence) is really random and not
deterministic. Moreover, \refT{TC}\ref{TC>} shows that any subsequential
limit has unbounded support. (The result by \citet[Theorem 1.1(a)]{HatamiMolloy}
shows that any subsequential limit is strictly positive a.s.)
This is in contrast to the supercritical case in Theorem \ref{T1}.
(This contrast is well known in the classical \ER{} case $G(n,p)$, see \eg{}
\citet{Aldo97}, who describes the limit distribution explicitly.)

\begin{remark}\label{RgDo}
Condition \eqref{gDo} can be written as
  \begin{equation}
	\max_i d_i^3 =o\Bigpar{\sumin d_i^3}.
  \end{equation}
It thus says that no single vertex gives a significant contribution to
$\sumin d_i^3$.
See \cite[Section 1.2]{HatamiMolloy} and \refE{Efel} below for  counterexamples in the
case when
\eqref{gDo} does not hold.
Note also that always $n\E D_n^3=\sumin d_i^3\ge\gD_n^3$, so $\gD_n\le (n\E
D_n^3)\qqq$. Hence, \eqref{gDo} is only a weak restriction.
(Hatami and Molloy~\cite{HatamiMolloy} use a slightly stronger assumption, which, roughly, amounts to assuming $\Delta_n = O( (n \E D_n^3)^{1/3}/\log n)$.)
\end{remark}

\begin{remark}
If $\vep_n\asymp n^{-1/3} (\E D_n^3)^{2/3}$, so we are on the upper boundary of
  the critical window in \refT{TC}, then, using \eqref{gDo},
$\eps_n\gD_n=o(\E D_n^3)$ and thus \refT{Tsurv}\ref{tsurv3} applies to a
  Galton--Watson process with offspring distribution $\tilde{D}_n$ starting
  with one individual (as in
  the proof of \refT{TDx}\ref{TDxc}), and yields
$\rho_n\asymp \eps_n/\E D_n^3\asymp n^{-1/3} (\E D_n^3)^{-1/3}$.
Thus \refT{TC} shows that the
  giant component is of order $\Op(n\rho_n)$ in this case too, although
  $\vx{\cC_1}/(n\rho_n)$ does not converge to a constant.
\end{remark}

\begin{example}[Power-law degrees]
\label{exam-power-law}
Many real-world networks are claimed to have power-law degree sequences, see e.g.,
\cite[Chapter 1]{Hofs16} and the references therein. As a result, various random graph models have been proposed that can yield such graphs,
the configuration model being one of the most popular. Let $\gamma>1$ and assume that, in addition to the assumptions
above, for some constants $C,c>0$,
  	\begin{align}
	\P(D_n> k)&\le C k^{-\gamma}, \qquad k\ge 1, \label{etau1}
	\\
	\P(D_n> k)&\ge c k^{-\gamma}, \qquad 1\le k < \eps_n^{-1/(\gamma-1)}.
	\label{etau2}
  	\end{align}
(The upper limit $\eps_n^{-1/(\gamma-1)}$ in \eqref{etau2} could be reduced
	by any fixed constant factor.  Note that some limit is required, since
	$D_n$ is discrete and \eqref{etau1} implies $\gD_n=O(n^{1/\gam})$.)
Then, in \refT{Tsurv}\ref{tsurv0} and
Example \ref{Ebeta} below, we show that $\rho_n\asymp \vep_n$ when
$\gamma>3$
(so $\E D_n^3=O(1)$), while
	\begin{equation}\label{charlietau}
	\rho_n\asymp \eps_n^{1/(\gamma-2)}
	\end{equation}
when $\gamma\in(2,3)$. \refT{T1} applies and yields that, for $\gamma>3$, and using the form in \refT{TD3},
$\vx{\cC_1}=\frac{2n}{\kappa} \E \bigpar{D_n(D_n-2)}(1+\op(1))$, while, for $\gamma\in(2,3)$,
$\vx{\cC_1}\asymp n \eps_n^{1/(\gamma-2)}.$
\end{example}


\begin{remark}\label{rem-crit}
The critical regime as in \refT{TC} has attracted considerable attention, see e.g., \cite{DhaHofLeeSen16a, DhaHofLeeSen16b, Jose10,Rior12} for results on the sizes of the largest connected components.
\citet{Rior12} investigates the scaling behavior of near-critical clusters under the assumption that all degrees are uniformly bounded. \citet{DhaHofLeeSen16a} perform an analysis under conditions that are close to ours when $\E D_n^3\to \E D^3$, but focus on the scaling limit of critical clusters when $\nu_n=1+\lambda n^{-1/3}+o(n^{-1/3})$ (also for percolation on the configuration model, where the dependence on $\lambda$ is identified as the multiplicative coalescent, see also \citet{Aldo97}).

In the case where $\E D^3=\infty,$ and in the same vein as Example
\ref{exam-power-law}, often stronger assumptions are made and our results in
\refT{TC} in this case are closest in spirit to those in \cite{HatamiMolloy}
in that they only depend on the scaling of $\vep_n$ and
$\expec[D_n^3]$. Order the degrees such that $d_1\geq d_2\geq \cdots \geq
d_n$. \citet{Jose10} assumes that $(d_i)_{i\in[n]}$ are an i.i.d.\ sample
from a distribution whose distribution function satisfies $1-F(x)=c
x^{-\gamma} (1+o(1))$ for $x$ large. In this case, $(d_i
n^{-1/\gamma})_{i\geq 1}$ jointly converge in distribution to
$(c'\Gamma_i^{-1/\gamma})_{i\geq 1},$ where $(\Gamma_i)_{i\geq 1}$ form a
Poisson point process. \citet{DhaHofLeeSen16a} instead take $d_i$ such that
$d_i n^{-1/\gamma}\to c_i$, and, in particular, $\E D_n^3\sim n^{3/\gamma-1}
\sum_{i\geq 1} c_i^3$, where it is assumed that $\sum_{i\geq 1}
c_i^3<\infty$, while $\sum_{i\geq 1} c_i^2=\infty$ (as is the case when
$c_i\asymp i^{-1/\gamma}$ with $\gamma\in(2,3)$). In this case, \refT{TC}
suggests that the largest critical components should scale like
	\eqn{
	n^{2/3} (\E D_n^3)^{-1/3} \asymp n^{2/3} (n^{3/\gamma-1})^{-1/3}=n^{(\gamma-1)/\gamma}.
	}
The results in \cite{Jose10}, as well as the work~\cite{DhaHofLeeSen16b} in preparation, confirm this scaling, and show
that the sizes of the largest connected components,  rescaled by
$n^{-(\gamma-1)/\gamma}$, converge to a limiting sequence, while the
critical window is of order $n^{-(\gamma-2)/\gamma}$. Interestingly, the
description of this limit looks quite different in  \cite{Jose10} compared
to  \cite{DhaHofLeeSen16b},
which is probably due to the fact that \citet{Jose10} also averages out over
the randomness in the degrees.
\end{remark}

\subsection{Complexity of large components}
The structure of components has received substantial attention in the
literature, in particular, the
existence of multicyclic components, \ie, components $\cC$ with
$\ex{\cC}>\vx{\cC}$.
The detailed scaling limit results in
\cite{DhaHofLeeSen16a, DhaHofLeeSen16b, Jose10,Rior12} resolve this question
completely in the critical case.
We  investigate this question in the barely supercritical setting
in \refS{Scomplex} and find asymptotically the complexity of the largest
component $\cC_1$,
see Theorems \ref{TK} and \ref{TK3}--\ref{TK3infty}.
Here, for power-law degrees as in Example \ref{exam-power-law}, the width of
the critical window is tightly related
to the growth of the complexity of the barely supercritical clusters. As can
be expected, the complexity of $\cC_1$
interpolates between tight, as in the critical case, and linear in $n$ as in
the strictly supercritical regime (as shown in \cite{MolRee98}).

\subsection{Discussion}
\label{sec-discussion}
In this section, we discuss our results and pose further questions.

\subsubsection*{CLT for the giant component.}
It would be of interest to extend Theorem \ref{T1} to a statement about
the fluctuations of $\vx{\cC_1}$ around $\mu_n \rho_n n$. In the light of central limit results for the processes
that characterize the component sizes (see, e.g., Lemma \ref{LB1}), it is tempting to
conjecture that a CLT holds for $\vx{\cC_1}$. From our methodology, however, this does not follow easily.
A related question involves proving a CLT for the complexity $k(\cC_1)$
  in the barely supercritical regime.
(Cf.~\cite{PittelWormald2005} for the \ER{} case.)

\subsubsection*{Related random graphs.}
Often, one can deduce results for rank-1 inhomogeneous random graphs
(see \cite{BolJanRio07} for the definition)
from those derived for the configuration model
conditioned on simplicity.
Examples of such graphs are
the Poissonian or {\it Norros--Reittu} random graph \cite{NorRei06}, the
\emph{generalized random graph model} \cite{BriDeiMar-Lof05}, and the
\emph{expected degree}
or \emph{Chung--Lu} random graph
\cite{ChuLu02a,ChuLu02b,ChuLu03,ChuLu06}. In each of these models, edges are
present independently: an edge between $i,j \in [n]$ is present with
probability $p_{ij}$, where $p_{ij}$ is close to $w_iw_j/\ell_n$ for
appropriately chosen vertex weights $(w_i)_{i\in[n]}$, and
$\ell_n=\sum_{i\in[n]}w_i$ denotes the total weight.
When the weight sequence satisfies conditions similar to
\ref{AA}--\ref{Anu=1}, then also the random vertex degrees do, and thus
results carry over rather easily from the configuration model to these
models.

In slightly more detail, by \cite{Jans08a}, in the case where $\expec[D_n^2]\rightarrow \expec[D^2]$,
the above three random graph models are asymptotically equivalent, so that proving a result for one immediately establishes it for any of
the others as well. Furthermore,
when conditioned on the degree sequence, the generalized random graph
is a uniform random graph with that degree sequence \cite{BriDeiMar-Lof05}. We already know
that Theorem \ref{T1} holds for
uniform random graphs whose vertex degrees obey
conditions~\ref{AA}--\ref{Anu=1}, so that, by conditioning on the degree
sequence,
in order to deduce
the same for rank-1 inhomogeneous random graphs,
it suffices to prove that \ref{AA}--\ref{Anu=1} indeed hold
(with convergence in probability)
for the degrees
for the generalized random
graph in the critical case.
This proof is standard, and can, for example, be found in \cite{BhaHofHoo13b} or \cite[Section 7.7]{Hofs16}.
The critical case of these models was studied in \cite{BhaHofLee09a,BhaHofLee09b}.

\section{The branching process survival probability}\label{Ssurv}

Our proofs of Theorems~\ref{T1}
and \ref{TD3}--\ref{TDx} will use
some estimates of the survival
probability of barely supercritical Galton--Watson processes.
In this section, we state and prove these
estimates in a general form, for general Galton--Watson processes. We return
to the setting of the configuration model in the later sections,
where we apply the results stated below with offspring distribution $X_n=\tD_n$.

Relation \eqref{e:tsurv1} below was conjectured and supported by a heuristic
argument by \citet{Ewens69};
\citet{Eshel81} gave counter-examples but also a proof of
\eqref{e:tsurv1} under some conditions. More general sufficient conditions
were given by \citet{Hoppe92} and \citet{Athreya92}; both also gave a
necessary and sufficient condition for \eqref{e:tsurv1} in terms of the
probability generating function of the offspring distribution $X_n$.
(The necessary and
sufficient conditions in \cite{Hoppe92} and \cite{Athreya92} are stated
differently, but they can be seen to be equivalent, using integration by parts.)
Here we give further results, stated in a form more suitable for our purposes,
but note that there are overlaps with earlier ones in the literature.
In particular, \refT{Tsurv}\ref{tsurv1} follows easily from results in both
\cite{Hoppe92} and \cite{Athreya92}.
Furthermore, \eqref{rho>} was given by \cite[Corollary 3.3]{Hoppe92} (in an
equivalent formulation).

\begin{theorem}[Survival probability of a near-critical branching process.]
\label{Tsurv}
Let $X_n$ be a sequence of non-negative integer-valued random variables such
that
$\expec[X_n]=1+\vep_n$, where
$\vep_n > 0$ and $\vep_n \to 0$ as $n \to \infty$.
Suppose also that
$\liminf_n \prob (X_n \not = 1) > 0$.
Let $\rho_n$ be the survival probability of a branching process with offspring
distribution $X_n$, starting with one individual,
i.e.,
the unique solution in $(0,1]$ to
    \eqn{
    \lbeq{rhon}
    1-\rho_n=\expec[(1-\rho_n)^{X_n}].
    }
Then $\rho_n\to0$ and, more precisely,
\begin{equation}
  \label{rhoO}
\rho_n=O(\eps_n).
\end{equation}
Furthermore,
\begin{equation}\label{rho>}
  \rho_n \ge \frac{2\eps_n}{\E X_n(X_n-1)}
\end{equation}
and
\begin{equation}\label{sar}
  \vep_n \asymp \E\bigpar{X_n\land(\rho_nX_n^2)}.
\end{equation}
Moreover:
\begin{romenumerate}
\item \label{tsurv0}
If\/ $\E X_n^2=O(1)$, then
$\rho_n\asymp\eps_n$.

\item \label{tsurv1}
If\/ $X_n\convd X$ for some random variable $X$ 
and\/
$\expec[X_n^2]\rightarrow \expec[X^2] < \infty$, then,
    \eqn{
    \lbeq{tsurv1}
    \rho_n\sim\frac{2\vep_n}{\expec[X(X-1)]}.
    }
\item \label{tsurv2}
If\/ $X_n\convd X$ for some random variable $X$ with
$\expec[X^2]= \infty$, then
    \eqn{
    \lbeq{tsurv2}
    \rho_n=o(\vep_n).
    }
\item \label{tsurv3}
  If\/ $\gD_n$ are numbers such that $X_n\le \gD_n$ \as{} and
  $\eps_n\gD_n=o(\E X_n^2)$,
then
\eqn{
    \lbeq{tsurv3}
    \rho_n\asymp\frac{\eps_n}{\E[X_n(X_n-1)]}
\asymp \frac{\eps_n}{\E X_n^2}.
    }
\end{romenumerate}
\end{theorem}

\begin{proof}
We first show that $\rho_n=o(1)$ as $n \to \infty$.
(For a more general result on continuity of the survival probability as a
functional of the offspring distribution, see \cite[Lemma 4.1]{SJ199}.)
To see this, assume, for a contradiction, that there exists a subsequence
$n_l$ such that
$\rho_{n_l}\rightarrow \rho>0$. Since $\E X_n=O(1)$, the sequence $X_n$ is
tight, so there exists a further subsequence with $X_n\dto X$ along the
subsequence, for some non-negative integer-valued random variable $X$.
Furthermore, by the Skorohod coupling theorem \cite[Theorem 4.30]{Kallenberg},
we may assume that the
variables $X_{n}$ are defined on a probability space where 
the convergence is almost sure.
Then, by dominated convergence, along the subsequence,
$\expec[(1-\rho_{n})^{X_{n}}] \to \expec[(1-\rho)^{X}]$, and so, by
\eqref{e:rhon},
    \eqn{
    1-\rho=\expec[(1-\rho)^X].
    }
In other words, $\rho$ is the survival probability of a branching process
with offspring distribution $X$.
On the other hand, by Fatou's lemma, $\E X\le \lim\E X_n=1$, so this
branching process is critical or subcritical;
furthermore,  $\P(X\neq1)\ge\liminf_n\P(X_n\neq1)>0$ which excludes the
case $X=1$ a.s. Consequently, the survival probability $\rho=0$, a
contradiction.
Hence $\rho_n \to 0$ as $n \to \infty$.

Define $\ga_n:=-\log(1-\rho_n)>0$, and note that $\rho_n\to0$ implies
\begin{equation}\label{garho}
\ga_n\sim\rho_n.
\end{equation}
Also let
\begin{equation}\label{phii}
  \phii(x):=\e^{-x}-1+x;
\end{equation}
note that $\phii(x)\sim x^2/2$ as $x\to 0$.
Then \eqref{e:rhon} can be written
\begin{equation}
\E \e^{-\ga_n X_n}
= \E (1-\rho_n)^{X_n}=1-\rho_n
=\e^{-\ga_n},
\end{equation}
and thus
\begin{equation}\label{gan1}
  \E\phii(\ga_nX_n)
=\E\bigpar{\e^{-\ga_nX_n}-1+\ga_nX_n}
=\e^{-\ga_n}-1+\ga_n(1+\eps_n)
=\phii(\ga_n)+\ga_n\eps_n.
\end{equation}
Hence,
\begin{equation}\label{gan2}
  \E\frac{\phii(\ga_nX_n)}{\ga_n^2}
=\frac{\phii(\ga_n)}{\ga_n^2}+\frac{\eps_n}{\ga_n}
=\frac12+o(1)+\frac{\eps_n}{\ga_n}.
\end{equation}

Suppose now that \eqref{rhoO} fails. Then there exists a subsequence with
$\eps_n/\rho_n\to0$ and thus, by \eqref{garho} and \eqref{gan2},
\begin{equation}\label{gan20}
  \E\frac{\phii(\ga_nX_n)}{\ga_n^2}\to\frac12.
\end{equation}
As above, by considering a
subsubsequence, we may also assume that $X_n\to X$ \as{} for some random variable $X$,
and then a.s., since $\ga_n\to0$,
\begin{equation}\label{1477}
\frac{\phii(\ga_nX_n)}{\ga_n^2}
\to \frac{X^2}2.
\end{equation}
By Fatou's lemma, \eqref{1477} and \eqref{gan20} yield
	\begin{equation}\label{EX2}
  	\tfrac12\E X^2 \le \liminf_\ntoo
 	\E\frac{\phii(\ga_nX_n)}{\ga_n^2}=\frac12.
	\end{equation}
Furthermore,
since the function $\phii(x)/x$ is increasing on $[0,\infty)$,
\eqref{gan20} implies
that,
for any $K>0$,
  	\begin{equation}
	\limsup_\ntoo \E\bigpar{X_n\indic{X_n\ge K}}
	\le \limsup_\ntoo \E{\frac{ K \phii(\ga_n X_n)}{\phii(\ga_n K)}}
	=\lim_\ntoo \frac{ K \ga_n^2/2}{\phii(\ga_n K)}
	=\frac{1}K.
  	\end{equation}
Hence, still along the subsequence,
the random variables $X_n$ are uniformly integrable, and, since
$\E X_n\to1$ and $X_n\to X$ a.s., we have $\E
X=1$. However, this together with \eqref{EX2} yields $\Var(X)=0$, so $X=1$
a.s., which as above is excluded by our assumption
$\liminf_n\P(X_n\neq1)>0$. This contradiction shows that \eqref{rhoO} holds.

Next, for any integer $m\ge0$ and $\rho\in[0,1]$, $(1-\rho)^m\le
1-m\rho+\binom m2 \rho^2$. Hence,
\begin{equation}\label{parix}
  \begin{split}
  1-\rho_n
&=\E(1-\rho_n)^{X_n}
\le \E\Bigpar{1- X_n\rho_n+\frac{X_n(X_n-1)}2\rho_n^2}
\\&
=1-(1+\eps_n)\rho_n+\frac{\E (X_n(X_n-1))}2\rho_n^2	
  \end{split}
\end{equation}
and \eqref{rho>} follows, recalling $\rho_n>0$.

To show \eqref{sar}, note that
\eqref{rhoO} and \eqref{garho} show that $\ga_n=O(\eps_n)$ and thus
$\eps_n/\ga_n$ is bounded below.
Furthermore, $\phii(x)\asymp x^2\land x$ for $x\ge0$,
and thus, by \eqref{gan2},
	\begin{equation}
	\frac{\vep_n}{\ga_n}
	\asymp \frac{\vep_n}{\ga_n}+\frac12
	\asymp  \E\frac{\phii(\ga_nX_n)}{\ga_n^2}
	\asymp \E\bigpar{(\ga_n\qw X_n)\land X_n^2}.
	\end{equation}
Hence, using \eqref{garho} again,
	\begin{equation}
	\vep_n
	\asymp \E\bigpar{X_n\land (\ga_nX_n^2)}
	\asymp \E\bigpar{X_n\land (\rho_nX_n^2)}.
	\end{equation}
\pfitemref{tsurv0}
An immediate consequence of \eqref{rhoO} and \eqref{rho>}.

\pfitemref{tsurv1}
As above, we may assume $X_n\to X$ \as{} (now for the full sequence),
and thus \eqref{1477}.
Since $X_n\dto X$ and $\E X_n^2\to \E X^2<\infty$, the sequence $X_n^2$ is
uniformly integrable. Furthermore,
$0\le\phii(x)\le x^2/2$ for $x\ge0$ and thus
$0\le \xfrac{\phii(\ga_nX_n)}{\ga_n^2}\le X_n^2/2$, so the sequence
$\xfrac{\phii(\ga_nX_n)}{\ga_n^2}$ is also uniformly integrable, which
together with \eqref{1477} implies
	\begin{equation}\label{e1477}
	\E\frac{\phii(\ga_nX_n)}{\ga_n^2}
	\to\tfrac{1}{2} \E X^2.
	\end{equation}
Moreover, the uniform integrability of $X_n^2$ also implies
$\E X = \lim_{\ntoo}\E X_n=1$.
Using \eqref{e1477} in \eqref{gan2}, we thus find
\begin{equation}\label{gani}
	\frac{\eps_n}{\ga_n}
	= \E\frac{\phii(\ga_nX_n)}{\ga_n^2} -\frac12+o(1)
	=\tfrac12 \bigpar{\E X^2-1}+o(1)
	=\tfrac{1}{2}\E(X(X-1))+o(1).
	\end{equation}
As noted above,
$\E(X(X-1))=\Var X>0$
and thus \eqref{gani}
yields, recalling \eqref{garho},
	\begin{equation}\label{ganj}
	\frac{\eps_n}{\rho_n}
	\sim
	\frac{\eps_n}{\ga_n}
	\sim\tfrac{1}{2}\E(X(X-1)).
	\end{equation}
A rearrangement yields \eqref{e:tsurv1}.

\pfitemref{tsurv2}
We may again assume that \eqref{1477} holds a.s.,
which now by Fatou's lemma implies (cf.\ \eqref{EX2})
\begin{equation}\label{e1477oo}
\E\frac{\phii(\ga_nX_n)}{\ga_n^2}
\to \infty.
\end{equation}
 Thus $\vep_n/\ga_n\to\infty$ by \eqref{gan2}. This yields
\eqref{e:tsurv2}, again using
\eqref{garho}. 

\pfitemref{tsurv3}
Note first that \eqref{rhoO} and \eqref{rho>} imply that
$1/\E[X_n(X_n-1)]=O(1)$, \ie, that $\E [X_n(X_n-1)]$ is bounded below. Since
$X_n(X_n-1)\le X_n^2\le 1+2X_n(X_n-1)$, it follows that
$\E [X_n(X_n-1)]\asymp\E X_n^2$,
and thus the final ``$\asymp$'' in \eqref{e:tsurv3} holds.

A lower bound for $\rho_n$ is given by \eqref{rho>}, and it remains only to
show a matching upper bound.
By \eqref{sar}, there exists a constant $C$ such that
$\E(X_n\land (\rho_n X_n^2))<C\eps_n$.
Let $\gb_n:=C\eps_n/\E X_n^2$. Then $\gb_n\gD_n=C\eps_n\gD_n/\E X_n^2=o(1)$
by assumption, so for large $n$ we have $\gb_n\gD_n\le1$ and then
$\gb_nX_n\le1$ \as{} so $X_n\land (\gb_nX_n^2)=\gb_n X_n^2$ \as{} and
\begin{equation}
  \E\bigpar{X_n\land (\gb_n X_n^2)}=\gb_n\E X_n^2=C\eps_n
>\E\bigpar{X_n\land (\rho_n X_n^2)}.
\end{equation}
Hence, $\rho_n<\gb_n$ for large $n$, and thus
$\rho_n=O(\gb_n)=O(\eps_n/\E X_n^2)$.
\end{proof}

\begin{remark}\label{RX=1}
  The assumption $\liminf \P(X_n\neq1)>0$ is essential: if $X_n\dto X=1$,
  almost anything can happen.
For a simple example, let $X_n\in\set{0,1,2}$ with $\P(X_n=0)=q_n$,
$\P(X_n=2)=p_n$ and $\P(X_n=1)=1-p_n-q_n$ where $p_n>q_n>0$ and $p_n\to0$.
Then $\eps_n=p_n-q_n$ and, by \eqref{e:rhon} and  a simple calculation
(we have equality in \eqref{parix} and thus in \eqref{rho>}),
$\rho_n=1-q_n/p_n=\eps_n/p_n$.
Thus \eqref{rhoO} fails. Moreover, $\rho_n=1-q_n/p_n$
may converge to any number in $[0,1]$, or may oscillate.
(See also the examples in \cite{Hoppe92}.)
\end{remark}

\begin{remark}\label{Rbadrho}
If $\E X_n^2\to\infty$ but $X_n\dto X$ with $\E X^2<\infty$,
it is not necessarily the case that \eqref{e:tsurv2} holds, but it is still possible;
see Examples \ref{E3a} and \ref{E3b} below. Similarly, if $\E X_n^2\to C<\infty$ and $X_n\dto X$ but $\E X^2<C$,
then \eqref{e:tsurv1} may or may not hold; see Examples \ref{E3a} and \ref{E3c}.
\end{remark}

We consider several examples illustrating various possible behaviours.
See also the  examples by \citet{Hoppe92}.

\begin{example}[Power laws]\label{Ebeta}
  Let $1<\gb<2$ and assume that for some constants $C,c>0$,
  \begin{align}
	\P(X_n> x)&\le C x^{-\gb}, \qquad x>0, \label{ebeta1}
\\
	\P(X_n> x)&\ge c x^{-\gb}, \qquad 1\le x < \eps_n^{-1/(\gb-1)}.
\label{ebeta2}
  \end{align}
Here, due to the size-biasing in \eqref{size-biased-degree}, $\beta$ is related to $\gamma$ in Example \ref{exam-power-law} by $\beta=\gamma-1$.
Then, by an integration by parts (or an equivalent Fubini argument),
for any $r>0$,
	\begin{equation}
  	\begin{split}
	\E\bigpar{X_n\land (r X_n^2)}
	&=
	\int_0^{1/r} 2rx\P(X_n>x)\dd x + \int_{1/r}^\infty \P(X_n>x)\dd x	
	\\&
	\le 2Cr \int_0^{1/r} x^{1-\gb}\dd x + C\int_{1/r}^\infty x^{-\gb}\dd x
	\\&
	=\Bigpar{\frac{2}{2-\gb}+\frac{1}{\gb-1}}Cr^{\gb-1}.
  	\end{split}
	\end{equation}
Taking $r=\rho_n$, this and \eqref{sar} yield
	\begin{equation}\label{garbo}
	\eps_n=O(\rho_n^{\gb-1}).
	\end{equation}
On the other hand, taking $r=A\eps_n^{1/(\gb-1)}$ for a (large) constant
$A>1$, and assuming that $n$ is so large that $r<1$, by \eqref{ebeta2},
\begin{equation}
  \begin{split}
\E\bigpar{X_n\land (r X_n^2)}
\ge \frac{1}r \P\Bigpar{X_n\ge \frac{1}r}	
\ge c r^{\gb-1}
= c A^{\gb-1}\eps_n.
  \end{split}
\end{equation}
Choosing $A$ sufficiently large, this and \eqref{sar} yield (for large $n$)
	\begin{equation}
  	\E\bigpar{X_n\land (r X_n^2)}
	>
	\E\bigpar{X_n\land (\rho_n X_n^2)}
	\end{equation}
and thus $r>\rho_n$. Consequently, $\rho_n=O(\eps_n^{1/(\gb-1)})$, which
together with \eqref{garbo} yield
	\begin{equation}\label{charlie}
	\rho_n\asymp \eps_n^{1/(\beta-1)}.
	\end{equation}

This example shows that $\rho_n$ may decrease as an arbitrarily large power of
$\eps_n$. (Choose $\beta$ close to 1.)
\end{example}

\begin{example}
For an instance of \refE{Ebeta},
let $1<\beta<2$, and let $X$ be a non-negative integer-valued random
  variable with $\E X=1$ and $\P(X>x)\asymp x^{-\beta}$ as $x\to\infty$.
Fix a sequence $\eps_n\to0$ (with $\eps_n>0$)
and a sequence $M_n$ of integers with
$M_n\ge \eps_n^{-1/(\beta-1)}$. Let $X_n':=X\land M_n$, and define $X_n$ by
\begin{equation}\label{extrunc}
  \P(X_n=k)=
  \begin{cases}
	\P(X'_n=0)-\gd_n, & k=0,\\
	\P(X'_n=1)+\gd_n, & k=1,\\
	\P(X'_n=k), & k\ge2,
  \end{cases}
\end{equation}
where $\gd_n:=\eps_n+\E(X-X_n')$.
Then $\E X_n=\E X_n'+\gd_n=1+\eps_n$ as required.
Note that $\E(X-X_n')\asymp
M_n^{-(\beta-1)}=O(\eps_n)$, so $\gd_n\asymp\eps_n$; in particular $\gd_n\to0$
and the definition \eqref{extrunc} is valid at least for large $n$ (since
$\P(X=0)>0$ by $\E X=1$). Clearly, $X_n\dto X$.

Furthermore, \eqref{ebeta1}--\eqref{ebeta2} hold, and thus \eqref{charlie}
holds.

Moreover, we may choose $M_n$ arbitrarily large, and thus
$\E X_n^2\asymp M_n^{2-\beta}$ can be made arbitrarily large; this shows that
there is no formula similar to \eqref{e:tsurv1} giving $\rho_n$, even
within a constant factor, in terms of $\eps_n$ and $\E X_n^2$ (or $\E X_n(X_n-1)$).

We may also take $M_n=\infty$; then \eqref{charlie} still holds
and $\E X_n^2=\infty$.
\end{example}

\begin{example}\label{E3}
Choose $\eps_n\in(0,1]$ with $\eps_n\to0$ and $p_n\in(0,1/n]$
  with $np_n\to0$ and define (for $n\ge3$) $X_n$ by
  \begin{equation}
	\P(X_n=k) =
	\begin{cases}
\frac{1-\eps_n+(n-2)p_n}{2}, & k=0,	  \\
\frac{1+\eps_n-np_n}{2}, & k=2,	  \\
p_n, & k=n.
	\end{cases}
  \end{equation}
Then $\E X_n=1+\eps_n$ as required, $X_n\dto X$ with
$\P(X=0)=\P(X=2)=\frac12$, and thus $\E X=1$, $\E X^2=2$ and $\E X(X-1)=1$,
and
\begin{equation}\label{paris}
  \E X_n^2 = 2+n^2 p_n+o(1).
\end{equation}
In particular, $\E X_n^2\to \E X^2$ if and only if $n^2p_n\to0$.

Furthermore,
\begin{equation}
  \E \phii(\ga_nX_n) = \frac{1+o(1)}2 \phii(2\ga_n)+p_n\phii(n\ga_n)
=\ga_n^2(1+o(1))+p_n\phii(n\ga_n),
\end{equation}
and thus \eqref{gan1} implies
\begin{equation}\label{hebdo}
  \ga_n\eps_n = \ga_n^2\bigpar{\tfrac12+o(1)}+p_n\phii(n\ga_n).
\end{equation}
\end{example}

We consider several cases of this in the following examples.

\begin{example}
  \label{E3a}
Choose $\eps_n$ and $p_n$ in \refE{E3} such that
$np_n=o(\eps_n)$. Then
$p_n\phii(n\ga_n)=O(p_nn\ga_n)=o(\eps_n\ga_n)$, and thus \eqref{hebdo} yields
$\ga_n\eps_n\sim\tfrac12\ga_n^2$ and thus
\begin{equation}\label{bonn}
\rho_n\sim  \ga_n \sim 2\eps_n,
\end{equation}
just as given by \eqref{e:tsurv1}. This includes cases with $n^2p_n\to0$,
when \refT{Tsurv}\ref{tsurv1} applies by \eqref{paris}, but also cases with
$n^2p_n\to\infty$, when $\E X_n^2\to\infty$ by \eqref{paris}.
(For example, take $\eps_n=n^{-1/4}$ and $p_n=n^{-3/2}$.)

If we instead take $p_n=n\qww$ and $\eps_n=n\qqw$, then $\E X_n^2\to 3>\E
X^2$ by \eqref{paris}, while \eqref{e:tsurv1} nevertheless holds by
\eqref{bonn}.
\end{example}

\begin{example}\label{E3b}
Choose $\eps_n \le n^{-1}$ in \refE{E3},
so $\rho_n=O(n^{-1})$ by \eqref{rhoO}.
Then $\rho_n X_n =O(1)$, so \eqref{sar} yields
\begin{equation}\label{em}
  \eps_n \asymp \E(\rho_n X_n^2) =\rho_n \E X_n^2.
\end{equation}
If we further choose $p_n$ with $np_n\to0$ and $n^2p_n\to\infty$, then $\E
X_n^2\to\infty$ by \eqref{paris}, and thus
$\rho_n=o(\eps_n)$ by \eqref{em}.
(For example, take $\eps_n=n^{-1}$ and $p_n=n^{-3/2}$.)
\end{example}

\begin{example}
  \label{E3c}
Choose $\eps_n=n\qw$ and $p_n=An\qww$ in \refE{E3}, for some constant $A>0$.
Thus $\E X_n^2\to 2+A$ by \eqref{paris}, and
$\E[ X_n(X_n-1)]\to 1+A$.
In particular, $\E X_n^2=O(1)$ and thus
\refT{Tsurv}\ref{tsurv0} yields $\ga_n\sim\rho_n\asymp n^{-1}$.
More precisely, \eqref{hebdo} yields, after multiplication by $n^2$,
\begin{equation}\label{eleonora}
  n\ga_n\sim \tfrac12(n\ga_n)^2+A\phii(n\ga_n).
\end{equation}
As just said, $n\ga_n=\ga_n/\eps_n$ is bounded above and below, and
\eqref{eleonora} shows that, if $n\ga_n\to a$ along some subsequence, then
$a=\frac12a^2+A\phii(a)$, or
\begin{equation}\label{winston}
  \frac{a-\frac12a^2}{\phii(a)}=A.
\end{equation}
Hence $0<a<2$. Furthermore, it is easy to see (by differentiating) that
$\phii(a)/(a-\frac12a^2)$ is strictly increasing on $(0,2)$. Hence
\eqref{winston} has a unique solution $a=a(A)\in(0,2)$ for any $A>0$,
and thus $n\ga_n\to a(A)$. Consequently, also $\rho_n/\eps_n=n\rho_n\to
a(A)$, given by \eqref{winston}.

It is easily verified that $2>a(A)>2/(1+A)$. Hence, \eqref{e:tsurv1} does
not hold, and neither does \eqref{e:tsurv1} with $\E[X(X-1)]$ replaced by
$\E[X_n(X_n-1)]$.
\end{example}

\section{Further preliminaries}\label{StD}

\subsection{More on $\rho_n$ and $\ga_n$ in the barely supercritical case}
\label{SSrho}
Suppose that \ref{AA}--\ref{AO} are satisfied,  and furthermore $\vep_n>0$.
(Note that the assumptions of Theorems~\ref{T1} and \ref{TD3}--\ref{TDx}
imply that $\eps_n>0$,
except possibly for some small $n$ that we may ignore.)

In what follows, $\rho_n$ will denote the survival probability of a
Galton--Watson process
with offspring distribution $\tD_n$, see \refS{SSsize-biased} and \eqref{rhon}.
As in \refS{Ssurv}, it will often be convenient to work with
\begin{equation}
  \label{gan}
\ga_n:=-\log(1-\rho_n).
\end{equation}

\begin{lemma}
  \label{Lrho}
If \ref{AA}--\ref{AO} are satisfied  and $\vep_n>0$, then
$\rho_n\to0$,
\begin{equation}\label{garho2}
\ga_n\sim\rho_n
\end{equation}
and
	\begin{equation}\label{sard}
  	\vep_n
	\asymp \E\bigpar{D_n^2\land(\rho_nD_n^3)}	
	\asymp \E\bigpar{D_n^2\land(\ga_nD_n^3)}	.
	\end{equation}
\end{lemma}

\begin{proof}
\refT{Tsurv} applies to $X_n=\tD_n$, with $X=\tD$ and with
$\eps_n$ as in \refS{sec-model} by \eqref{EtD}.
In particular, $\rho_n=O(\eps_n)\to0$ and thus,
by \eqref{gan}, $\ga_n\sim\rho_n$.
Furtermore,
by \eqref{sar},
	\begin{equation}\label{sardrho1}
  	\vep_n \asymp \E\bigpar{\tD_n\land(\rho_n\tD_n^2)}
	\le \E\bigpar{\xD_n\land(\rho_n(\xD_n)^2)}.
	\end{equation}
Moreover, if $\xD_n>1$ then $\xD_n\le 2\tD_n$. Thus, using \eqref{rhoO}
and \eqref{sar},
	\begin{equation}\label{sardrho2}
  	\begin{split}
	\E\bigpar{\xD_n\land(\rho_n(\xD_n)^2)}
	\le \rho_n+4 \E\bigpar{\tD_n\land(\rho_n\tD_n^2)}
	=O(\eps_n).
  	\end{split}
	\end{equation}
Combining \eqref{sardrho1}--\eqref{sardrho2} and using \eqref{fxD}, we find
	\begin{equation}\label{sardrho}
 	 \begin{split}
  	\vep_n &
	\asymp \E\bigpar{\xD_n\land(\rho_n(\xD_n)^2)}
	=\frac{1}{\E D_n} \E\bigpar{D_n(D_n\land(\rho_nD_n^2))}
	\asymp \E\bigpar{D_n^2\land(\rho_nD_n^3)},
  	\end{split}
	\end{equation}
proving the first part of \eqref{sard}; the second follows from
$\ga_n\sim\rho_n$.
\end{proof}

Note also that
\eqref{rhon} implies, by \eqref{fxD},
\begin{equation}\label{frigg}
  (1-\rho_n)^2 =  \E (1-\rho_n)^{\xD_n}
=\frac{\E \bigpar{D_n(1-\rho_n)^{D_n}}}{\E D_n}
=\frac{\E \bigpar{D_n(1-\rho_n)^{D_n}}}{\mu_n}
,
\end{equation}
which 
can be rewritten as
	\begin{equation}\label{frej}
  	\E\bigpar{D_n\e^{-\ga_nD_n}}=\mu_n\e^{-2\ga_n}.
	\end{equation}

In the case $\E D_n^3\to \E D^3<\infty$,
\ie{}, when $D^3_n$ are uniformly integrable,
we have  $\E \tD_n^2\to\E\tD^2$ by \eqref{ftD};
hence \eqref{e:tsurv1} applies and yields,
using \eqref{ftD} again and the notation \eqref{kk},
where now $\kk>0$ by \eqref{e:tsurv1} or \refR{Rneq2},
	\begin{equation}\label{ganD3}
  	\ga_n\sim \rho_n \sim \frac{2\eps_n}{\E (\tD(\tD-1))}
  	= \frac{2\eps_n \mu}{\E (D(D-1)(D-2))}
  	= \frac{2\eps_n}{\kk}.
	\end{equation}

\subsection{The  Skorohod coupling theorem}\label{SSSkorohod}
We assume in \ref{AA} that $D_n\dto D$.
By the Skorohod coupling theorem \cite[Theorem 4.30]{Kallenberg},
we may without loss of generality assume the stronger
$D_n\asto D$; this will be convenient (although not really necessary)
in some proofs.
(We have already used
the Skorohod coupling theorem in a similar way for $X_n$ in \refS{Ssurv},
and will use it for a third set of variables in the proof of \refL{L0}.)

\subsection{A semimartingale inequality}
Our proofs below will use a semimartingale inequality to
control the deviations of various random processes.

We say that a stochastic process $X(t)$, defined on an interval $[0,T]$, is
a \emph{semimartingale with drift $\xi(t)$} (with respect to a filtration
$(\cF_t)$) if $X(t)$ is adapted and 
	\begin{equation}\label{drift}
  	X(t)=M(t)+\int_0^t \xi(u)\dd u,
	\end{equation}
for some martingale $M(t)$.
It is proved in \cite[Lemma 2.2]{Jans94} that, if $X(t)$ is a bounded
semimartingale with drift $\xi(t)$, then
	\begin{equation}\label{memoirs}
  	\begin{split}
  	\E \sup_{s\le t\le u} |X(t)|^2
	&\le 13 \E |X(u)|^2 + 13 \Bigpar{\int_s^u \sqrt{\E \xi(t)^2}\dd t}^2
	\\
	&\le 13 \E |X(u)|^2 + 13 (u-s)\int_s^u \E\bigsqpar{\xi(t)^2}\dd t.
  	\end{split}
	\end{equation}
We will be using the following modification of \eqref{memoirs}.
\begin{lemma}\label{LX}
Let $X(t)$ be a semimartingale with drift $\xi(t)$, defined on $[0,u]$.
Then
\begin{equation}\label{lx}
  \E \sup_{0\le t\le u} |X(t)|^2
\le 13 \sum_{j=0}^\infty \E |X(2^{-j}u)|^2
+ 13 \int_0^u t\E\bigsqpar{\xi(t)^2}\dd t.
\end{equation}
\end{lemma}
\begin{proof}
Let $u_j:=2^{-j}u$.
We have
\begin{equation}\label{lx1}
  \sup_{0\le t\le u} |X(t)|^2
\le \sum_{j=0}^\infty \sup_{u_{j+1}\le t\le u_j} |X(t)|^2,
\end{equation}
since $X(t)$ is a.s.\ right-continuous at 0 (and everywhere) by \eqref{drift}.
We take the expectation, and note that by \eqref{memoirs},
\begin{equation}\label{lx2}
  \begin{split}
  \E \sup_{u_{j+1}\le t\le u_j} |X(t)|^2
&\le
13 \E |X(u_j)|^2
+ 13 (u_j-u_{j+1})\int_{u_{j+1}}^{u_j} \E\bigsqpar{\xi(t)^2}\dd t
\\&
\le
13 \E |X(u_j)|^2
+ 13 \int_{u_{j+1}}^{u_j} t \E\bigsqpar{\xi(t)^2}\dd t.	
  \end{split}
\end{equation}
The result follows by \eqref{lx1} and \eqref{lx2}.
\end{proof}

Inequality~\eqref{lx} will yield better estimates than inequality~\eqref{memoirs} in cases when process $(X(t))$ takes relatively small values near time $0$ (so that $\sum_{j=0}^\infty \E |X(2^{-j}u)|^2$ is finite and not too large) but has quite significant drift (so that $\int_0^u t\E\bigsqpar{\xi(t)^2}\dd t$ is significantly smaller than $u\int_0^u \E\bigsqpar{\xi(t)^2}\dd t$).

\section{The supercritical case}\label{Ssuper}
As explained in \refSS{SSresults}, it suffices to prove Theorems
\ref{T1} and \ref{TD3}--\ref{TDx} for the multigraph $G_n^*$, since the simple graph case
follows
by conditioning on simplicity. We thus consider the random multigraph $G_n^*:=G^*(n,(d_i)_1^n)$
constructed by the configuration model.

\subsection{A more general theorem}\label{SSgeneral}
We follow the structure of proof in \cite{JanLuc07}.
We explore the clusters of the multigraph given by the configuration model
one by one, using the cluster exploration
strategy introduced in \cite[Section 4]{JanLuc07}. We regard each edge as
consisting
of two half-edges, each half-edge having one endpoint.
We label the vertices as
\emph{sleeping} or \emph{awake}, and the half-edges as \emph{sleeping},
\emph{active} or \emph{dead}. The sleeping and active half-edges are called
\emph{living} half-edges. (During the exploration, the endpoint of a
sleeping half-edge is sleeping, while the endpoint of an active or dead
half-edge is awake.)

We start with all vertices and half-edges sleeping. We pick a vertex,
make it awake and label its half-edges as active. We then take
any active half-edge, say $x$, and find its partner half-edge $y$ in the
graph; we label these
two half-edges as dead; further, if the endpoint of $y$ is sleeping, we
label it awake and
all the other half-edges at this vertex active. We repeat the above steps as
long as there is an
active half-edge available. When there is no active half-edge left, then we
have obtained the first component. We then pick another vertex, and reveal
its component, and so on, until all the components have been found.

We apply this procedure to $G^*_n$, revealing its edges during
the process. This means that, initially, we only observe the vertex degrees
and the half-edges,
but not how they are joined into edges. Hence, each time we need a partner
of an edge, it is uniformly distributed over all living half-edges, and the dead
half-edges correspond to the half-edges that have already been paired.
We choose our pairings by giving the half-edges
i.i.d.\ random maximal lifetimes with distribution $\mbox{Exp}(1)$.
In other words, each half-edge dies spontaneously at rate 1 (unless
killed earlier). Each time we need to find the partner of a half-edge $x$, we then wait until
the next living half-edge $\neq x$ dies, and take that one.
This gives the following algorithm for simultaneously constructing and exploring the components
of $G^*(n,(d_i)_1^n)$:

\begin{itemize}
\item[\sC1]
Select a
sleeping vertex and declare it awake and all of its half-edges active.
To be precise,
we choose the vertex by choosing a half-edge uniformly at random among all
sleeping
half-edges. The process stops when there is no sleeping half-edge left; the remaining
sleeping vertices are all isolated and we have explored all other components.

\item[\sC2] Pick an active half-edge (which one does not matter) and kill it, i.e., change its status to dead.

\item[\sC3]
Wait until the next half-edge dies (spontaneously). This half-edge is
paired to the one killed in the previous step \sC2 to form an edge of the
graph. If
the vertex it belongs to is sleeping, then we declare this vertex awake and
all of its other half-edges active.
Repeat from \sC2 if there is any active half-edge; otherwise from \sC1.
\end{itemize}

The components are created between the successive times \sC1 is performed:
the vertices in the component
created between two successive
such times are the vertices awakened during the corresponding interval.

We let $S_n(t)$ and $A_n(t)$ be the numbers of sleeping and active
half-edges, respectively,
at time $t\geq 0$, and let $L_n(t)=S_n(t)+A_n(t)$ denote the number of living half-edges.
Further, we let $V_{n,k}(t)$ denote the number of sleeping vertices of
degree $k$ at time $t$,
and let $V_n(t)$ be the number of sleeping vertices at time $t$;
thus
\begin{align}\label{VS}
 V_n(t)=\sum_{k=0}^{\infty} V_{n,k}(t),
 &&&
 S_n(t)=\sum_{k=0}^{\infty} kV_{n,k}(t).
\end{align}

Let $T_1<T_2$ be random times when $\sC1$ are performed.
Then the exploration starts on new components at times $T_1$ and $T_2$,
and the components
found between  $T_1$ and $T_2$ in total have
$V_n(T_1-)-V_n(T_2-)$ vertices and
$S_n(T_1-)-S_n(T_2-)$ half-edges, and hence $[S_n(T_1-)-S_n(T_2-)]/2$ edges.
Note also that $A_n(t-)=0$ when $\sC1$ is performed, and $A_n(t)\ge0$ for every
$t$.

We also introduce variants of $(S_n(t), A_n(t), V_n(t))_{t\geq 0}$ obtained
by ignoring the effect of \sC1.
Let $\widetilde V_{n,k}(t)$ denote the number of vertices of degree $k$ such that
all of their $k$ half-edges have their exponential maximal life times
greater than $t$. Then
$\widetilde V_{n,k}(t)$ has a $\Bin(n_k, \mathrm{e}^{-kt})$ distribution,
and the $\{\widetilde V_{n,k}(t)\}_{k=1}^{\infty}$ are independent random
variables
for any fixed $t$.
Let
\begin{align}\label{m22}
 \tV_n(t):=\sum_{k=0}^{\infty} \tV_{n,k}(t),
 &&&
 \tS_n(t):=\sum_{k=0}^{\infty} k\tV_{n,k}(t),
\end{align}
and
\begin{align}\label{m33}
 \tA_n(t):=L_n(t)-\tS_n(t)
=A_n(t)-(\tS_n(t)-S_n(t)).
\end{align}
It is obvious that $\tS_n(t)\ge S_n(t)$;
moreover,  $\tS_n(t)- S_n(t)$ increases only when \sC1 is performed,
and it is not difficult to show
that,
see \JL{Lemma 5.3 and (5.7)},  
\begin{align}\label{jlss}
  0\le \tS_n(t)-S_n(t)
=A_n(t)-\tA_n(t)
< -\inf_{s\le t} \tA_n(s) + \Delta_n,
\end{align}
where, as before, $\Delta_n := \max_{1 \le i \le n} d_i$ is the maximum vertex degree.

In order to explain the argument used to prove \refT{T1} more clearly,
and to
explain the connections to the previous versions of this argument used in
\cite{JanLuc07}, we give the argument in a general form  (that includes the
two versions in \cite{JanLuc07}), using certain parameters and functions,
$\tau,\ga_n,\gam_n,\hg(t),\hh(t), \psi_n(t)$.
We assume  that these satisfy conditions
\ref{Xfirst}--\ref{Xlast} below, and then prove a general result, \refT{TX}.
The choices of these parameters and functions used in the proofs of
\cite[Theorems 2.3 and 2.4]{JanLuc07} are described in Remarks \ref{R07T2.3}
and \ref{R07T2.4}. In order to prove \refT{T1}, we instead make
the choices in \eqref{svatau}--\eqref{svapsi} below.
(The reader who only wants a proof of \refT{T1} can thus assume these
choices throughout.)
We verify in \refSS{SSspecial} that the choices in
\eqref{svatau}--\eqref{svapsi}
actually satisfy the assumptions \ref{Xfirst}--\ref{Xlast}.

Assumptions \ref{Xfirst}--\ref{Xlast} are as follows.
\begin{Benumerate}
  \item \label{Xtau} \label{Xfirst}
$\tau>0$ is fixed.
  \item \label{Xag}
$(\ga_n)$ and $(\gam_n)$ are sequences of positive numbers such that
	$\gamn=O(\gan)$.
\item \label{Xgh}
$\hg, \hh\colon [0,\infty)\to\bbR$ are continuous functions;
$\hg$ is strictly positive on $(0,\infty)$ and
$\hh$ is strictly increasing on $(0,\infty)$.

\item \label{Xpsi}
\xdef\QXpsi{\arabic{enumi}}
$(\psi_n)$ is a sequence of continuous functions on $[0,2\tau]$ such that:
  \begin{enumerate}
\renewcommand{\labelenumii}{(\alph{enumii})}%
\renewcommand{\theenumii}{\labelenumii}%
\item \label{Xpsi0}
  $\psi_n(0)=0$;
\item \label{Xpsitau}
  $\psi_n(\tau)=o(1)$;
\item \label{Xpsitau'}
  for some $\tau'>0$, $\psi_n(t)\ge 0$ on $[0,\tau']$;
\item \label{Xpsiab}
 for any compact interval $[a,b]\subset(0,\tau)$,
 $\liminf_\ntoo\inf_{a\le t\le b}\psi_n(t)>0$;
\item \label{Xpsi-}
  for every $t>\tau$, $\limsup_\ntoo\psi_n(t)<0$;
\item \label{Xpsiequi}
 $(\psi_n)$ is equicontinuous at  $\tau$, i.e.,
 if $t_n\to\tau$, then $\psi_n(t_n)\to0$.
 \end{enumerate}

\item \label{XA}
  \begin{equation*}
\sup_{t\le2\taux}\lrabs{\frac1{n\gam_n}\tA_n(\gant)-\psi_n(t)}\pto0.	
  \end{equation*}

\item \label{XV}
$$
\sup_{t\le2\taux}\lrabs{\frac1{n\ga_n}
\bigpar{\tV_n(0)-\tV_n(\gant)}-\hg(t)}
\pto0;
$$
\item \label{XS}
$$\sup_{t\le2\taux}\lrabs{\frac1{n\ga_n}
 \bigpar{\tS_n(0)-\tS_n(\gant)}-\hh(t)}\pto0;$$

\item \label{Xdmax}\label{Xlast}
$$
\frac{\dmax_n}{n\gam_n}\to0.
$$
\end{Benumerate}
Note that \ref{XV} and \ref{XS} imply that necessarily $\hg(0)=\hh(0)=0$.

\begin{remark}
\label{Rpsin1}
In the case when $\psi_n=\psi$ does not depend on $n$,
\ref{Xpsi} says simply that $\psi$ is continuous with
$\psi(0)=\psi(\tau)=0$,
$\psi>0$ on $(0,\tau)$ and $\psi<0$ on $(\tau,2\tau)$.
In general, \ref{Xpsi} should be interpreted as an asymptotic version of this.
In particular, for any $\eps>0$ with $\eps<\tau$, we have
$\psi_n(\tau-\eps)>0$ and $\psi_n(\tau+\eps)<0$ for all large $n$;
it follows that, at least for large $n$, $\psi_n$ has a zero $t_n>0$ such
that $t_n\to\tau$. Furthermore, every zero of $\psi_n$ is $o(1)$,
$\tau+o(1)$ or $\ge\tau$.
\end{remark}

\begin{remark}\label{Rpsin2}
If,
at least for all large $n$,
$\psi_n$ is concave on $[0,2\tau]$
(which is the case in our main application),
then  \ref{Xpsi} can be replaced by the simpler
\begin{Benumerate}
  \renewcommand{\labelenumi}{\textup{(B\QXpsi${}'$)}}%
\renewcommand{\theenumi}{\labelenumi}%
  \item
\label{Xpsi'}
$\psi_n$ is continuous and concave on $[0,2\tau]$ and such that $\psi_n(0)=0$,
	$\psi_n(\tau)=o(1)$,
$\psi_n(2\tau)=O(1)$ and
$\liminf_\ntoo \psi_n(\tau/2)>0$;
  \end{Benumerate}
in fact, \ref{Xpsi'} is easily seen to imply \ref{Xpsi} (with, e.g.,
$\tau'=\tau/2$), at least for large $n$, which suffices.
\end{remark}

We now state a general theorem concerning the largest and second largest component sizes under assumptions \ref{Xfirst}--\ref{Xlast}.
Recall that, for a component $\cC$, we write $v(\cC)$ and $e(\cC)$ to denote
the number of vertices and edges in the component, respectively.
(In \refL{LC+} we extend this notation to the case where $\cC$
is a union of several components.)

\begin{theorem}
\label{TX}
Under assumptions \ref{Xfirst}--\ref{Xlast},
\begin{align}
  v(\cC_1)
&=n\gan\hg(\taux)+\op(n\gan),
\label{lcv}
\\
e(\cC_1)
&=n\gan\hh(\taux)/2+\op(n\gan).
\label{lce}
\end{align}
Furthermore, $v(\cC_{2}),e(\cC_{2})=\op(n\gan)$.
\end{theorem}

The proof  of Theorem \ref{TX}
follows \cite[Sections 5 and 6]{JanLuc07} with minor modifications,
omitting some details (and repeating others).
Before giving the details, we offer some intuition behind its statement.
Suppose that we are able to show (as we will later) that $(S_n(t), A_n(t),
V_n(t))_{t\geq 0}$
are close to $(\widetilde S_n(t), \widetilde A_n(t), \widetilde V_n(t))_{t\geq 0}$.
By \refR{Rpsin1}
and \ref{XA},
there is a large component whose exploration commences within time $\op(\alpha_n \taux)$ and ends
at time $\alpha_n \taux(1+\op(1))$; this turns out to be the largest
component. Moreover,
by \ref{XV}, the number of vertices in this component
is $n\ga_n\hg(\tau)(1+\op(1))$; and, by \ref{XS}, the number of half-edges is
close to $n\ga_n\hh(\tau)(1+\op(1))$.

\begin{remark}\label{R07T2.3}
We note that \JL{Theorem 2.3} is one
example of \refT{TX}, with
$\taux=-\ln\xi$,
$\gan=\gamn=1$,
$\psi_n(t)=\psi(t)=H(\eet)$,
$\hg(t)=1-g(\eet)$, $\hh(t)=h(1)-h(\eet)=\gl(1-{\mathrm e}^{-2t})+\psi(t)$;
in this case, \ref{XA}, \ref{XV}, \ref{XS}
are \JL{(5.6), (5.2), (5.3)}.
(Here, $\psi_n(t)$ is not always concave.)
\end{remark}

\begin{remark}\label{R07T2.4}
Similarly, \JL{Theorem 2.4} is another instance of \refT{TX},
now with $\gan=\E D_n(D_n-2)\to0$ as in \cite{JanLuc07},
$\gamn=\gan^2$, $\psi_n(t)=\psi(t)=t-\tfrac\gb2t^2$,
$\taux=2/\gb$, $\hg(t)=\gl t$,
$\hh(t)=2\gl t$; for \ref{XA}, \ref{XV},
\ref{XS}, see \JL{(6.7), Lemma 6.3 and the Taylor
expansions in the proof of Lemma 6.4}.
Note that \ref{Xdmax} holds, since $n^{2/3} \gamma_n \to \infty$ and
$\dmax_n = o(n^{1/3})$, because $D_n^3$ is uniformly integrable.
\end{remark}

We will see later
that also Theorem  \ref{T1} follows from \refT{TX}.

The proof of Theorem \ref{TX} will use the following lemmas; the second
generalizes \JL{Lemmas 5.6 and 6.4}.

\begin{lemma}
  \label{L0}
Assume \ref{Xfirst}--\ref{Xlast} and let $T_n$ be random times such that
$T_n\pto\tau$.
Then
\begin{equation}\label{aT}
\begin{split}
 \sup_{0\le t\le T_n}\frac1{n\gamn}\Bigabs{\tS_n(\gant)-S_n(\gant)}
=
 \sup_{0\le t\le T_n}\frac1{n\gamn}\Bigabs{\tA_n(\gant)-A_n(\gant)}
\pto0.
\end{split}
\end{equation}
\end{lemma}

\begin{proof}
We may replace $T_n$ by $T_n\land(2\tau)$, since \whp{}
$T_n\land(2\tau)=T_n$; hence we may assume that $T_n\le2\tau$.
Furthermore, using the Skorohod coupling theorem, we may assume that
$T_n\asto\tau$.
We note next that this implies
\begin{equation}\label{paddington}
  \inf_{0\le t\le T_n} \psi_n(t) \to0
\end{equation}
a.s., and thus in probability.
In fact, if \eqref{paddington} fails at some point in the probability
space,
and $T_n\to\tau$,
then there exists $t_n$,
at least for some subsequence of $n$,
with $0\le t_n\le T_n=\tau+o(1)$ and $\psi_n(t_n)<-\eps$, for some $\eps<0$.
(Recall that $\psi_n(0)=0$, so the infimum is never positive.)
We may select a further subsequence with
$t_n\to t'\in[0,\tau]$;
this contradicts \ref{Xpsi}.
(Consider the cases $t'=0$, $t'=\tau$ and $0<t'<\tau$ separately, and use
\ref{Xpsitau'}, \ref{Xpsiequi}, \ref{Xpsiab}.)

 By \eqref{paddington} and
\ref{XA},
\begin{equation}
 \inf_{0\le t\le T_n}\frac1{n\gamn}\tA_n(\gant)
\pto0,
\end{equation}
and thus by \eqref{jlss} and \ref{Xdmax}
\begin{equation}\label{aTx}
\begin{split}
 \sup_{0\le t\le T_n}\frac1{n\gamn}\Bigabs{\tS_n(\gant)-S_n(\gant)}
&=
 \sup_{0\le t\le T_n}\frac1{n\gamn}\Bigabs{A_n(\gant)-\tA_n(\gant)}
\\&
\le-
 \inf_{0\le t\le T_n}\frac1{n\gamn}\tA_n(\gant)+\frac{\dmax_n}{n\gamn}
\pto0.
\end{split}
\end{equation}
\end{proof}

In what follows we consider several random times. They generally depend on $n$
but we simplify the notation and denote them by $T_1, \Tx_1,\dots$ as an
abbreviation of $T_{1n},\dots$

\begin{lemma}
  \label{LC+}
Let\/ $\Tx_1$ and $\Tx_2$ be two (random) times when \sC1 are performed,
with $\Tx_1\le \Tx_2$, and assume that $\Tx_1/\gan\pto t_1$ and
$\Tx_2/\gan\pto t_2$ where
$0\le t_1\le t_2\le \taux$.
If\/ $\tC$ is the union of all components explored between $\Tx_1$ and
$\Tx_2$, then, under assumptions \ref{Xfirst}--\ref{Xlast},
\begin{align*}
  v(\tC)
&=n\gan\bigpar{\hg(t_2)-\hg(t_1)}+\op(n\gan),
\\
  e(\tC)
&=\frac12n\gan\bigpar{\hh(t_2)-\hh(t_1)}+\op(n\gan).
\end{align*}
In particular, if $t_1=t_2$, then $v(\tC)=\op(n\gan)$ and $e(\tC)=\op(n\gan)$.
\end{lemma}

\begin{proof}
Taking, for $j=1,2$, $T_n=\Tx_j/\ga_n+\taux-t_j$ in \eqref{aT}, we see
that
\begin{equation}
  \label{e11s}
\sup_{0\le t\le \Tx_j}\bigabs{\tS_n(t)-S_n(t)}=\op\bigpar{n\gamn}.
\end{equation}
Since further
$0\le \tV_n(t)-V_n(t)\le\tS_n(t)-S_n(t)$,
see \eqref{VS},
we have also
\begin{equation}
  \label{e11v}
\sup_{0\le t\le \Tx_j}\bigabs{\tV_n(t)-V_n(t)}=\op\bigpar{n\gamn}.
\end{equation}

Since $\tC$ consists of the vertices awakened in the interval $[\Tx_1,\Tx_2)$,
by \eqref{e11v}, \ref{XV} and \ref{Xgh},
\begin{align*}
v(\tC)
&=V_n(\Tx_1-)-V_n(\Tx_2-)
=\tV_n(\Tx_1-)-\tV_n(\Tx_2-)+\op(n\gamn)	
\\&
=n\gan\bigpar{\hg(\Tx_2/\gan)-\hg(\Tx_1/\gan)+\op(1)}
\\&
=n\gan\bigpar{\hg(t_2)-\hg(t_1)+\op(1)}.
\end{align*}
Similarly,
using \eqref{e11s} and \ref{XS},
\begin{align*}
2e(\tC)
&=S_n(\Tx_1-)-S_n(\Tx_2-)
=\tS_n(\Tx_1-)-\tS_n(\Tx_2-)+\op(n\gamn)	
\\&
=n\gan\bigpar{\hh(\Tx_2/\gan)-\hh(\Tx_1/\gan)+\op(1)}
\\&
=n\gan\bigpar{\hh(t_2)-\hh(t_1)+\op(1)}.
\qedhere
\end{align*}
\end{proof}

\begin{proof}[Proof of Theorem \ref{TX}]
Note that \eqref{aT} (with $T_n=\tau$) and \ref{XA} show that
\begin{equation}\label{magn}
\sup_{t\le\taux}\lrabs{\frac1{n\gam_n}A_n(\gant)-\psi_n(t)}\pto0.
\end{equation}
Hence, using \ref{Xpsiab},
for every
$\eps>0$, \whp{} $A_n(t)>0$ on
$[\gan\eps,\gan(\taux-\eps)]$,
so no new components are started during that interval. On the other hand,
if $0<\vep<\tau$, then
by \eqref{m33}, \ref{XA} and \eqref{aT},
\begin{equation*}
  \begin{split}
\frac1{n\gamn}\bigsqpar{
\bigpar{\tS_n(&\gan(\taux+\vep))-S_n(\gan(\taux+\vep))}
-
\bigpar{\tS_n(\gan\taux)-S_n(\gan\taux)}}
\\&
=\frac1{n\gamn}\bigsqpar{
\bigpar{A_n(\gan(\taux+\vep))-\tA_n(\gan(\taux+\vep))}
-
\bigpar{A_n(\gan\taux)-\tA_n(\gan\taux)}}
\\&
\ge-\frac1{n\gamn}
\tA_n(\gan(\taux+\vep))
-\frac1{n\gamn}
\bigpar{A_n(\gan\taux)-\tA_n(\gan\taux)}
\\&
=-\psi_n(\taux+\vep)+\op(1).
  \end{split}
\end{equation*}
This is \whp{} positive,
since
$\limsup_{\ntoo}\psi_n(\taux+\vep)<0$ by \ref{Xpsi-},
and then \sC1 is
performed at least once between $\alpha_n \tau$ and $\alpha_n (\tau+ \vep)$.

Consequently, if $T_1$ is the last time \sC1 is performed
before $\gan\taux/2$ and $T_2$ is the next time, then
\whp{} $0\le T_1\le\gan\vep$ and
$\gan(\taux-\vep)\le
T_2\le\gan(\taux+\vep)$. Since $\vep$ can be
chosen arbitrarily small, this shows that $T_1/\gan\pto0$
and $T_2/\gan\pto\taux$.

Let $\cC'$ be the component explored between $T_1$ and
$T_2$. By \refL{LC+} (with $t_1=0$ and $t_2=\tau$),
$\cC'$ has
\begin{equation}\label{ccv}
v(\cC')=n\gan(\hg(\taux)+\op(1))
\end{equation}
vertices and
\begin{equation}\label{cce}
e(\cC')=
  \tfrac12 n\gan(\hh(\taux)+\op(1))
\end{equation}
edges.

It remains to prove that all other components  have only $\op(n\gan)$ edges
(and thus vertices) each. (This implies $\cC_1=\cC'$.)
We argue
as in \JL{pp.\ 213--214 (end of Section 6)}. We fix a small $\vep>0$
and say that a component is \emph{large} if it has at least
$\vep n\ga_n$ edges, and thus at least $2\vep n\ga_n$ half-edges.
If $\vep$ is small enough,
then \whp{} $\cC'$ is large by \eqref{cce}, and further
$(\hht-\vep)n\gan<2e(\pC)<(\hht+\vep)n\gan$.
Let $\eee$ be the event that
$2e(\pC)<(\hht+\vep)n\gan$ and that the total number of half-edges in large
components is at least
$(\hht+2\vep)n\gan$.

It follows by \refL{LC+} applied to $T_0=0$ and $T_1$ that the total
number of vertices and half-edges in components found before $\pC$ is
$\op(n\gan)$.
Thus there exists a sequence $\gan'$ of constants such that
$\gan'=o(\gan)$ and \whp{} at most $n\gan'$ vertices are awakened and at most $n\gan'$ half-edges are made active before
$T_1$, when the first large component is found.

Let us now condition on the final
graph obtained through our component-finding algorithm.
It follows from our specification of $\cC_1$ that,
given
$G^*(n,(d_i)_1^n)$,
the components appear in our process in size-biased
order (with respect to the number of edges), obtained by
picking half-edges uniformly at random (with replacement, for
simplicity) and taking the corresponding components, ignoring every
component that already has been taken. We have seen that \whp{} this
finds components containing at most $n\gan'$ vertices and half-edges before a
half-edge in a large component is picked. Therefore, starting again
at $T_2$, \whp{} we find at most $n\gan'$ half-edges in new components
before a half-edge is chosen in some large component; this half-edge
may belong to $\pC$, but if $\eee$ holds, then with probability at
least $\vep_1\=1-(\hht+\vep)/(\hht+2\vep)>0$ it does not, and
therefore it belongs to a new large component. Consequently, with
probability at least $\vep_1\P(\eee)+o(1)$, the algorithm
finds a second large component at a time $T_3$, and
less than $n\gan'$ vertices and half-edges between $T_2$ and $T_3$. In this
case,
let $T_4$ be the time this second large component is completed. If
no such second large component is found, let for definiteness
$T_3=T_4=T_2$.

The number of half-edges found between $T_2$ and $T_3$ is,
using $\tS_n(t)\ge S_n(t)$, \eqref{aT} with $T_n=T_2/\gan$,
\ref{Xag}
and \ref{XS} together with the fact that $T_2/\gan\le2\tau$ \whp,
\begin{equation*}
  \begin{split}
S_n(T_2-)-S_n(T_3-)	
&\ge
\tS_n(T_2-)-(\tS_n(T_2-)-S_n(T_2-))-\tS_n(T_3-)	
\\&
=
\tS_n(T_2-)-\tS_n(T_3-)	+\op(n\gamn)
\\&
\ge
\tS_n(T_2-)-\tS_n((2\gan\taux)\land T_3-)	+\op(n\gamn)
\\&
=n\gan\bigpar{\hh((2\taux)\land (T_3/\gan))-\hh(T_2/\gan)}+\op(n\gan).
  \end{split}
\end{equation*}
Since, by the definitions above,  this is at most $n\gan'=o(n\gan)$,
it follows that
$\hh((2\taux)\land (T_3/\gan))-\hh(T_2/\gan)\le \op(1)$.
Furthermore, $T_2\le T_3$ and $T_2/\gan\pto\taux$, and thus \whp{}
$T_2/\gan\le 2\tau$. Hence, using \ref{Xgh},
it follows that $(2\taux)\land (T_3/\gan)-\taux=\op(1)$, and thus
$T_3/\gan\pto\taux$.
Consequently, \eqref{aT} applies to $T_n=T_3/\gan$, and, since no \sC1 is
performed between $T_3$ and $T_4$, using also \ref{Xdmax} again,
\begin{equation}
  \label{f1}
\sup_{t\le T_4}\bigabs{\tS_n(t)-S_n(t)}
\le
\sup_{t\le T_3}\bigabs{\tS_n(t)-S_n(t)}+\dmax_n
=\op(n\gamn).
\end{equation}

Let $t_0\in(\taux,2\taux)$; then by \ref{Xpsi-},
for some $\gd>0$,
$\psi_n(t_0)<-2\gd$ for all large $n$, and thus \ref{XA}
shows that
\whp{} $\tA_n(\gan t_0) \le -n\gamn\gd$ and thus
\begin{equation*}
  \tS_n(\gan t_0)-S_n(\gan t_0)
=A_n(\gan t_0)-\tA_n(\gan t_0) \ge
-\tA_n(\gan t_0) \ge
n\gamn\gd.
\end{equation*}
Hence \eqref{f1} shows that \whp{} $T_4<\gan t_0$. Since $t_0-\taux$
can be chosen arbitrarily small,
and further $T_2\le T_3\le T_4$ and $T_2/\gan\pto\taux$, it
follows that $T_4/\gan\pto\taux$.

Finally, by \refL{LC+} again, this time applied to $T_3$ and $T_4$, the
number of edges found between $T_3$ and $T_4$ is
$\op(n\gan)$. Hence,
\whp{} there is no large component found there, although the
construction gave a large component with probability at least
$\vep_1\P(\eee)+o(1)$.
Consequently,
$\vep_1\P(\eee)=o(1)$ and thus
$\P(\eee)=o(1)$.

Recalling the definition of $\eee$, we see that \whp{} the total
number of half-edges in large components is less than $(\hht+2\vep)n\gan$;
since \whp{} at least $(\hht-\vep)n\gan$ of these belong to $\pC$,
see \eqref{cce},
there are at most $3\vep n\gan$ half-edges, and therefore
at most $\frac32\vep n\gan+1$ vertices, in any other component.

Choosing $\vep$ small enough, this shows that \whp{} $\cC_1=\pC$, and
further
$v(\cC_2)\le e(\cC_2)+1\le \frac32\vep n\gan+1$. This completes the proof
of Theorem \ref{TX}.
\end{proof}

\subsection{Proof of Theorems \ref{T1}--\ref{TDx}}\label{SSspecial}

Now suppose that we are given a sequence of degree distributions $D_n$
satisfying the conditions \refAApos.
We choose the parameters in \ref{Xfirst}--\ref{Xlast}
as follows, where $\rho_n$ as before is the survival probability of a
Galton--Watson process with offspring distribution $\tD_n$, see
\eqref{rhon}.
(Note that \eqref{svaga} is the same as \eqref{gan}.)
\begin{align}
\tau&:=1 \label{svatau}
\\
\gan&:=-\log(1-\rho_n) \label{svaga}
\\
    \hg(t)&:=\mu t, \qquad \hh(t):=2\mu t,
\label{svagh}
\\
  \gamma_n&:=\E (D_n(1\land\ga_n D_n)^2),\label{svagam}
\\
\psi_n(t)&:=
\gamn\qw\bigpar{ \mu_n\e^{-2\ga_n t} - \E\bigpar{D_n\e^{-\ga_n t D_n}} }
.\label{svapsi}
\end{align}
Recall that, by \refL{Lrho}, $\rho_n\to0$ and $\ga_n\to0$.

We next show that under the conditions of \refT{T1},
these parameters
satisfy \ref{Xfirst}--\ref{Xlast} (possibly except for some small $n$ that
we may ignore). This will take a series of lemmas.

\begin{lemma}\label{L123}
Assume \refAApos.
Then
the parameters defined in \eqref{svatau}--\eqref{svapsi}
satisfy \ref{Xtau}, \ref{Xag}, \ref{Xgh} and \ref{Xpsi'},
  and thus also \ref{Xpsi}, at least for $n$ large.
Furthermore,
\begin{equation}\label{ga2gam}
  \ga_n^2=O(\gam_n).
\end{equation}
\end{lemma}

\begin{proof}
\pfitemref{Xtau} Trivial.

\pfitemref{Xag}
Since $\eps_n>0$, we have $\rho_n>0$ and thus $\ga_n>0$ and
$\gam_n>0$.
Furthermore, by
\eqref{svagam},
\begin{equation}\label{gamga}
  \gam_n \le \E (D_n(\ga_n D_n)) = \ga_n \E D_n^2 = O(\ga_n).
\end{equation}

\pfitemref{Xgh}
Trivial by \eqref{svagh}.

\pfitemref{Xpsi'}
By the definition \eqref{svapsi} and \eqref{frej},
$\psi_n(0)=0$ and $\psi_n(\tau)=\psi_n(1)=0$.

$\psi_n(t)$ is trivially continuous. (Recall that each $D_n$ is a discrete
random variable taking only a finite number of different values.)

We next show that $\psi_n$
is concave on
$[0,2\tau]=[0,2]$ for large $n$.
(It is not always concave on $(0,\infty)$, nor does it have to be concave on
$[0,2]$ for small $n$,
as can be seen by simple counterexamples with $D_n\in\set{1,3}$.)
By \eqref{svapsi},
	\begin{equation}\label{fj}
 	 \gamn\gan\qww\psi_n''(t)
	=
	4\mu_n\e^{-2\ga_n t} - \E\bigpar{D_n^3\e^{-\ga_n t D_n}} .
	\end{equation}
For every $t\in[0,2]$ we have
\begin{equation}\label{sala}
  \begin{split}
\E\bigpar{D_n^3\e^{-\ga_n t D_n}} &
= \E\bigpar{D_n(D_n-2)^2\e^{-\ga_n t D_n}}
+ 4 \E\bigpar{D_n^2\e^{-\ga_n t D_n}}
- 4 \E\bigpar{D_n\e^{-\ga_n t D_n}}
\\&
\ge \E\bigpar{D_n(D_n-2)^2\e^{-2\ga_n  D_n}}
+ 4 \E\bigpar{D_n^2\e^{-2\ga_n D_n}}
- 4 \E D_n.
  \end{split}
\end{equation}
For the first term on the \rhs{} of \eqref{sala} we may assume, by the Skorohod coupling (see Section~\ref{SSSkorohod}), that $D_n\asto D$
and thus $D_n(D_n-2)^2\e^{-2\ga_n D_n} \asto D(D-2)^2$;
thus Fatou's lemma yields
\begin{equation}\label{salb}
\liminf_\ntoo \E\bigpar{D_n(D_n-2)^2\e^{-2\ga_n D_n}}
\ge \E\bigpar{D(D-2)^2}.
\end{equation}
Next, using \eqref{sard},
\begin{equation}\label{salc}
\E\bigpar{D_n^2\bigpar{1-\e^{-2\ga_n D_n}}}
\le
2\E\bigpar{D_n^2(1\land \ga_n D_n)}
=O(\eps_n)=o(1)
\end{equation}
and thus, using also \eqref{EDD-2},
\begin{equation}\label{jularbo}
\E\bigpar{D_n^2\e^{-2\ga_n D_n}}
=
\E\bigpar{D_n^2}+O(\eps_n)
=
\E\bigpar{D_n(D_n-2)}+2\E D_n+O(\eps_n)
=
2\mu +o(1).
\end{equation}
Combining \eqref{sala}--\eqref{jularbo}
and $\E D_n=\mu_n\to \mu$,
we obtain
\begin{equation}\label{lyman}
  \liminf_\ntoo\inf_{t\in[0,2]}\E\bigpar{D_n^3\e^{-\ga_n t D_n}}
\ge \E \bigpar{D(D-2)^2}+8\mu-4\mu
\end{equation}
and thus by \eqref{fj} and \ref{AD>2},
\begin{equation}
  \limsup_\ntoo\sup_{t\in[0,2]}
\gam_n\gan\qww\psi_n''(t)
\le -\E \bigpar{D(D-2)^2}
<0.
\end{equation}
Consequently, for $n$ large, $\psi''_n(t)<0$ on $[0,2]$, and thus
$\psi_n$ is concave in this interval.

Next we verify \eqref{ga2gam}. In fact, if $D_n\neq0$, then $1\land
\ga_nD_n\ge\ga_n$ and thus the definition \eqref{svagam} implies
\begin{equation}\label{kaj}
  \gamn\ge\E(D_n\gan^2)=\gan^2\mu_n.
\end{equation}
Thus $\gan^2/\gam_n\le 1/\mu_n=O(1)$, since $\mu_n\to\mu>0$.

We now complete the proof of \ref{Xpsi'}. We can write the definition \eqref{svapsi} as
\begin{equation}
  \gam_n\psi_n(t)
=
\E\bigpar{D_n(1-\e^{-\ga_n t D_n})}
-
 \mu_n\bigpar{1-\e^{-2\ga_n t}}.
\end{equation}
Since $\psi_n(1)=0$, we thus have
\begin{equation}
  \begin{split}
\gam_n\psi_n(2)&=\gam_n\psi_n(2)-2\gam_n\psi_n(1)
\\&
=
-\E\bigpar{D_n(1-2\e^{-\ga_n D_n}+\e^{-2\ga_n  D_n})} 	
+  \mu_n\bigpar{1-2\e^{-2\ga_n }+\e^{-4\ga_n }}
\\&
=
-\E\bigpar{D_n(1-\e^{-\ga_n D_n})^2} 	
+  \mu_n\bigpar{1-\e^{-2\ga_n }}^2.
  \end{split}
\end{equation}
Consequently, using \eqref{kaj},
\begin{equation}
  \gam_n\psi_n(2)
\le  \mu_n\bigpar{1-\e^{-2\ga_n }}^2
\le  4\mu_n\ga_n ^2
\le4\gam_n,
\end{equation}
and, by \eqref{svagam},
\begin{equation}\label{psi2}
-  \gam_n\psi_n(2)
\le
\E\bigpar{D_n(1-\e^{-\ga_n D_n})^2} 	
\le
\E\bigpar{D_n(1\land \ga_n D_n)^2} 	
=\gam_n.
\end{equation}
Consequently,
$-1\le\psi_n(2)\le4$ and thus
$|\psi_n(2)|\le4$.

Similarly,
\begin{equation}\label{mau}
  \begin{split}
2\gam_n\psi_n(\tfrac12)&=2\gam_n\psi_n(\tfrac12)-\gam_n\psi_n(1)
\\&
=\E\bigpar{D_n(1-2\e^{-\ga_n D_n/2}+\e^{-\ga_n  D_n})} 	
-  \mu_n\bigpar{1-2\e^{-\ga_n }+\e^{-2\ga_n }}
\\&
=
\E\bigpar{D_n(1-\e^{-\ga_n D_n/2})^2} 	
-  \mu_n\bigpar{1-\e^{-\ga_n }}^2.
  \end{split}
\end{equation}
Denote the two terms
on the right-hand side of \eqref{mau} by $A_1$ and $A_2$.
Since $1-\ee{-x}\asymp 1\land x$,
	\begin{equation}\label{mmmau}
  	A_1\asymp \E\bigpar{D_n(1\land(\gan D_n))^2}=\gam_n.
	\end{equation}
In order to show that $\liminf_\ntoo\psi_n(\frac12)>0$, it thus remains only
to show that $A_1$ is not cancelled by $A_2$.
First, $\ga_n\to0$ 
and thus
$A_2\sim\mu_n\gan^2\sim\mu\gan^2$.
Furthermore, since $1-\ee{-x}\ge x\ee{-x}$ for $x\ge0$,
\begin{equation}
  A_1\ge \frac{\ga_n^2}{4}\E\bigpar{D_n^3\ee{-\ga_nD_n}}.
\end{equation}
Thus, using \eqref{lyman} and \ref{AD>2},
\begin{equation}
  \begin{split}
\liminf_\ntoo\frac{A_1}{A_2}
&\ge\liminf_\ntoo\frac{\E\bigpar{D_n^3\e^{-\ga_n D_n}}}{4\mu}
\ge\frac{\E\bigpar{D(D-2)^2}+4\mu}{4\mu}
>1.
 \end{split}
\raisetag{\baselineskip}
\end{equation}
Since $A_1,A_2\ge0$, it follows that $A_1-A_2\asymp A_1$, and thus
\eqref{mau}--\eqref{mmmau} yield
\begin{equation}
  2\gam_n\psi_n(\tfrac12) \asymp A_1\asymp\gam_n,
\end{equation}
which verifies $\liminf_\ntoo \psi_n(\tfrac12)>0$.
This completes the proof of \ref{Xpsi'}.
\end{proof}

\begin{remark}
  \label{RB}
Note, for later use, that we have shown that, for large $n$ at least,
$\psi_n$ is concave on $[0,2]$ with $\psi_n(0)=\psi_n(1)=0$ and, by
\eqref{psi2}, $\psi_n(2)\ge-1$;
hence $0\ge\psi_n'(1)\ge-1$, and thus $0\le\psi_n(t)\le 1-t$ for $t\in[0,1]$ and
$1-t\le \psi_n(t)\le 0$ for $t\in[0,2]$, so $|\psi_n(t)|\le1$ for $t\in[0,2]$.
\end{remark}

We next show that \ref{XA}, \ref{XV}, \ref{XS} hold if
we replace the random processes
$\tA_n$, $\tV_n$ and $\tS_n$ by their expectations, at least under the extra
assumption that $n\gam_n\to\infty$.

\begin{lemma}
[Asymptotics of means of $\tS_n(t), \tA_n(t), \tV_n(t)$]
\label{L-E-SVA}
Assume \ref{AA}--\ref{Anu=1}, $\eps_n > 0$, and additionally that $n\gam_n\to\infty$.
Then, with parameter values as in \eqref{svatau}--\eqref{svapsi},
for any fixed $t_0$,
    \eqan{
    \sup_{t\le t_0}\lrabs{\frac{1}{n\gan}
\bigl(\E[\tS_n(0)]-\E[\tS_n(\gant)]\bigr)-\hh(t)}&=o(1),
    \lbeq{mean-S}\\
    \sup_{t\le t_0}\lrabs{\frac{1}{n\gan}\bigl(\E\tV_n(0)-\E\tV_n(\gant)\bigr)-\hg(t)}&=o(1),
    \lbeq{mean-V}\\
    \sup_{t\le t_0}\lrabs{\frac{1}{n\gamn}\expec[\tA_n(\gant)]-\psi_n(t)}&=o(1).
    \lbeq{mean-A}
    }
\end{lemma}

\begin{proof}
%
We have, using
	\begin{equation}\label{EDn2}
  	\E D_n^2 = \frac{1}n\sum_k k^2 n_k
	=\E D_n(D_n-1)+\E D_n
	=\mu_n\nu_n+\mu_n=\mu_n(2+\eps_n),
	\end{equation}
and the definition \eqref{phii},
	\begin{equation}\label{vile}
  	\begin{split}
 	\frac{1}{n\ga_n} \bigpar{\E[\tS_n(0)]-\expec[\widetilde S_n(\alpha_n t)]}
	&=
	\frac{1}{n\ga_n} \sum_{k=1}^{\infty} k\Bigl(\expec[\widetilde V\nk(0)]
    	-\expec[\widetilde V\nk(\alpha_n t)]\Bigr)
	\\&
    	=\frac{1}{n\ga_n} \sum_{k=1}^{\infty} k n_k (1-\mathrm{e}^{-\alpha_n
	  tk})
	=\frac{1}{\ga_n} \E\bigpar{D_n(1-\e^{-\ga_ntD_n})}
	\\
    	&=t \E D_n^2+
	\frac{1}{\ga_n} \E\bigpar{D_n(1-\e^{-\ga_ntD_n}-\ga_n t D_n)}
	\\
    	&=t \mu_n(2+\eps_n)
	-\frac{1}{\ga_n} \E\bigpar{D_n \phii(\ga_ntD_n)}.
  	\end{split}
\raisetag{\baselineskip}
	\end{equation}
We now estimate the last term, noting that
\begin{equation}\label{phiixx}
0 \le\phii(x)\le x \land x^2.
\end{equation}
Thus, for all $t\in[0,t_0]$,
\begin{equation}\label{ve}
0\le
  \frac{1}{\ga_n} \E\bigpar{D_n \phii(\ga_ntD_n)}
\le \E\bigpar{D_n(t_0D_n\land (\ga_nt_0^2D_n^2))}
\le (t_0+t_0^2) \E\bigpar{D_n^2\land (\ga_n D_n^3)}.
\end{equation}
By \eqref{sard}, this is $O(\eps_n)=o(1)$, and \eqref{e:mean-S} follows
from \eqref{vile} by the definition \eqref{svagh} of $\hh(t)$.

The proof of \refeq{mean-V} is similar, and easier, as there is one fewer power
of $k$ involved.

To prove \refeq{mean-A},
note first that $L_n(t)$ is a death process where individuals die at rate 1,
except that when someone dies, another is immediately killed (by  \sC2), so
the number of living individuals drops by 2, except when the last is killed;
moreover $L_n(0)=\ell_n-1$,
where we recall from~\eqref{elln} that $\ell_n=n\mu_n$ is the total number of half-edges.
We can couple $L_n(t)$ with a similar process $\LL_n(t)$ starting at
$\LL_n(0)=\ell_n$ so that both processes jump whenever the smaller jumps,
and then
	\begin{equation}
  	\label{serrander}
	|L_n(t)-\LL_n(t)|\le1
	\end{equation}
for all $t$,
cf.\ \cite[Proof of Lemma 6.1]{JanLuc07}.
Then $\frac12\LL_n(t)$ is a standard death process with intensity 2,
starting at $\ell_n/2$, and
thus $\E\LL_n(t)=\ell_n \ee{-2t}$.
Hence,
	\begin{equation}\label{EL}
	\lrabs{\E L_n(t)-\ell_n \e^{-2t}}
	= \lrabs{\E L_n(t)-\E\LL_n(t)}
	\le 1
	\end{equation}
for all $t\ge0$.
Consequently, uniformly in all $t\ge0$,
\begin{equation}
    \expec[\tA_n(\alpha_n t)]
    =\expec[L_n(\alpha_n t)]-\E[\tS_n(\alpha_n t)]
=\ell_n{\e}^{-2\alpha_n t}+O(1)-\sum_{k=0}^{\infty} k n_k\e^{-\alpha_n tk}
\end{equation}
and thus, 
by \eqref{svapsi} and the assumption $n\gam_n\to\infty$,
\begin{equation}\label{veum}
  \begin{split}
\frac{1}n\E \tA_n(\ga_n t)
&= \mu_n\e^{-2\ga_n t} - \E\bigpar{D_n\e^{-\ga_n t D_n}} +O\bigpar{n^{-1}}
=\gamn\psi_n(t)+o(\gam_n), 
\end{split}
\end{equation}
which proves \eqref{e:mean-A}.
\end{proof}

\begin{remark}\label{Rgamma}
In the case when $D_n^3$ is uniformly integrable, or equivalently
$\E D_n^3\to\E D^3<\infty$, the sequence
$(\ga_n^{-2}D_n)\land D_n^3$ is uniformly integrable (since $D_n^3$ is),
and converges a.s.\ to
$D^3$ if we assume $D_n\asto D$, as we may by \refSS{SSSkorohod};
consequently, using \eqref{svagam},
 	 \begin{equation}\label{gamD3}
	\frac{\gam_n}{\ga_n^2} = \E ((\ga_n^{-2}D_n)\land
	D_n^3)\to \E D^3 <\infty.	
  	\end{equation}
Thus, in this case, $\gam_n\asymp\gan^2$.
In other words, we could have defined $\gam_n$ as $\ga_n^2$ or, e.g.,
$\mu_n\ga_n^2$ in this case, instead of by \eqref{svagam} (provided we
modify $\psi_n$ accordingly).
Moreover, a simple calculation using \eqref{ganD3},
which we omit, shows that, with $\kk$ given
by \eqref{kk}--\eqref{kk4},
\begin{equation}\label{psiD3}
\psi_n(t):=\frac{\kk\mu}{2\E D^3} (t -t^2)+o(1), 
\end{equation}
uniformly on each compact interval; thus we may in this case as an
alternative take
$\psi_n(t):=\frac{\kk\mu}{2\E D^3} (t -t^2)$, independently of $n$.
(Cf.\ \refR{Rpsin1} and, with a simple change of time scale, \refR{R07T2.4}.)

On the other hand, if $\E D^3=\infty$,
then, assuming again $D_n\asto D$, we have
$\ga_n^{-2}D_n\land D_n^3 \asto D^3$ since $\ga_n\to0$.
Thus Fatou's lemma yields, instead of \eqref{gamD3},
$\gam_n/\ga_n^2\to\E D^3=\infty$, i.e.,
\begin{equation}\label{ga2gamo}
  \ga_n^2=o(\gam_n).
\end{equation}
Moreover, in this case it is, using \eqref{ga2gamo}, easy to see that
if we
define
\begin{equation}\label{gfn}
  \gf_n(t):=\E\bigpar{D_n(1-\e^{-t\ga_nD_n})}-2\ga_n\mu_nt,
\end{equation}
then
\begin{equation}
\psi_n(t):=\gf_n(t)/\gam_n +o(1)
\end{equation}
uniformly on each compact interval; thus we may in this case as an
alternative take
$ \psi_n(t):=\gf_n(t)/\gam_n$.

In both these cases we can thus use simpler versions of $\gam_n$ and
$\psi_n$; however, we prefer not to do so; instead we use definitions
\eqref{svagam}--\eqref{svapsi}, which work in all cases.
\end{remark}

\begin{remark}\label{Rgamn}
  Typically, as in \refE{exam-power-law},
$\E \bigpar{D_n((\ga_nD_n)\land(\ga_n D_n)^2)}
 \asymp\E \bigpar{D_n(1\land(\ga_n D_n)^2)}$ and then, by \eqref{svagam}
	 and \eqref{sard},
 \begin{equation}\label{rgamn}
\gamma_n=\E \bigpar{D_n(1\land(\ga_n D_n)^2)}
\asymp  \E\bigpar{D_n(\ga_nD_n\land(\ga_n D_n)^2)}
\asymp\ga_n\eps_n.
 \end{equation}
In this case, we could have used $\gam_n:=\ga_n\eps_n$ instead of the choice
\eqref{svagam} (provided we modify $\psi_n$ accordingly).
\end{remark}

We next show that random variables
$\tA_n(t)$, $\tV_n(t)$ and $\tS_n(t)$ are so well concentrated for all $t$ that we may replace them in conditions
\ref{XA}, \ref{XV}, \ref{XS} by their expectations.
For later use, we state the next estimate in a more general form than
needed here; we then give its simpler consequence in Lemma~\ref{LXC1}.

\begin{lemma}[Concentration of $\tS_n(t), \tA_n(t), \tV_n(t)$]
\label{LXC0}
Assume \refAA.
Then there exists a constant $C$ such that,
for any $u\ge0$,
\begin{align}
\E\Bigl[\sup_{t\le u}|\tS_n(t)-\E\tS_n(t)|^2\Bigr]
&\le C n \E \bigpar{D_n^2(1\land uD_n)},
\label{bs0}
\\
\E\Bigl[\sup_{t\le u}|\tV_n(t)-\E\tV_n(t)|^2\Bigr]
&\le C n \E \bigpar{D_n^2(1\land uD_n)},
\label{bv0}
\\
\E\Bigl[\sup_{t\le u}|\tA_n(t)-\E\tA_n(t)|^2\Bigr]
&\le C n \E \bigpar{D_n^2(1\land uD_n)}+C.
\label{ba0}
    \end{align}
\end{lemma}

The final ``${}+C$'' in \eqref{ba0} is probably an artefact of our proof, but it
is harmless for our purposes.

\begin{proof}
The process $\tV\nk(t)$ is a simple death process where each individual dies
with rate $k$;
it follows that
$\tV\nk(t)$ is a semimartingale with drift
$-k\tV\nk(t)$. 
Consequently, $\tS_n(t)=\sum_{k=0}^{\infty}k\tV\nk(t)$ is a
semimartingale with drift $-\sum_{k=0}^{\infty}k^2\tV\nk(t)$, and
$\tS_n(t)-\E\tS_n(t)$ is a semimartingale with drift
$\xi(t)\=-\sum_{k=0}^{\infty}k^2(\tV\nk(t)-\E\tV\nk(t))$.

We have, 
noting that
$\tV\nk(t)$ are independent and $\tV\nk(t) \sim\Bin(n_k,{\mathrm e}^{-kt})$ for
each $k$,
\begin{equation}\label{jb1}
  \begin{split}
\E|\tS_n(t)-\E\tS_n(t)|^2
    &=\sum_{k=0}^{\infty}\Var(k\tV\nk(t))
    =\sum_{k=0}^{\infty}k^2\Var(\tV\nk(t))
 \\&
    =\sum_{k=0}^{\infty}k^2n_k{\mathrm e}^{-kt}(1-{\mathrm e}^{-kt})
    \le\sum_{k=0}^{\infty} n_k k^2 \bigpar{kt\land (kt)^{-1}}.
  \end{split}
\end{equation}
Similarly
\begin{equation}\label{jb2}
  \begin{split}
\E|\xi(t)|^2
    &=\sum_{k=0}^{\infty}\Var(k^2\tV\nk(t))
    =\sum_{k=0}^{\infty}k^4\Var(\tV\nk(t))
    =\sum_{k=0}^{\infty}k^4n_k{\mathrm e}^{-kt}(1-{\mathrm e}^{-kt})
    \\&
    \le\sum_{k=0}^{\infty} n_k k^4 \e^{-kt}(1\land kt).
  \end{split}
\end{equation}
Hence, for some constant $C_1$,
\begin{equation}
  \begin{split}
\sum_{j=0}^\infty\E|\tS_n(2^{-j}u)-\E\tS_n(2^{-j}u)|^2
&
\le\sum_{k=0}^{\infty} n_k k^2
\sum_{j=0}^\infty\bigpar{2^{-j}ku\land (2^{-j}ku)^{-1}}.
\\&
\le\CC\sum_{k=0}^{\infty} n_k k^2 \xpar{ku\land 1}
  \end{split}
\end{equation}
and
\begin{equation*}
\int_0^u t \E \bigsqpar{\xi(t)}^2\dd t
 \le\sum_{k=0}^{\infty} n_k k^4 \int_0^u \e^{-kt}(t\land kt^2) \dd t
 \le\sum_{k=0}^{\infty} n_k k^2 (1\land (ku)^3).
\end{equation*}
Consequently, \refL{LX} yields
\begin{equation}\label{jb4}
  \begin{split}
    \E\Bigl[\sup_{t\le u}|\tS_n(t)-\E\tS_n(t)|^2\Bigr]
    &
\le\CC\sum_{k=0}^{\infty} n_k k^2 \xpar{ku\land 1}
+ \CC\sum_{k=0}^{\infty} n_k k^2 \xpar{(ku)^3\land 1}
\\&
\le\CC\sum_{k=0}^{\infty} n_k k^2 \xpar{ku\land 1} \CCdef\CCjb
=\CCx n \E \bigpar{D_n^2((uD_n)\land 1)}.
  \end{split}
\end{equation}
This yields \eqref{bs0}.

We obtain \eqref{bv0} similarly; the estimates are the same, but with
smaller powers of $k$, which can only help us.

Moreover, by a similar argument (but without having to sum over $k$), or
by \JL{Lemma 6.1} (with a modification for $u>1$),
	\begin{equation}\label{elfstrand}
	\E\Bigsqpar{ \sup_{t\le u}\lrabs{\LL_n(t)-\E \LL_n(t)}^2}
	\le \CC n(u\land1)
	\end{equation}
and thus, by \eqref{serrander},
	\begin{equation}\label{elfbrink}
	\E\Bigsqpar{ \sup_{t\le u}\lrabs{L_n(t)-\E L_n(t)}^2}
	\le \CC n(u\land1)+\CC.
	\end{equation}
By definition, $\tA_n(t)=L_n(t)-\tS_n(t)$,
and thus \eqref{ba0} follows by combining \eqref{elfbrink} and \eqref{bs0},
noting that $\E D_n^2(uD_n\land 1)\ge \P (D_n=1)(u\land 1)$ and
$\P(D_n=1)\to\P(D=1)>0$ by \refR{Rneq2}.
\end{proof}

\begin{lemma}[Concentration of $\tS_n(t), \tA_n(t), \tV_n(t)$]
\label{LXC1}
Assume \refAApos.
Let, as above, $\ga_n$ be given by \eqref{svaga}. Then,
for any fixed $t_0$,
    \begin{align}
    \sup_{t\le t_0}\lrabs{\tS_n(\gant)-\E\tS_n(\gant)}
&=\Op\bigpar{(n\eps_n)^{1/2}},
    \label{bs}\\
    \sup_{t\le t_0}\lrabs{\tV_n(\gant)-\E\tV_n(\gant)}
&=\Op\bigpar{(n\eps_n)^{1/2}},
    \label{bv}\\
    \sup_{t\le t_0}\lrabs{\tA_n(\gant)-\E\tA_n(\gant)}
&=\Op\bigpar{(n\eps_n)^{1/2}+1}.
    \label{ba}
    \end{align}
\end{lemma}
\begin{proof}
  Taking $u=\gan t_0$, we obtain by \eqref{sard},
\begin{equation}
  \E \bigpar{D_n^2(1\land u D_n)}
\le (1\vee t_0)\E \bigpar{D_n^2(1\land \gan D_n)}
=O(\eps_n).
\end{equation}
Thus the \rhs{s} of \eqref{bs0}--\eqref{bv0} and \eqref{ba0} are
$O(n\eps_n)$ and $O(n\eps_n+1)$, respectively; hence
\eqref{bs}--\eqref{ba} follow using Markov's inequality.
\end{proof}

The final three lemmas provide further estimates of the quantities
$\alpha_n$ and $\gam_n$ defined in~\eqref{svaga} and~\eqref{svagam}.

\begin{lemma}
  \label{Lredlund}
Assume \refAApos.
If $\ga_n\gD_n=O(1)$, then
\begin{align}
\rho_n&\sim \ga_n \asymp \frac{\eps_n}{\E D_n^3}, \label{reda}
\\
  \gam_n & \asymp \ga_n\eps_n \asymp \frac{\eps_n^2}{\E D_n^3}.\label{redb}
\end{align}
\end{lemma}
\begin{proof}
  We have $\ga_n D_n\le \ga_n\gD_n =O(1)$, and thus
  \begin{equation}\label{red}
(1\land\ga_n D_n) \asymp \ga_n D_n.
  \end{equation}
Hence \eqref{sard} implies
\begin{equation}\label{redx}
  \eps_n \asymp \E(\ga_n D_n^3)
\end{equation}
and \eqref{reda} follows, recalling \eqref{garho2}.

Furthermore, \eqref{red} and \eqref{svagam} yield, using
\eqref{redx},
\begin{equation}
  \gam_n \asymp \E (\ga_n^2D_n^3) = \ga_n\E (\ga_nD_n^3) \asymp \ga_n\eps_n,
\end{equation}
showing  \eqref{redb}.
\end{proof}

\begin{lemma}\label{Lmarta}
Assume \refAApos.
If
\begin{align}
  (n\eps_n)\qq &= o(n\gam_n)\label{elfeps},
\end{align}
then \ref{Xdmax} holds, \ie,
\begin{align}
\gD_n&=o(n\gam_n) \label{elfgd}.
\end{align}
\end{lemma}
\begin{proof}
Suppose first that $\ga_n\gD_n\le1$. Then, using \eqref{redb}, 
\begin{equation}
  \frac{\gD_n}{n\gam_n}
\le  \frac{1}{n\gam_n\ga_n}
= O\Bigpar{\frac{\eps_n}{n\gam_n^2}}
= O\Bigpar{\frac{n\eps_n}{(n\gam_n)^2}},
\end{equation}
and thus \eqref{elfgd} follows from \eqref{elfeps} in this case.

Suppose next that $\ga_n\gD_n\ge1$.
Since $\P(D_n=\gD_n)\ge 1/n$, we have by \eqref{sard}
	\begin{equation}
	\eps_n\asymp
  	\E\bigpar{D_n^2(1\land (\ga_nD_n))}
	\ge \frac{1}n \gD_n^2(1\land (\ga_n\gD_n))
	=\frac{\gD_n^2}{n}.
	\end{equation}
Consequently, $\gD_n=O\bigpar{(n\eps_n)\qq}$, and thus \eqref{elfeps}
implies \eqref{elfgd} in this case too.
\end{proof}

\begin{lemma}\label{Lolof}
Assume \refAApos.
Then
	\begin{equation}\label{ulla}
\eps_n^2 = O\bigpar{\gam_n \E D_n^3}.	
	\end{equation}
\end{lemma}

\begin{proof}
 The Cauchy--Schwarz inequality yields, using \eqref{svagam},
  \begin{equation}
\bigpar{\E\bigpar{D_n^2(1\land\ga_nD_n)}}^2
\le
\E\bigpar{D_n(1\land\ga_nD_n)^2}
\E\bigpar{D_n^3}
=O\bigpar{\gam_n\E D_n^3}.
  \end{equation}
Hence the result follows by \eqref{sard}.
\end{proof}

\begin{proof}[Proof of Theorem \ref{T1}]
First note that \ref{Xtau}--\ref{Xpsi} hold for the
parameter values in~\eqref{svatau}--\eqref{svapsi} by \refL{L123}.

Next, by \refL{Lolof},
\begin{equation}
  \begin{split}
\frac{n\eps_n}{(n\gam_n)^2}
=\frac{\eps_n^4}{\gam_n^2 n\eps_n^3}
=O\lrpar{\frac{(\E D_n^3)^2}{n\eps^3_n}},	
  \end{split}
\end{equation}
which is $o(1)$ by the assumption. Hence \eqref{elfeps} holds.
Consequently,
\refL{Lmarta}
shows that \eqref{elfgd} holds. In other words, \ref{Xdmax} holds.

Since $\gD_n\ge1$, \eqref{elfgd} implies
\begin{equation}
  \label{elf1}
n\gam_n\to\infty,
\end{equation}
and thus
\refL{L-E-SVA} applies and shows \eqref{e:mean-S}--\eqref{e:mean-A}.

Moreover, \eqref{elfeps} and \eqref{elf1} imply that the
right-hand sides of \eqref{bs}--\eqref{ba} are $\op(n\gam_n)$.
Furthermore, $\gam_n=O(\ga_n)$ by
\ref{Xag}, see \eqref{gamga}.
Hence Lemmas \ref{L-E-SVA} and \ref{LXC1} yield \ref{XA}, \ref{XV} and \ref{XS}.

We have verified \ref{Xfirst}--\ref{Xlast}, so
\refT{TX} applies and the result follows,
recalling \eqref{svatau}, \eqref{svagh}, \eqref{mu} and  \eqref{garho2}.
Note that $\hg(\tau)=\hh(\tau)/2$, so the asymptotics for $\vx{\cC_1}$ and
$\ex{\cC_1}$ are the same.
\end{proof}

\begin{proof}[Proof of Theorem \ref{TD3}]
By assumption, $\E D_n^3=O(1)$ and $\eps_nn\qqq\to\infty$, so \refT{T1}
applies;
thus \eqref{t11} holds.
Furthermore, as said in \refSS{SSrho},
\refT{Tsurv}\ref{tsurv1} applies with $X_n = \tilde{D}_n$ and yields \eqref{ganD3},
which together with \eqref{mu} yields the first equality in \eqref{tdn3};
the second equality then follows by \eqref{eps}.
Similarly, \eqref{t12} and
\eqref{ganD3} (or \eqref{rhoO}) yield \eqref{td32}.
\end{proof}

\begin{proof}[Proof of \refT{TD3infty}]
  Again, \refT{T1} applies.
Moreover,  by \eqref{ftD}, we have
$\E \tD^2=\E \bigpar{D(D-1)^2}/\E D=\infty$,  and so
\refT{Tsurv}\ref{tsurv2} applies, yielding
$\rho_n=o(\eps_n)$.
\end{proof}

\begin{proof}[Proof of \refT{TDx}]
\refT{T1} applies.
\pfitemref{TDxa} Follows from \eqref{t11}, \eqref{rho>} for $X_n = \tilde{D}_n$ and \eqref{kkn}.

\pfitemref{TDxb}
Now, by \eqref{ftD}, $\E \tD_n^2\le \E D_n^3/\E D_n=O(1)$.
Hence \refT{Tsurv}\ref{tsurv0} applies and yields
$\rho_n\asymp \eps_n$; consequently \eqref{t11} implies \eqref{ele}.

\pfitemref{TDxc}
By \eqref{kknR} and \eqref{kkn},
$\E D_n^3=O(\kk_n)=O\bigpar{\E[\tD_n(\tD_n-1)]}=O\bigpar{\E[\tD_n^2]}$.
Thus, the assumption implies
$\eps_n\gD_n=o\bigpar{\E[\tD_n^2]}$.
Hence, \refT{Tsurv}\ref{tsurv3} applies and
\eqref{win} follows by \eqref{kkn} and \eqref{kknR}.
\end{proof}

\section{The critical case}
\label{Scritical}

We define, for convenience and for comparison with \citet{HatamiMolloy},
	\begin{equation}\label{Rn}
  	R_n:=\E D_n^3.
	\end{equation}
The basic condition for the critical case in \refT{TC} is thus, as in
\cite{HatamiMolloy},
	\begin{equation}
  	\label{epscrit}
	\eps_n = O\bigpar{n\qqqw R_n^{2/3}}.
	\end{equation}

\begin{remark}
  Our $R_n$ is not exactly the same as $R$ defined by \citet{HatamiMolloy},
which equals our $\E D_n(D_n-2)^2/\E D_n=\kk_n-\eps_n$,
see \eqref{kkn} and \eqref{eps},
but the two values are equivalent in the sense
$R_n\asymp  R_{\textrm{HatamiMolloy}}$, see \eqref{kknR} and \eqref{kkn>};
hence the two values are equivalent for our purposes.
\end{remark}

Note that, as said in \refR{RgDo}, $R_n\ge \frac{1}n\gD_n^3$ and hence always
	\begin{equation}\label{gDR}
  	\gD_n\le (nR_n)\qqq.
	\end{equation}
Note also that in \refT{TC} we impose the slightly stronger condition \eqref{gDo},
\ie,
	\begin{equation}\label{gDRo}
  	\gD_n=o\bigpar{ (nR_n)\qqq}.
	\end{equation}
Furthermore, by \ref{A2},
  	\begin{equation}\label{RnqqO}
	R_n = \E D_n^3\le \gD_n\E D_n^2 = O(\gD_n).
  	\end{equation}
Hence, \eqref{gDRo} implies
$ 
\gD_n^3=o(nR_n)=o(n\gD_n)
$ 
and thus $\gD_n^2=o(n)$ and
	\begin{equation}\label{gDqqo}
  	\gD_n=o\bigpar{n\qq},
	\end{equation}
and thus also, by \eqref{RnqqO},
  \begin{equation}\label{Rnqqo}
R_n = o\bigpar{n\qq}.
  \end{equation}

In \refT{TC} we assume both \eqref{epscrit} and \eqref{gDRo}, and it follows
from \eqref{epscrit} and \eqref{Rnqqo} that
$\eps_n=o(1)$, so \ref{AO} follows from the other conditions. (However, for
emphasis we keep it in the statements in \refT{TC} and below.)

Note also that, using \eqref{EDn2} and \eqref{mu},
$ 
R_n\ge\E D_n^2\to 2\mu>0,
$ 
so $R_n$ is bounded below and $R_n\qw=O(1)$.

We continue to work with the configuration model and the multigraph $G_n^*$
as in the preceding section. In \refSS{SSTC>graph} we give additional
arguments for the graph case.

\subsection{Proof of \refT{TC}\ref{TC<}}
The idea is to use \refT{T1} for the supercritical case and a kind of
monotonicity in $\eps_n$; it is intuitively clear that a larger $\eps_n$ ought
to result in a larger largest component, and thus the supercritical case
will provide an upper bound for the critical case.
The formal details are as follows.

\begin{proof}[Proof of \refT{TC}\ref{TC<}]
Let $\go(n)\to\infty$ slowly, so slowly that,
\cf{} \eqref{Rnqqo} and \eqref{gDRo},
\begin{align}
  \go(n)R_n &=o\bigpar{n\qq},\label{go1}
\\
\go(n)\gD_n &\le (nR_n)\qqq.\label{go2}
\end{align}

Let $m_n:=\floor{n^{2/3} R_n^{2/3}\go(n)\qqqb}$. Change the degree sequence
$(d_i)_{i\in[n]}$  to $(\hd_i)_{i\in[n]}$ by replacing $2m_n$ vertices of degree
1 by $m_n$ vertices of degree 0 and $m_n$ vertices of degree 2.
This is possible (at least for large $n$)
because $n_1/n=\P(D_n=1)\to\P(D=1)>0$, see \refR{Rneq2}, and
thus, using \eqref{go1},
	\begin{equation}\label{gom}
  	m_n\le n^{2/3} (R_n\go(n))\qqqb =o(n)=o(n_1).
	\end{equation}

We denote the variables for the modified degree sequence by $\hD_n$ and so
on.
Note that the modification does not change the sum of vertex degrees, so
$\E\hD_n=\E D_n=\mu_n$, but it increases $\E[D_n(D_n-1)]$ by
$2m_n/n\sim 2 n^{-1/3}R_n^{2/3}\go(n)\qqqb$.
Thus, using \eqref{epscrit} and $\go(n)\to\infty$,
	\begin{equation}\label{hepsn}
  	\hat\eps_n = \eps_n+2m_n/n \sim 2 n^{-1/3}R_n^{2/3}\go(n)\qqqb.
	\end{equation}
Similarly, $R_n=\E D_n^3$ is increased to
\begin{equation}\label{hR}
  \hat R_n = \E \hD_n^3 =R_n + \frac{6m_n}n =R_n+o(1)\sim R_n,
\end{equation}
where we used \eqref{gom} to see that the difference is insignificant.
Furthermore, it is easily seen that \refAA{} still hold (with the same $D$),
using \eqref{gom} and \eqref{hepsn}
for \ref{AA} and \ref{AO}.

Since \eqref{hepsn} and \eqref{hR} imply $\hat\eps_n\gg n\qqqw \hR_n^{2/3}$,
and \eqref{hepsn} and \eqref{go2} imply $\hat\eps_n\gD_n=o(R_n)$,
\refT{TDx}\ref{TDxc} applies to the modified degree sequence and yields, \whp,
\begin{equation}\label{asa}
\vx{\hat\cC_1}\le C'\frac{\hat\eps_n n}{R_n}
=o\bigpar{n\qqqb R_n\qqqw \go(n)}.
\end{equation}
In particular, \whp{}
\begin{equation}
  \label{ronja}
\vx{\hat\cC_1}\le  n\qqqb R_n\qqqw \go(n).
\end{equation}

We can obtain
$G^*(n,\hddn)$ from $G^*(n,\ddn)$ by merging $m_n$ pairs of vertices of
degree 1 into vertices of degree 2, and adding $m_n$ vertices of degree 0 to
keep the total number of vertices. Any connected set $\cC$ of vertices in
$G^*(n,\ddn)$ then corresponds to a connected set of at least $\vx{\cC}/2$
vertices in $G^*(n,\hddn)$. Consequently, $\vx{\hat\cC_1}\ge\frac12\vx{\cC_1}$
and thus \eqref{ronja} implies,
\whp,
\begin{equation}\label{johan}
  \vx{\cC_1}\le
2\vx{\hat\cC_1}\le
2 n\qqqb R_n\qqqw \go(n).
\end{equation}

Since $\go(n)\to\infty$ arbitrarily slowly, \eqref{johan} implies
$\vx{\cC_1}=\Op(n\qqqb R_n\qqqw)$.
(If not, we could find $\gd>0$ and $K=K(n)\to\infty$ such that, at least
along a subsequence, $\P \bigpar{\vx{\cC_1} \ge K(n) n^{2/3} R_n^{-1/3}} \ge \delta$.
We choose $\go(n)$ with
$\go(n)\le K(n)/2$ to obtain a contradiction.
See also \cite{SJN6}.) This completes our proof of \eqref{tc<}.
\end{proof}

\begin{remark}
In our proof we needed only the simple, deterministic bound
$\vx{\cC_1}\le 2\vx{\hat\cC_1}$. Actually, when \refT{TC}\ref{TC<} is  proved,
it implies together with
\refT{TDx}\ref{TDxc} that \whp{} $\vx{\cC_1}\ll \vx{\hat\cC_1}$, \ie, that
the giant component $\hat\cC_1$ for the modified sequence \whp{}
is much larger
than $\cC_1$ for the original sequence; the reason is that, in the merging
described above, the giant component typically absorbs many small components.
\end{remark}

\begin{example}\label{Efel}
Consider a critical example with $\eps_n=O(n\qqqw)$, $R_n=O(1)$ and
$\gD_n=o(n\qqq)$. For example (as in \cite{HatamiMolloy}), we can
let $3/4$ of all vertices have degree 1 and the rest degree 3. Alternatively,
we can take the \ER{} graph $G(n,1/n)$ and condition on the degree
sequence, as described for general rank-1 inhomogeneous random graphs
in \refS{sec-discussion}. Then $v(\cC_1)$ is typically of order $n\qqqb$,
see \refT{TC} and \cite{HatamiMolloy}.

Let $m_n$ be integers with $n\qqq\ll m_n\ll n\qq$.
Modify the degree sequence $\ddn$ to $\hddn$ by merging $m_n$ vertices of
degree 1 to a single vertex of degree $m_n$, and adding $m_n-1$ vertices of
degree 0. Then it is easily seen that $\hat\eps_n\asymp m_n^2/n$, $\hR_n\sim
m_n^3/n$ and $\hat\gD_n=m_n$. Thus \eqref{epscrit} holds for the modified
sequence but not \eqref{gDRo}. Furthermore,
\begin{equation}
  \vx{\hat\cC_1}
\ge   \vx{\cC_1}-m_n
\end{equation}
so $  \vx{\hat\cC_1}$ is typically also of order (at least) $n\qqqb$. Hence,
\eqref{tc<} fails.
\end{example}

\subsection{Proof of \refT{TC}\ref{TC>} in the multigraph case}
In this section, we consider only the multigraph case.
Unlike all other results in this paper, the graph case does not follow
immediately by conditioning. We treat the graph case in the next section.

We use the cluster exploration process and notation from \refS{SSgeneral}.
Let
	\begin{equation}\label{t1}
  	t_1:=\bigpar{nR_n}\qqqw,
	\end{equation}
and note that $t_1=O(n\qqqw)=o(1)$ and, by \eqref{Rnqqo}, $t_1\gg n^{-1/2}$ and thus
$nt_1\to\infty$
and $nt_1^2\to\infty$.
Furthermore, let
	\begin{equation}
  	\gss_n:=\Var \tS_n(t_1).
	\end{equation}

\begin{lemma}\label{LB1}
Assume \refAA{} and \eqref{gDRo}.
\begin{romenumerate}
\item \label{LB1S}
Then
\begin{equation}\label{gssR}
\gss_n
\sim
\bigpar{nR_n}^{2/3}.
 \end{equation}
Moreover, $\tS_n(t_1)$ is asymptotically normal:
\begin{equation}\label{asnS}
\bigpar{\tS_n(t_1)-\E\tS_n(t_1)}/\gs_n\dto N(0,1).
\end{equation}

\item \label{LB1L}
Let $\gsLn^2:=4nt_1\mu_n$. Then
$L_n(t_1)$ is asymptotically normal, with
\begin{equation}\label{asnL}
\bigpar{L_n(t_1)-\E L_n(t_1)}/\gsLn\dto N(0,1).
\end{equation}
Furthermore, $\limsup \gsLn^2/\gss_n<1$.
\item \label{LB1A}
For any $b>0$, there exists $c(b)>0$ such that
\begin{equation}\label{Ssmall}
\P\bigpar{\tA_n(t_1)-\E\tA_n(t_1)>b\gs_n}\ge c(b)+o(1).
\end{equation}
\end{romenumerate}
\end{lemma}

\begin{proof}
\pfitemref{LB1S}
We have, see \refS{SSgeneral} and in particular \eqref{m22},
$\tS_n(t)=\sumin d_i I_i(t)$, where $I_i(t)$ is the indicator that no
half-edge at vertex $i$ has died spontaneously up to time $t$.
These indicators are independent and $I_i(t)\sim \Be(\ee{-d_it})$.
Hence, as in \eqref{jb1} but written slightly differently,
noting that $t_1d_i\le t_1\gD_n=o(1)$ by
\eqref{t1} and \eqref{gDRo},
\begin{equation*}
  \begin{split}
  \Var \tS_n(t_1)=\sumin d_i^2\Var I_i(t_1)
=\sumin d_i^2 \ee{-d_it_1}\bigpar{1-\ee{-d_it_1}}
\sim \sumin d_i^3 t_1
=t_1nR_n = \bigpar{nR_n}^{2/3},	
  \end{split}
\end{equation*}
which is \eqref{gssR}.
Similarly, with $Y_i:=d_iI_i(t_1)$ and using \eqref{gDRo},
\begin{equation}
  \begin{split}
\sumin   \E|Y_i-\E Y_i|^3
&
=\sumin d_i^3\E|I_i(t_1)-\E I_i(t_1)|^3
\le \sumin d_i^3 \Var I_i(t_1)
\le \sumin d_i^4 t_1
\\&
=t_1n\E D_n^4
\le t_1n\gD_nR_n
=o(nR_n) = o\bigpar{\gs_n^3}.
  \end{split}
\end{equation}
Consequently, the central limit theorem with Lyapounov's condition
\cite[Theorem 7.2.2]{Gut}
applies and yields \eqref{asnS}.

\pfitemref{LB1L}
We use the modified process $\LL_n(t)$ defined just before \eqref{EL}.
Then $\frac12\LL_n(t)\sim\Bin\bigpar{\frac12\ell_n,\ee{-2t}}$ for every $t\ge0$.
In particular, recalling from~\eqref{elln} and \eqref{mun-def}
that $\ell_n=n\mu_n$,
\begin{equation}
  \Var \LL_n(t_1) = 4\cdot\tfrac12\ell_n \ee{-2t_1}\bigpar{1-\ee{-2t_1}}
\sim 4\ell_n t_1=\gsLn^2.
\end{equation}
Since $nt_1\to\infty$, we have $\gsLn^2\to\infty$, and the central limit
theorem for the binomial distribution yields
$\bigpar{\LL_n(t_1)-\E \LL_n(t_1)}/\gsLn\dto N(0,1)$.
Since $|L_n(t_1)-\LL_n(t_1)|\le1$
by \eqref{serrander},
\eqref{asnL} follows.

Furthermore,
	\begin{equation}
 	\frac{\gss_n}{\gsLn^2}
	=
 	\frac{\gss_n}{4nt_1\mu_n}
	\sim
	\frac{(nR_n)^{2/3}}{4n(nR_n)^{-1/3}\mu_n}
	=
	\frac{R_n}{4\mu_n}
	=
	\frac{\E D_n^3}{4\E D_n}.
	\end{equation}
Consequently, using \eqref{lyman} (with $t=0$),
\begin{equation}
  \begin{split}
\liminf_\ntoo \frac{\gss_n}{\gsLn^2}
&=
\frac{\liminf \E D_n^3}{4\E D}
\ge
\frac{\E \bigpar{D(D-2)^2}+4\mu}{4\mu}
>1.	
  \end{split}
\end{equation}

\pfitemref{LB1A}
By \ref{LB1L}, there exists $\gd>0$ such that, for large $n$,
$\gsLn<(1-2\gd)\gs_n$. Let $a:=\gd\qw b$ and let $\Phi$ be the usual
standard normal distribution function.
Then, by \eqref{asnS} and \eqref{asnL},
\begin{align}
  \P\bigpar{\tS_n(t_1)-\E\tS_n(t_1)<-a\gs_n}&\to \Phi(-a),
\\
  \P\bigpar{L_n(t_1)-\E L_n(t_1)<-(1+\gd)a\gsLn}&\to \Phi(-(1+\gd)a).
\end{align}
Hence, with probability at least $c+o(1)$, where
$c:=\Phi(-a)-\Phi(-(1+\gd)a)>0$,
we have
$\tS_n(t_1)-\E\tS_n(t_1)<-a\gs_n$
and
$L_n(t_1)-\E L_n(t_1)\ge -(1+\gd)a\gsLn$, and thus, recalling \eqref{m33},
\begin{equation}
  \tA_n(t_1)-\E\tA_n(t_1)>a\gs_n -(1+\gd)a\gsLn
>a\gs_n -(1+\gd)(1-2\gd)a\gs_n >\gd a\gs_n=b\gs_n.
\end{equation}
\end{proof}

\begin{remark}
  Presumably, $\tS_n(t_1)$ and $L_n(t_1)$ are asymptotically {\em jointly} normal,
which would imply that $\tA_n(t_1)$ is asymptotically normal and yield a
more direct proof of \eqref{Ssmall}. However, it seems more technical to
prove joint asymptotic normality here, so instead we prefer the more elementary
argument above.
\end{remark}

\begin{lemma}
\label{LB2a}
Assume \refAA{} and \eqref{epscrit} and \eqref{gDRo}.
Then, uniformly for $t\le t_1$,
\begin{align}
 \E \tS_n(t) &=  n\mu_n -2tn\mu_n +  O(\gs_n),\label{drake}
\\
  \E L_n(t) &= n\mu_n\ee{-2t}+O(1)= n\mu_n -2tn\mu_n + O(\gs_n) ,
\label{drakeL}
\\
  \E \tA_n(t) &=  O(\gs_n). \label{drakeA}
 \end{align}
\end{lemma}

\begin{proof}
Similarly to the proof of \refL{L-E-SVA},
$V\nk(t)\sim \Bin(n_k,\ee{-kt})$ and thus, using \eqref{EDn2},
\begin{equation}\label{mask}
  \begin{split}
\E \tS_n(t)
&= \sumko k\E V\nk(t) =\sumko kn_k \ee{-kt}
= \sumko kn_k\bigpar{1-kt+O(k^2t^2)}
\\&
=n\E D_n-t n\E D_n^2 + O(t^2n\E D_n^3)
\\&
=n\mu_n-tn\mu_n(2+\eps_n) + O(t_1^2 nR_n)
,
  \end{split}
\end{equation}
which yields \eqref{drake} by   \eqref{epscrit}, \eqref{t1} and \eqref{gssR}.

Furthermore, by \eqref{EL},
\begin{equation}
  \E L_n(t)=n\mu_n\ee{-2t}+O(1)
= n\mu_n-2tn\mu_n+O(nt_1^2+1),
\end{equation}
and \eqref{drakeL} follows because,
by \eqref{t1} and \eqref{gssR},
\begin{equation}\label{ntt}
  nt_1^2+1\sim nt_1^2 =n\qqq R_n^{-2/3}\sim \gs_n R_n\qw=O(\gs_n).
\end{equation}

Finally, \eqref{drakeA} follows from \eqref{drake} and \eqref{drakeL}.
\end{proof}

\begin{lemma}\label{LB8}
  Assume \refAA{} and \eqref{epscrit} and \eqref{gDRo}.
Then,
\begin{align}
  \E\Bigsqpar{\sup_{t\le t_1}\bigabs{\tA_n(t)}^2} &=  O(\gss_n). \label{drakeAA}
 \end{align}
\end{lemma}

\begin{proof}
  By \refL{LXC0}, together with \eqref{t1} and \eqref{gssR},
  \begin{equation}
	\begin{split}
\E\Bigl[\sup_{t\le t_1}|\tA_n(t)-\E\tA_n(t)|^2\Bigr]
&\le C n \E \bigpar{D_n^2(1\land t_1D_n)}+C
\\&
\le C n t_1 \E \bigpar{D_n^3}+C
=Cnt_1 R_n+C
=O(\gss_n).	
	\end{split}
  \end{equation}
Furthermore, $\sup_{t\le t_1}|\E\tA_n(t)|=O(\gs_n)$ by \eqref{drakeA},
and \eqref{drakeAA} follows.
\end{proof}

For ease of notation, let $N_k:=\tV_{n,k}(t_1)$,
the (random) number of vertices of degree $k$ such that none of their half-edges
dies spontaneously by time $t_1$. Thus $\tS_n(t_1)=\sum_k kN_k$, see
\eqref{m22}.
Let further
\begin{equation}\label{Zn}
Z_n:=\sumko k^2(n_k-N_k)\ge0.
\end{equation}

\begin{lemma}  \label{LZ}
Assume \refAA{} and \eqref{epscrit} and \eqref{gDRo}.
Then,
there exists a constant $\CC\CCdef\CCZ$ such that \whp{}
\begin{equation}\label{Zbound}
  Z_n\le \CCZ \gss_n.
\end{equation}
\end{lemma}

\begin{proof}
$N_k\sim \Bin(n_k,\ee{-kt_1})$ and thus,
using \eqref{t1} and \eqref{gssR},
\begin{equation}\label{EZ}
  \begin{split}
\E Z_n
&= 	
\sumko k^2n_k\bigpar{1- \ee{-kt_1}}
\le \sumko k^3n_k t_1
=t_1 n R_n =O(\gss_n).
  \end{split}
\end{equation}
Furthermore, using also \eqref{gDRo},
\begin{equation}\label{VZ}
  \begin{split}
\Var Z_n
&= 	
\sumko k^4\Var N_k
\le
\sumko k^4n_k\bigpar{1-\ee{-kt_1}}
\le \sumko k^5n_k t_1
\\&
=t_1 n \E D_n^5 \le t_1 n \gD_n^2 R_n
=o\bigpar{(nR_n)^{4/3}}
=o(\gs_n^4).
  \end{split}
\end{equation}
Now \eqref{Zbound} follows by \eqref{EZ}--\eqref{VZ} and Chebyshev's
inequality.
\end{proof}

We condition on $\cF_{t_1}$, the $\gs$-field generated by all events up to
time $t_1$. Note that $\cF_{t_1}$ determines $N_k$, and thus $\tS_n(t_1)$
and $Z_n$, and also $L_n(t_1)$ and $\tA_n(t_1)$.

\begin{lemma}
  \label{LB3}
Assume \refAA{} and \eqref{epscrit} and \eqref{gDRo}.
For any fixed $\KK <\infty$ and all $t\in[0,\KK t_1]$,
\begin{align}
 \E\bigpar{\tS_n(t_1+t)\mid \cF_{t_1}} &=\tS_n(t_1)-2tn\mu_n+tZ_n + O(\gs_n)
\label{lb3},
  \\
 \E\bigpar{L_n(t_1+t)\mid \cF_{t_1}} &\ge L_n(t_1)-2tn\mu_n + O(\gs_n)
\label{lb3L},
\\
 \E\bigpar{\tA_n(t_1+t)\mid \cF_{t_1}} &\ge\tA_n(t_1)-tZ_n + O(\gs_n)
\label{lb3A}.
\end{align}
\end{lemma}

\begin{proof}
  We have, in analogy with \eqref{mask}, using \eqref{Zn},
  \begin{equation}
	\begin{split}
\E\bigpar{\tS_n(t_1+t)\mid\cF_{t_1}}
&
=\sumko k N_k \ee{-kt}
=\sumko k N_k \bigpar{1-kt+O(k^2t^2)}
\\&
=\tS_n(t_1)-t\biggpar{\sumko k^2n_k -Z_n}+O\biggpar{t^2\sumko k^3n_k}
\\&
=\tS_n(t_1)-t n \E D_n^2 +tZ_n+O\bigpar{t^2n R_n}.
	\end{split}
  \end{equation}
Then \eqref{lb3} follows by \eqref{EDn2} and estimates
as in the proof of
\refL{LB2a},
using   \eqref{epscrit}, \eqref{t1}, \eqref{gssR}
and the assumption $t=O(t_1)$.

For $L_n$ we use again the coupling with $\LL_n$. As
$\frac12\LL_n(t)$ is a standard death process with intensity 2,
\begin{equation}
  \begin{split}
\E\bigpar{L_n(t_1+t)\mid\cF_{t_1}}
&=
\E\bigpar{\LL_n(t_1+t)\mid\cF_{t_1}}+O(1)
=\LL_n(t_1)\ee{-2t}+O(1)
\\&
=L_n(t_1)-2t L_n(t_1)+O(1+nt^2).	
  \end{split}
\end{equation}
Then \eqref{lb3L} follows, since $L_n(t_1)<\ell_n=n\mu_n$, using again
\eqref{ntt}.

Finally, \eqref{lb3A} follows from \eqref{lb3} and \eqref{lb3L}.
\end{proof}

\begin{lemma}
  \label{LB4}
Assume \refAA.
For any fixed $\KK <\infty$ and all $t\in[0,\KK t_1]$,
\begin{align}\label{lb4}
\E\Bigl[\sup_{t\le \KK t_1}
  \Bigabs{\tS_n(t_1+t)-\E\bigpar{\tS_n(t_1+t)\mid\cF_{t_1}}}^2 \Bigm|\cF_{t_1}
\Bigr]
&=O(\gss_n),
\\
\E\Bigl[\sup_{t\le \KK t_1}
  \Bigabs{L_n(t_1+t)-\E\bigpar{L_n(t_1+t)\mid\cF_{t_1}}}^2 \Bigm|\cF_{t_1}
\Bigr]
&=O(\gss_n), \label{lb4L}
\\
\E\Bigl[\sup_{t\le \KK t_1}
  \Bigabs{\tA_n(t_1+t)-\E\bigpar{\tA_n(t_1+t)\mid\cF_{t_1}}}^2 \Bigm|\cF_{t_1}
\Bigr]
&=O(\gss_n). \label{lb4A}
\end{align}
\end{lemma}
\begin{proof}
  Conditioned on $\cF_{t_1}$, the process $\tS_n(t_1+t)$ is exactly as
  $\tS_n(t)$, but starting with $N_k$ vertices of degree $k$ instead of
  $n_k$.
Hence the arguments in \eqref{jb1}--\eqref{jb4} in the proof of \refL{LXC0}
hold in this case too and, since $N_k\le n_k$, we obtain, for any $u\ge0$,
\begin{equation*}
\E\Bigl[\sup_{t\le u}
\Bigabs{\tS_n(t_1+t)-\E\bigpar{\tS_n(t_1+t)\mid\cF_{t_1}}}^2\Bigm|\cF_{t_1}
 \Bigr]
\le\CCjb n \E \bigpar{D_n^2(uD_n\land 1)}
\le\CCjb nu \E {D_n^3}.
\end{equation*}
The result \eqref{lb4} follows by taking $u=\KK t_1$, using again \eqref{Rn},
\eqref{t1} and \eqref{gssR}.

Similarly,  as in \eqref{elfbrink}, or by \cite[Lemma 6.1]{JanLuc07} after
conditioning on $\cF_{t_1}$, we obtain,
since $t_1=o(1)$,
\begin{equation}\label{lexe}
\E\Bigl[\sup_{t\le \KK t_1}
\Bigabs{L_n(t_1+t)-\E\bigpar{L_n(t_1+t)\mid\cF_{t_1}}}^2\Bigm|\cF_{t_1}
 \Bigr]
=O(nt_1+1).
\end{equation}
Furthermore,  as said above, $nt_1\to\infty$ and $R_n\qw=O(1)$,
and thus, \cf{} \eqref{ntt},
\begin{equation}
nt_1+1\asymp nt_1=n^{2/3}R_n\qqqw \asymp\gss_n R_n\qw =O(\gss_n).
\end{equation}
Hence, \eqref{lexe} yields \eqref{lb4L}.
Finally,
\eqref{lb4A}
follows by combining \eqref{lb4} and \eqref{lb4L}.
\end{proof}

\begin{lemma}\label{LB5}
Assume \refAA{} and \eqref{epscrit} and \eqref{gDRo}.
For any fixed $\KK >1$, there is some $p(\KK )>0$ such that
with probability at least $p(\KK )+o(1)$,
\begin{equation}\label{lb5}
\tA_n(t)
>0
\qquad\text{for all $t\in[t_1,\KK t_1]$}.
\end{equation}
\end{lemma}

\begin{proof}
  Fix $\KK >1$. Let $b>0$ be another fixed number, to be determined later.
Consider the event
\begin{equation}\label{ceb}
\ceb:=\bigset{\tA_n(t_1)-\E\tA_n(t_1)>b\gs_n
\quad\text{and}\quad Z_n\le \CCZ\gss_n},
\end{equation}
with $\CCZ$ as in \refL{LZ}.
By \eqref{Ssmall} and \eqref{Zbound}, $\P(\ceb)\ge c(b)+o(1)$, where $c(b)>0$
is independent of $n$.
Define also the family of events $\{\cec :  C>0\}$, with $\cec$ given by
\begin{equation}\label{cec}
\cec:=\Bigset{
\sup_{t\le \KK t_1}
  \Bigabs{\tA_n(t_1+t)-\E\bigpar{\tA_n(t_1+t)\mid\cF_{t_1}}}
\le C \gs_n}.
\end{equation}
Further, let
\begin{equation}\label{cebc}
  \cebc:=\ceb\cap \cec.
\end{equation}
Note that $\ceb\in\cF_{t_1}$. Hence,
by \refL{LB4} and Chebyshev's inequality,
there exists a constant $\CC\CCdef\CCpec$ such that
\begin{equation}\label{pec}
  \P(\cec\mid\ceb)
\ge1-\frac{\CCpec \gss_n}{(C\gs_n)^2}=1-\frac{\CCpec}{C^2}.
\end{equation}
Consequently, if we choose $C:=2\CCpec\qq$,
then
\begin{equation}
  \P(\cebc)=\P(\cec\mid\ceb)\P(\ceb)\ge \tfrac34\P(\ceb)\ge\tfrac34c(b)+o(1).
\end{equation}

On the event $\cebc$, we have
by \eqref{cec}, \eqref{lb3A}, \eqref{ceb}, \eqref{drakeA} and \eqref{t1},
for any $t\in[0,\KK t_1]$,
\begin{equation}\label{brx}
  \begin{split}
\tA_n(t_1+t)&
\ge \E\bigpar{\tA_n(t_1+t)\mid\cF_{t_1}}-C\gs_n
\ge\tA_n(t_1)-tZ_n+O(\gs_n)
\\&
> b\gs_n+\E\tA_n(t_1)-\CCZ t\gss_n+O(\gs_n)
\\&
=b\gs_n+O(\gs_n).
  \end{split}
\end{equation}
The implicit constants here depends on $\KK $ but not on $b$; thus the final
error term $O(\gs_n)\ge -\CC(\KK )\gs_n$ for some $\CCx(\KK )$.
Hence we may for
any $\KK $ choose $b=b(\KK ):=\CCx(\KK )$, and the result follows, with
$p(\KK )=\frac34c(b(\KK ))$.
\end{proof}

We can obtain results for $\tV_n$ similar to the results for $\tS_n$ above
(in Lemmas \ref{LB1}, \ref{LB2a}, \ref{LB3}, and \ref{LB4})
by the same arguments. However, we have no need for such results involving
conditioning and uniform estimates; the following simple results  are enough.

\begin{lemma}
  \label{LB6}
Assume \refAA{} and \eqref{gDRo}.
Fix $\KK >0$. For any $t\in[0,\KK t_1]$,
\begin{equation}\label{lb6}
\tV_n(t)=n-n\mu_n t+\Op\bigpar{nt_1^2+\sqrt{nt_1}}
  =n-n\mu_n t+\op\xpar{nt_1}.
\end{equation}
\end{lemma}
\begin{proof}
  Recall that $\tV_n(t)=\sum_k\tV_{n,k}(t)$ where $\tV_{n,k}(t)$ are
  independent and $\tV_{n,k}(t)\sim\Bin\bigpar{n_k,\ee{-kt}}$.
Hence,
\begin{equation}\label{lb6q}
  \E \tV_n(t) =\sumko n_k\ee{-kt}
=\sumko n_k\bigpar{1-kt+O(k^2t^2)}
=n-n\mu_n t + O(nt^2)
\end{equation}
and
\begin{equation}\label{lb6w}
  \Var \tV_n(t) =\sumko n_k\ee{-kt}\bigpar{1-\ee{-kt}}
\le\sumko n_kkt
=n\mu_n t =O(nt).
\end{equation}
The first equality in \eqref{lb6} follows from \eqref{lb6q}--\eqref{lb6w}.
The second follows because $nt_1^2=o(nt_1)$ and $\sqrt{nt_1}=o(nt_1)$.
\end{proof}

\begin{lemma}
  \label{LB7}
Assume \refAA{} and \eqref{gDRo}, and define $\VVn(t):=\tV_n(t)-V_n(t)\ge0$.
Fix $B>1$. Then
\begin{equation}
\VVn(t_1)-\VV_n(Bt_1)
\le \Op(t_1\gs_n)
=\op(nt_1).
\end{equation}
\end{lemma}

\begin{proof}
  $\VV\nk(t):=\tV_{n,k}(t)-V_{n,k}(t)$ is the number of vertices of degree
  $k$ that are
  awake at time $t$, but their $k$ half-edges all have maximal life times
  larger than $t$. This number may increase when $\sC1$ is performed, and it
  decreases when a half-edge at one of these vertices dies spontaneously
  (and $\sC3$ is performed). Consequently, conditioning of $\cF_{t_1}$,
for any $t\ge0$,
  \begin{equation*}
	\begin{split}
\E \bigpar{(\VV\nk(t_1)-\VV\nk(t_1+t))_+\mid \cF_{t_1}}
\le kt \VV_{n,k}(t_1).
	\end{split}
  \end{equation*}
Summing over $k$ yields, using \eqref{m22},
  \begin{equation}\label{EVV}
\E \bigpar{(\VVn(t_1)-\VVn(t_1+t))_+\mid \cF_{t_1}}
\le t \bigpar{  \tS_n(t_1)-S_n(t_1)}.
  \end{equation}

By \eqref{jlss} and \refL{LB8}, noting that $\gD_n=O(\gs_n)$ by \eqref{gDR}
and \eqref{gssR},
\begin{equation}
  \tS_n(t_1)-S_n(t_1) < \sup_{t\le t_1} |\tA_n(t)|+\gD_n=\Op(\gs_n).
\end{equation}
In other words, for every $\eps>0$ there exist $K(\eps)$ independent of $n$
such that
\begin{equation}\label{gronkulla}
\P\bigpar{\tS_n(t_1)-S_n(t_1))>K(\eps)\gs_n}\le \eps.
\end{equation}
Furthermore, for any fixed $K$, \eqref{EVV} implies
  \begin{equation}\label{EVV2}
\E \bigpar{(\VVn(t_1)-\VVn(t_1+t))_+\mid \tS_n(t_1)-S_n(t_1)\le K\gs_n}
=O( t \gs_n).
  \end{equation}
It follows by \eqref{EVV2}, Markov's inequality and \eqref{gronkulla} that,
for any $t>0$,
\begin{equation}
(\VVn(t_1)-\VVn(t_1+t))_+ = \Op(t\gs_n).
\end{equation}
Now take $t=(B-1)t_1$.
\end{proof}

\begin{proof}[Proof of \refT{TC}\ref{TC>}]
Note that the assumptions include \eqref{epscrit} and \eqref{gDRo}.
Recall also that $A_n(t)\ge\tA_n(t)$ for all $t$, see \eqref{jlss}.
Hence by \refL{LB5},
for every $\KK >1$, there exists $p(\KK )>0$ such that
with probability at least $p(\KK )+o(1)$,
$A_n(t)\ge \tA_n(t)>0$ for all $t\in[t_1,\KK t_1]$. By the
discussion in \refSS{SSgeneral}, this means that $\sC1$ is not performed
during the interval $[t_1,\KK t_1]$ and thus all vertices awakened during this
interval belong to the same component, say $\cC$. The number of these
vertices is $V_n(t_1)-V_n(\KK t_1)$.
Consequently,
with probability at least $p(\KK )+o(1)$,
\begin{equation}\label{req}
  \vx{\cC_1}\ge\vx{\cC}\ge
V_n(t_1)-V_n(\KK t_1).
\end{equation}
Furthermore,
by Lemmas \ref{LB7} and  \ref{LB6},
\begin{equation}\label{kyr}
  \begin{split}
  V_n(t_1)-V_n(\KK t_1)
&
=
 \tV_n(t_1)-\tV_n(\KK t_1)+\VVn(\KK t_1)-\VVn(t_1)
\ge
  \tV_n(t_1)-\tV_n(\KK t_1)+\op(nt_1)
\\&
=n\mu_n(\KK -1)t_1+\op(nt_1)
=n\mu(\KK -1)t_1+\op(nt_1).	
  \end{split}
\raisetag\baselineskip
\end{equation}
Hence,
$V_n(t_1)-V_n(\KK t_1)> \bigpar{\mu(\KK -1)-1}nt_1$ \whp

Finally, given any $K>0$, choose $\KK $ such that $\mu(\KK -1)=K+1$.
Then \eqref{req} and \eqref{kyr} thus show that,
with probability at least $p(\KK )+o(1)$,
recalling \eqref{t1},
\begin{equation}\label{chr}
  \vx{\cC_1}\ge V_n(t_1)-V_n(\KK t_1)> Knt_1 = Kn^{2/3}R_n\qqqw,
\end{equation}
which completes the proof of \eqref{tc>}.
\end{proof}

\subsection{Proof of \refT{TC}\ref{TC>} in the graph case}
\label{SSTC>graph}

Unlike the other results in this paper, \refT{TC}\ref{TC>} says that a
certain event asymptotically has a positive but possibly small probability.
In order to obtain the same result for the simple random graph $G_n$ from the
result for $G_n^*$, we have to
show that this event has a large intersection with the event
$\es:=\set{G_n^*\text{ is simple}}$.

Recall that \ref{A2} implies $\P(\es)\ge \cs+o(1)$ for some $\cs>0$.
In fact, \eqref{gDqqo} and  \ref{Anu=1} imply,
see \eg{} \cite[Corollary 1.4]{Jans06b} or \cite[Theorem 1.1]{AngHofHol16},
	\begin{equation}
	\label{pes}
  	\P(\es)=\e^{-\nu_n/2-\nu_n^2/4}+o(1)=\e^{-3/4}+o(1),
	\end{equation}
so we take $\cs:=\e^{-3/4}$.

We claim the following:
\begin{lemma}
  \label{Ls}
Assume \refAA{} and \eqref{gDRo}.
Then the asymptotic normality \eqref{asnS} and \eqref{asnL} hold also
conditioned
on $\es$.
(The expectations in \eqref{asnS} and \eqref{asnL} are still for the
configuration model, without conditioning.)
\end{lemma}

We postpone the proof of the lemma.

\begin{proof}[Proof of \refT{TC}\ref{TC>} in the graph case]
Note that, given \refL{Ls}, we obtain
also \eqref{Ssmall} conditioned on $\es$
by the argument in the proof of \refL{LB1}. That is,
for any $b>0$, there exists $c(b)>0$ such that
	\begin{equation}
	\label{bra}
  	\P \bigpar{\tA_n(t_1)-\E\tA_n(t_1)>b\gs_n \mid \es }\ge c(b)+o(1).
	\end{equation}

Consider now \refL{LB5}.
It follows, similarly to the first part of the proof of \refL{LB5}, that
$\P(\ceb\mid\es)\ge c(b)+o(1)$.
Hence, $\P(\ceb\cap\es)
\ge c(b)\cs+o(1)$.
Since $\es\notin\cF_{t_1}$, we modify the next part of the proof of \refL{LB5}.
By \eqref{pec},
	\begin{equation}\label{brb}
  	\begin{split}
    	\P\bigpar{\cec\cap\es\mid\ceb}
	&\ge \P\bigpar{\cec\mid\ceb}
	+ \P\bigpar{\es\mid\ceb}-1
	\\&
	\ge1-\frac{\CCpec}{C^2}
	+ \P\bigpar{\es\cap\ceb}-1
	\\&
	\ge \cs c(b)-\frac{\CCpec}{C^2}+o(1)
	\ge \tfrac12\cs c(b)+o(1),	
 	 \end{split}
	\end{equation}
for a suitable choice of $C$.
The rest of the proof of \refL{LB5} works as before.
We  obtain, using \eqref{brb},
	\begin{equation}\label{brc}
 	\P \bigpar{\cebc \cap \es}=\P\bigpar{\ceb\cap\cec\cap\es}\ge\tfrac12\cs c(b)^2+o(1).
	\end{equation}
Hence we conclude, using \eqref{brx} as before,
that, for any $B > 1$,
	\begin{equation}
 	\P \bigpar{ \{\tilde{A}_n(t) > 0\colon t \in [t_1,Bt_1]\} \cap \es} \ge p(B) + o(1)
	\end{equation}
for some (new) $p(B)>0$, where $t_1$ is as in~\eqref{t1}.
Finally, the proof of \refT{TC}\ref{TC>} above yields, cf.~\eqref{chr},
$\P ( \{\vx{\cC_1}\ge Kn^{2/3}R_n\qqqw \}\cap \es) \ge p(B) + o(1)$,
and thus
$\P ( \vx{\cC_1}\ge Kn^{2/3}R_n\qqqw \mid \es) \ge p(B) + o(1)$,
which completes the proof of \refT{TC}\ref{TC>} for the simple random graph
$G_n$.
\end{proof}

It remains only to prove the lemma.
This could be done by the method used for similar
results in \cite{SJ282} and \cite{SJ300}, see also \cite{SJ196},
but we prefer an alternative, simpler, argument.

\begin{proof}[Proof of \refL{Ls}]
Consider the conditional analogue of \eqref{asnS}; the proof of conditional \eqref{asnL} is identical.

Let $a\in\bbR$ and let
$\ea:=\set{\bigpar{\tS_n(t_1)-\E\tS_n(t_1)}/\gs_n\le a}$;
thus, by \eqref{asnS},
	\begin{equation}
	\label{pea}
	\P(\ea)\to\Phi(a).
	\end{equation}
Let $T'$ denote the first time that a connected component is completely
explored after time $t_1$.
Let $B>1$. If $T'>Bt_1$, then the component $\cC$ explored until $T'$ has at
least
$V_n(t_1)-V_n(T'-)\ge V_n(t_1)-V_n(Bt_1)$ vertices,
and hence, using
Lemmas \ref{LB6} and \ref{LB7},
	\begin{equation}
  	v(\cC_1)\ge v(\cC)
	\ge V_n(t_1)-V_n(Bt_1)
	=n\mu_n(B-1)t_1+\op(nt_1).
	\end{equation}
It follows from
\refT{TC}\ref{TC<} that, for any $\gd>0$ and any fixed
$B$ such that $\mu(B-1)>K(\gd)$, we have $\P(T'> Bt_1) <\gd+o(1)$.
Consequently, if $B_n\rightarrow \infty$, then
$\P(T'> B_nt_1) <2\gd$ for any $\gd>0$ and all large $n$, and thus
$T'\le B_nt_1$ \whpx.
Note that   \eqref{t1} and \eqref{Rnqqo} imply that
	\eqn{
	\label{small-number}
	t_1 R_n= n\qqqw R_n^{2/3}=o(1),
	}
and that,
since $1=O(R_n)$, we also have
$t_1=o(1)$. We may  thus fix a sequence $B_n\to\infty$
such that $B_nt_1=o(1)$ and $B_nt_1R_n=o(1)$.

Let $T''$ be the first time that
the number of sleeping half-edges $S_n(t)$
drops
below $\ell_n/2$. (Recall that $S_n(0)=\ell_n=n\mu_n$.)
At time $B_nt_1$, the expected number of times that $\sC3$ has been performed is
at most $B_nt_1\ell_n=o(\ell_n)$, and corresponding to a few of these times also $\sC1$ was
performed; it follows easily that the expected number of sleeping half-edges
at $B_nt_1$ is $\ell_n-o(\ell_n)$, and thus \whp{} $T''>B_nt_1$.

Let $\BT$ denote the event that all the components explored by time $T'$ are
simple.

The probability that vertex $i$ is awakened no later than time $(B_nt_1)\land T''$ by using
\sC1 or \sC3 is $O(B_nt_1d_i)$, and, in the event that it is awakened,
the probability that two of its half-edges will form a loop is
$O(d_i^2/\ell_n)$ and the probability that it will be joined by a multiple
edge to a vertex $j$ awakened later is $O(d_i^2d_j^2/\ell_n^2)$.
Consequently,
	\begin{multline}
  	\label{comp-degree-rem}
	\prob\bigpar{\BT^c\cap \{T'\leq B_nt_1\}\cap\set{T''>B_nt_1}}\\
	\leq O(B_nt_1)\sum_{i\in[n]} {d_i}
	\Big[\frac{d_i^2}{\ell_n}+\frac{d_i^2}{\ell_n}\sum_{j\in [n]}
  	\frac{d_j^2}{\ell_n}\Big]=O(t_1 B_n R_n)
	=o(1),
	\end{multline}
and thus $\prob(\BT^c)=o(1)$, i.e., $\BT$ holds \whpx.

Then, we condition on the $\sigma$-algebra $\cF_{T'}$ of all randomness up
to time $T'$, and note that $\ea$ and $\BT$ are $\cF_{T'}-$measurable to obtain
	\eqn{\label{sat}
	\P(\ea\cap \es)
	=\P(\ea\cap \es\cap\BT)+o(1)
	=\E[\indicwo{\ea\cap \BT} \P(\es \mid \cF_{T'})]+o(1).
	}
The configuration model multigraph can be partitioned into the connected components found until time $T'$ and those that are found afterwards.
The multigraph consisting of all the connected components found after time $T'$ is
again (conditioned on $\cF_{T'}$)
a configuration model, now with a random number
$\tilde{n}=n(1-o(1))$ vertices and
degrees that are a (random) subset of  size $\tilde{n}$ from $[n]$. We denote this degree sequence by $(\tilde{d}_i)_{i\in[\tilde{n}]}$. In particular,
conditional on $\cF_{T'}$,
	\begin{equation}
  	\label{pese}
	\P(\es \mid \cF_{T'})=\indicwo{\BT}\prob(G(\tilde{n}, (\tilde{d}_i)_{i\in[\tilde{n}]})\text{ simple}).
	\end{equation}
By the discussion above~\eqref{comp-degree-rem},
the probability that
the event $\{T'\le B_nt_1\le T''\}$ occurs
and that
vertex $i$ is part of one of the connected components
found before time $T'$
is $O(B_nt_1 d_i)$.
Hence,
	\eqn{
	\E\Bigl[ \Bigpar{\sum_{i\in[n]}	d_i^2
	-\sum_{i\in[\tilde{n}]}\tilde{d}_i^2}
	\indic{ T'\le B_nt_1 < T''}
	\Bigr]
	\leq O(B_nt_1)\sum_{i\in[n]} {d_i^3}
	=O(nt_1B_nR_n)=o(n).
	}
Consequently, using Markov's inequality and recalling that $T'\le B_nt_1\le
T''$ \whp, we obtain
	\begin{equation}
	\sum_{i\in[\tilde{n}]}\tilde{d}_i^2
	=
  	\sum_{i\in[n]} d_i^2-
	\op(n)
	=\bigpar{1+\op(1)}  \sum_{i\in[n]} d_i^2.
	\end{equation}
Similarly,
or as a consequence,
$\sum_{i\in[\tilde{n}]}\tilde{d}_i
=\bigpar{1+\op(1)}  \sum_{i\in[n]} d_i$.

Thus, with $\tilde{\nu}_n$ denoting  $\nu_n$ in \eqref{nun-def} for the
(random) degree sequence $(\tilde{d}_i)_{i\in[\tilde{n}]}$,
and noting that $\nu_n = \sum_i d_i^2/\sum_i d_i -1$ and $\tilde{\nu}_n = \sum_i \tilde{d}_i^2/\sum_i \tilde{d}_i -1$
we obtain
	\begin{equation}
  	\tnu_n=\nu_n+\op(1)=1+\op(1).
	\end{equation}
Consequently,
	\eqref{pes} yields
	\begin{equation}
	\prob(G(\tilde{n}, (\tilde{d}_i)_{i\in[\tilde{n}]})\text{ simple})
	=\e^{-\tnu_n/2-\tnu_n^2/4}+\op(1)=\e^{-3/4}+\op(1)=\P(\es)+\op(1),
	\end{equation}
and, since $\BT$ holds \whp, \eqref{pese} yields
	\begin{equation}
  	\label{final}
	\P(\es \mid \cF_{T'})=\P(\es)+\op(1).
	\end{equation}

Finally, \eqref{sat} and \eqref{final} yield, together with \eqref{pea},
\begin{equation}
	\P(\ea\cap \es)
	=\E[\indicwo{\ea\cap \BT} \P(\es)]+o(1)
	=\P(\ea\cap \BT) \P(\es)+o(1)
	=\Phi(a) \P(\es)+o(1),
\end{equation}
and thus $	\P(\ea\mid\es)\to\Phi(a)$,
which completes the
proof of the lemma, and thus of the theorem.
\end{proof}
\medskip

\section{The complexity}\label{Scomplex}

Define the \emph{complexity} of a component $\cC$ by $k(\cC):=e(\cC)-v(\cC)+1$;
this is the number of independent cycles in $\cC$.
The estimates in \refT{T1} show only that $k(\cC_1)=\op(v(\cC_1))$.
(This is in contrast to the strongly supercritical case $\E D(D-2)>0$,
when $v(\cC_1)=c_vn\ettop$ and $e(\cC_1)=c_en\ettop$ for two
positive constants $c_v$ and $c_e$, see e.g.~\JL{Theorem 2.3}, and it is
easily verified that $c_e>c_v$ so $k(\cC_1)$ also is linear in $n$.)
We can use our methods to obtain a much sharper result.

\begin{theorem}
  \label{TK}
Suppose that \ref{AA}--\ref{AO} are
satisfied,  in particular  $\vep_n=o(1)$.
Suppose also that $\vep_n\gg n^{-1/3} (\expec\Dn^3)^{2/3}$.
Then
\begin{equation}\label{tk}
  k(\cC_1)
=n\chi_n\bigpar{1+\op(1)},
\end{equation}
where
\begin{align}
\chi_n
&:=
\frac12\mu_n\bigpar{1-(1-\rho_n)^2}-\E\bigpar{1-(1-\rho_n)^{D_n}}
  \label{chirho}
\\&\phantom:
=\frac12\mu_n\bigpar{1-\e^{-2\gan}}-\E\bigpar{1-\e^{-\ga_nD_n}}
  \label{chi1}
\\&\phantom:
=\E h(\ga_n D_n)-\tfrac12 \E D_n h(2\ga_n),	
\label{chih}
\end{align}
with
	\begin{equation}\label{h}
  	h(x):=\Bigpar{1+\frac{x}2}\e^{-x}-1+\frac{x}2
	=\frac12 \sum_{n\ge3} (-1)^{n-1}\frac{n-2}{ n!}x^n.
	\end{equation}
Moreover, $n\chi_n\to\infty$, $\chi_n=O(\ga_n^2\eps_n)=O(\eps_n^3)$ and
	\begin{equation}
	\label{chi2}
  	\chi_n\asymp \gan\gam_n\asymp\E\bigpar{(\gan D_n)\land (\gan D_n)^3}.
	\end{equation}
\end{theorem}

\begin{remark}
The expression \eqref{chirho} is what would be intuitively expected from the
branching process approximation:
if we multiply by $n$, then the first term is the number of edges
($\ell_n/2=n\mu_n/2$) times the approximate probability that one of the
endpoints of an edge attaches to the largest component, and the second term
is the approximate probability that a random vertex attaches to the largest
component. Indeed, it follows from Theorem \ref{T1} that the two terms
approximate
$e(\cC_1)/n$ and $v(\cC_1)/n$ within a factor $1+\op(1)$.
However,  the  two terms in \eqref{chirho} differ only
by a factor $1+o(1)$, so there is a significant cancellation and
we need a different argument to show the result.
\end{remark}

\begin{remark}\label{Roo}
  By \eqref{h} and simple calculus,
$h(0)=h'(0)=0$ and $h''(x)=\frac12 x\e^{-x}$, so $h(x)$ is positive and
  convex on $(0,\infty)$. Moreover, $h(x)\sim\frac{1}{12}x^3$ as $x\to0$
and $h(x)\le \frac{1}{12}x^3$ for $x\ge0$.
Although the expressions in \eqref{chirho}--\eqref{chi1} are simpler, there
is (as said in \refR{Roo}) a lot of
cancellation, and \eqref{chih} better highlights the order of $\chi_n$.
\end{remark}

We postpone the proof of \refT{TK} and state first some consequences for the
most  important cases.

\begin{theorem}
\label{TK3}
Suppose that \ref{AA}--\ref{AO} are satisfied,
and that $D_n^3$ is uniformly integrable.
Suppose further that $\vep_n n^{1/3} \to \infty$.
Then
	\begin{equation}\label{tk3}
  	k(\cC_1)=\frac{\kk\mu}{12}n\rho_n^3 \bigpar{1+\op(1)}
	=\frac{2\mu}{3\kk^2}n\eps_n^3 \bigpar{1+\op(1)},
	\end{equation}
where $\kk\in(0,\infty)$ is given by \eqref{kk}.
\end{theorem}

This extends the result for the \ER{} random graph $G(n,p)$.
There, in the barely supercritical case $k(\cC_1)\sim\frac{2}3n\eps_n^3$
(see, with more details, \cite{PittelWormald2005}
and, for $\eps\le n^{1/12}$, \cite{birth}),
which
corresponds to the case $D\sim\Po(1)$
(when $\mu=\kk=1$),
of \refT{TK3}
by conditioning on the vertex degrees as in \refS{sec-discussion}.
The order of the complexity in \eqref{tk3} interpolates nicely between the
known cases
of $\vep_n=\vep>0$ independently of $n$, where $k(\cC_1)$ is of order $n$, and the critical case $\vep_n=O(n^{-1/3})$,
where $k(\cC_1)$ converges in distribution \cite{DhaHofLeeSen16a}.

\begin{theorem}
\label{TK3infty}
Suppose that \ref{AA}--\ref{AO}
are satisfied,
and that $\E D^3=\infty$. (Thus $\E D_n^3\to\infty$.)
Suppose further that $\vep_n\gg n^{-1/3} (\E D_n^3)^{2/3}$.
Then
	\begin{align}
   	k\xpar{\cC_1}&=\op(n \eps_n^3 )  \label{tk3infty}
	.\end{align}
\end{theorem}

\begin{example}[Power-law degrees]
Consider again the power-law example in \refE{exam-power-law},
with $2<\gam<3$.
It follows from \eqref{chi2}, \eqref{rgamn} and \eqref{charlietau} that
$\chi_n\asymp\gan^\gam\asymp\eps_n^{\gam/(\gam-2)}$.
Again, this interpolates nicely between the known cases
of $\vep_n=\vep>0$ independently of $n$, where $k(\cC_1)$ is of order $n$,
and the critical case $\vep_n=O(n^{-(\gamma-2)/\gamma})$,
where $k(\cC_1)$ converges in distribution. The latter is shown in
\cite{DhaHofLeeSen16b} under stronger power-law assumptions on the degrees,
including that
$d_i n^{-1/\gamma}\to c_i$ with $\sum_{i\geq 1} c_i^3<\infty$, while
$\sum_{i\geq 1} c_i^2=\infty$,
such as for $c_i\asymp i^{-1/\gamma}$ with $\gamma\in(2,3)$.
(Recall Remark \ref{rem-crit}, where this is discussed in more detail.)
\end{example}

\begin{example}
Suppose that \ref{AA}--\ref{AO}
are satisfied,
$\E D^3=\infty$, and, furthermore,
 $\rho_n\gD_n=O(1)$.
Then \refL{Lredlund} applies and yields together with \eqref{chi2}
	\begin{equation}
 	 \chi_n\asymp\gan\gam_n\asymp \eps_n^3/(\E D_n^3)^{2},
	\end{equation}
showing that \eqref{tk3infty} in this case can be sharpened to
$k(\cC_1)\asymp n\eps_n^3/(\E D_n^3)^2$ \whpx.
\end{example}

\begin{lemma}
  \label{Lgagam}
Suppose that \ref{AA}--\ref{AO} are
satisfied and that $\vep_n\gg n^{-1/3} (\expec\Dn^3)^{2/3}$.
Then $n\gan\gam_n\to\infty$.
\end{lemma}

\begin{proof}
  We consider only $n$ such that $\eps_n>0$; this holds at least for all
  large $n$.

First, if $\ga_n\gD_n\le1$, then \refL{Lredlund} and the assumptions yield
\begin{equation}
  n\gan\gam_n\asymp n\frac{\eps_n^3}{(\E D_n^3)^2}\to\infty.
\end{equation}

On the other hand, if $\ga_n\gD_n>1$, then by \ref{Xdmax}, which was verified
in the proof of \refT{T1},
$1<\ga_n\gD_n=o(\gan n\gam_n)$, and thus $n\gan\gam_n\to\infty$ in this case
too.
\end{proof}

\begin{proof}[Proof of \refT{TK}]
Let $N(t)$ be the number of times up to
time $t$ that a new cycle is created.
Thus,
if $T$ is a time when $\sC1$ is performed, then $N(T)$ is the sum of the
complexities of the components explored up to $T$.

 During the exploration process, we create a new cycle each time \sC3 is
performed and the half-edge that dies is an active half-edge, i.e, each time
an  active half-edge dies spontaneously.
This happens with rate $A_n(t)$.
Consequently,
	\begin{equation}\label{MN}
	M(t):=N(t)-\int_0^t A_n(u)\dd u
	\end{equation}
is a martingale, with $M(0)=0$.

Let $T_1$ and $T_2$ be as in the proof of \refT{TX}, so \whp{} $\cC_1$ is
explored between $T_1$ and $T_2$. Thus \whp{} $k(\cC_1)=N(T_2)-N(T_1)$.
Recall that $T_1/\gan\pto0$ and $T_2/\gan\pto\tau=1$, and note that $T_2$ is
a stopping time.

Recall that \ref{Xfirst}--\ref{Xlast} were verified in the proof of
\refT{T1}.
By \ref{XA} and \refL{L0},
  	\begin{equation}\label{isaac}
	\sup_{t\le T_2}\lrabs{\frac1{n\gam_n}A_n(\gant)-\psi_n(t)}\pto0.	
  	\end{equation}
Consequently, using also that $\psi_n(t)$ is uniformly bounded on $[0,2]$ by
\refR{RB}, and that $T_2/\gan\pto1$ so that $T_2/\gan\le2$ \whp,
	\begin{equation}
  	\label{cam}
  	\begin{split}
	\int_0^{T_2}A_n(u)\dd u
	&
	=\gan\int_0^{T_2/\gan} A_n(\gan u)\dd u
	=n\gam_n\gan \int_0^{T_2/\gan} \psi_n(u)\dd u+\op\bigpar{n\gam_n\gan}	
	\\&
	=n\gan\gam_n \int_0^{1} \psi_n(u)\dd u+\op\bigpar{n\gan\gam_n}	.
 	 \end{split}
	\end{equation}
Let
	\begin{equation}\label{Psi}
  	\Psi_n:=\int_0^1\psi_n(t)\dd t,
	\end{equation}
and note that by \refR{RB} and \ref{Xpsiab}, $\Psi_n\asymp1$.
Define also the stopping time $T$ by
	\begin{equation}
  	\label{Tpsi}
	\int_0^T A_n(u)\dd u = n\gan\gam_n\bigpar{\Psi_n+1}.
	\end{equation}
By \eqref{cam}, $T_2\le T$ \whpx.

All jumps in the martingale $M(t)$ are $+1$, so the quadratic variation
(see \eg{} \cite[Theorem 26.6]{Kallenberg}) is
	\begin{equation}
  	[M,M]_t=\sum_{u\le t}\bigpar{\gD M(u)}^2=\sum_{u\le t}\gD M(u)
	=N(t).
	\end{equation}
Hence, for the stopped martingale $M(t\land T)$,
using \eqref{MN} and the definition \eqref{Tpsi}
of $T$, as well as
\cite[Corollary 3 to Theorem II.6.27, p.~73]{Protter},
	\begin{equation*}
  	\begin{split}
	\E\bigpar{M(T_2\land T)^2}
	&
	=\E[M,M]_{T_2\land T}
	=\E N(T_2\land T)
	=\E \int_0^{T_2\land T} A_n(u)\dd u
	+\E M(T_2\land T)
	\\&
	\le n\gan\gam_n\bigpar{\Psi_n+1}+0
	=O\bigpar{n\gan\gam_n}.
  	\end{split}
	\end{equation*}
Hence it follows that, using also \refL{Lgagam},
	\begin{equation}
  	\label{ini}
	M(T_2\land T)=\Op\bigpar{(n\gan\gam_n)\qq}
	=\op\bigpar{n\gan\gam_n}.
	\end{equation}

By \eqref{MN}, \eqref{cam}, \eqref{ini} and $T_2\land T=T_2$ \whp,
	\begin{equation}\label{NT2}
 	N(T_2)=\int_0^{T_2}A_n(u)\dd u+M(T_2)
	=n\gan\gam_n\Psi_n+\op\bigpar{n\gan\gam_n}.
	\end{equation}
Furthermore, for any fixed $\gd>0$, $T_1<\gd\gan$ \whp{} and thus
$N(T_1\land T)\le N(T\land(\gd\gan))$.
Hence, again since $M$ is a martingale,
	\begin{equation}
	\label{cam1}
  	\E N\bigpar{T_1\land T} \le \E N\bigpar{T\land(\gd\gan)}
	=\E\int_0^{T\land(\gd\gan)}A_n(u)\dd u.
	\end{equation}
Furthermore, by \eqref{isaac} and \refR{RB},
	\begin{equation}\label{cam2}
  	\begin{split}
	  \int_0^{T\land(\gd\gan)} A_n(u)\dd u
	&
	\le \int_0^{\gd\gan} A_n(u)\dd u
	=\gan\int_0^{\gd} A_n(\gan t)\dd t
	\\&
	=n\gan\gam_n\Bigpar{\int_0^{\gd}\psi_n(t)\dd t+\op(1)}
	\le n\gan\gam_n\bigpar{\gd+\op(1)}.
 	 \end{split}
	\end{equation}
It follows from \eqref{cam1} and \eqref{cam2}, by dominated convergence
justified by \eqref{Tpsi}, that
	\begin{equation}
  	\bigpar{n\gan\gam_n}\qw\E N(T_1\land T) \le \gd+o(1).
	\end{equation}
Since $\gd\in(0,1)$ is arbitrary, it follows that
$\E N(T_1\land T)=o\xpar{n\gan\gam_n}$, and thus \whp{}
$N(T_1)=N(T_1\land T)=\op(n\gan\gam_n)$.
Consequently, recalling \eqref{NT2}, \whp
	\begin{equation}
  	k(\cC_1)=N(T_2)-N(T_1)
	=n\gan\gam_n\bigpar{\Psi_n+\op(1)}
	=n\gan\gam_n\Psi_n\bigpar{1+\op(1)}
	,
	\end{equation}
which shows \eqref{tk} with
	\begin{equation}\label{chin}
	\chi_n=\gan\gam_n\Psi_n.
	\end{equation}

Recalling $\Psi_n\asymp1$,
we have $\chi_n\asymp\gan\gam_n$ and thus $n\chi_n\to\infty$ by \refL{Lgagam}.
Furthermore, \eqref{chi2} follows from \eqref{svagam}.
It follows from
\eqref{chi2} and \eqref{sard} that
	\begin{equation}
	\label{chiaae}
  	\chi_n\asymp \E\bigpar{(\gan D_n)\land (\gan D_n)^3}
	\le \E\bigpar{(\gan D_n)^2\land (\gan D_n)^3}
	\asymp \gan^2\eps_n,
	\end{equation}
i.e., $\chi_n=O(\gan^2\eps_n)$;
furthermore $\ga_n\sim\rho_n=O(\eps_n)$ by \eqref{rhoO}.

It remains to evaluate $\chi_n$ in \eqref{chin} and show that it agrees with
\eqref{chirho}--\eqref{chih}.
By \eqref{Psi},  \eqref{svapsi} and Fubini's theorem,
\begin{equation}
  \begin{split}
\chi_n=
 \gan\gam_n\Psi_n
&
=\gan\int_0^1\bigpar{\mu_n\e^{-2\gan t}-\E\bigpar{D_n \e^{-\gan tD_n}}}\dd t
\\&
=\frac12\mu_n\bigpar{1-\e^{-2\gan}}-\E\bigpar{1-\e^{-\ga_nD_n}},
  \end{split}
\end{equation}
which shows \eqref{chi1}.
By the definition \eqref{gan} of $\gan$, this is the same as \eqref{chirho}.
Furthermore, the equality of \eqref{chih} and \eqref{chi1} follows by a
simple calculation using \eqref{frej}.
\end{proof}

\begin{proof}[Proof of \refT{TK3}]
Under the assumptions in  \refT{TK3},
$\gam_n\sim \gan^2\E D^3$ by \eqref{gamD3} and
  	\begin{equation}
	\Psi_n=\int_0^1\psi_n(t)\dd t\to\frac{\kk\mu}{12\E D^3}	
  	\end{equation}
as a consequence
  of \eqref{psiD3}. Hence \eqref{tk3} follows from \eqref{tk}, \eqref{chin}
  and \eqref{ganD3}.
\end{proof}

\begin{proof}[Proof of \refT{TK3infty}]
 As in the proof of \refT{TD3infty}, \eqref{e:tsurv2} yields
$\ga_n=o(\eps_n)$. Hence, \eqref{chiaae} implies $\chi_n=o(\eps_n^3)$,
and \eqref{tk3infty} follows.
\end{proof}

\section*{Acknowledgements}
This work was commenced while the authors were visiting the Mittag-Leffler Institute in 2009 for the programme `Discrete Probability'.
Part of the work
was done during the authors' visit to 
the International Centre for Mathematical Sciences in Edinburgh
to attend a workshop `Networks: stochastic models for populations and epidemics' in 2011.
The paper was completed during the authors' visit
to the
Isaac Newton Institute for Mathematical Sciences
for the programme
``Theoretical Foundations for Statistical Network Analysis'',
supported by EPSCR grant EP/K032208/1.
RvdH is supported by the Netherlands
Organisation for Scientific Research (NWO) through VICI grant 639.033.806 and the Gravitation {\sc Networks} grant 024.002.003.
SJ is supported by a grant from
the Knut and Alice Wallenberg Foundation
and a grant from
the Simons foundation. ML is supported by an EPSRC Leadership Fellowship, grant reference EP/J004022/2.

  \newcommand\cprime{$'$}

\end{document}